\documentclass[12pt]{article}

\usepackage[top=170pt,bottom=110pt,left=140pt,right=140pt]{geometry}

\usepackage{hyperref}
\usepackage{amsmath,amsthm,amssymb,amsfonts,enumerate,xcolor,layout,tikz,enumitem,fullpage, verbatim, graphicx, caption, subcaption}

\usetikzlibrary{decorations.pathreplacing,patterns}

\usepackage{tikz-cd}

\newtheorem{theorem}{Theorem}[section]
\newtheorem{proposition}[theorem]{Proposition}
\newtheorem{corollary}[theorem]{Corollary}
\newtheorem{lemma}[theorem]{Lemma}
\newtheorem{claim}[theorem]{Claim}

\theoremstyle{definition}
\newtheorem{definition}[theorem]{Definition}
\newtheorem{remark}[theorem]{Remark}
\newtheorem{example}[theorem]{Example}

\theoremstyle{plain}

\newtheorem{question}[theorem]{Question}

\definecolor{linkblue}{rgb}{0,0,.6}
\definecolor{citered}{rgb}{.7,0,0}
\hypersetup{colorlinks =true, linkcolor=linkblue, citecolor = citered, urlcolor=linkblue}

\definecolor{darkgreen}{rgb}{0,.5,0}

\def\C{{\mathbb C}}
\def\R{{\mathbb R}}
\def\Z{{\mathbb Z}}

\def\T{{\mathbb T}}

\newcommand{\De}{\Delta}
\newcommand{\om}{\omega}
\newcommand{\fg}{\mathfrak{g}}

\newcommand{\proj}{\text{proj}}
\newcommand{\aff}{\text{Aff}_{\Z}(\R^n)}
\newcommand{\dd}{\mathrm{d}}
\renewcommand{\Tilde}{\widetilde}

\newcommand{\optional}[1]{} 

\begin{document}

\title{Integrable systems with symmetries: toric, semitoric, and beyond}
\author{Joseph Palmer}
\date{\today}

\maketitle

\begin{abstract}
This article presents an overview of the theory of integrable systems with symmetries, focusing on toric systems, semitoric systems, and their classifications via decorated polygons. 
We discuss certain one-parameter families of integrable systems called semitoric families, and explain how deforming systems through controlled bifurcations in such families (and their generalizations) can be used to construct explicit semitoric systems with prescribed invariants.
The first part of the paper serves as a quick introduction to integrable systems for newcomers to the field, such as graduate students, while the majority of the exposition surveys recent developments and technical details that will be of interest to experts.
It closes with a look at future directions, including hypersemitoric systems and complexity one integrable systems.
\end{abstract}

\setcounter{tocdepth}{2}
\tableofcontents

\section{Introduction}

Consider the example of the spherical pendulum: a mass swinging on a rod. Can we predict its motion over time?
Though it has now expanded far beyond its original use,
symplectic geometry was first developed to model the phase space and dynamics of classical dynamical systems, such as this spherical pendulum.
A quantity is said to be conserved by the system if it remains constant as the system evolves over time.
For instance, in the pendulum described above the total energy (the kinetic energy of the motion plus the gravitational potential energy) is conserved by the system.
Roughly speaking, such a dynamical system is called \emph{integrable} if it admits the maximal number of independent conserved quantities.

Many familiar dynamical systems are integrable, such as the two-body problem, the harmonic oscillator, various spinning tops, and the spherical pendulum described above.
In fact, these specific examples share another important property: a rotational symmetry. This symmetry can be represented as an action of the circle $S^1$ on the phase space of the system which preserves the symplectic structure.

There is already a long history of studying group actions which preserve the symplectic form, called \emph{symplectic group actions}, and a special class of such actions which are generated by a function, called \emph{Hamiltonian group actions}. 
Each conserved quantity in an integrable system, viewed as a real valued function, is associated to an infinitesimal symmetry of the system by taking the flow of the Hamiltonian vector field of the function.
This paper is focused on integrable systems in which the symmetries associated to some subset of the conserved quantities generate a Hamiltonian action of the torus $\T^k:=(S^1)^k$.
We start from the beginning, 
assuming no prior knowledge of symplectic geometry or integrable systems. 
That being said, experts can skip the beginning of the paper and jump right into contemporary results, including a detailed exposition on the classification of those four-dimensional systems called \emph{semitoric}, a discussion of certain one-parameter families of systems of this type, called \emph{semitoric families}, and a glimpse of various other recent results around these topics.

To get more precise, on a symplectic manifold $(M,\om)$ of dimension $2n$, an integrable system is the data of $n$ smooth real-valued functions which Poisson commute and whose differentials are almost-everywhere independent. Associated to each function is a vector field called its Hamiltonian vector field, and if following the flows of these Hamiltonian vector fields generates a $\T^n$-action, then the system is called toric and it is extremely well understood (due to the work of Atiyah, Guillemin-Sternberg, and Delzant). 

A large portion of this paper is focused on semitoric integrable systems, which generalize toric integrable systems in dimension four. A semitoric integrable system has two conserved quantities, $J$ and $H$, and we assume, among other things, that the symmetry associated to $J$ is periodic (i.e.~it is an $S^1$-action).
Semitoric systems were classified by Pelayo and V\~{u} Ng\d{o}c~\cite{PeVN2009,PeVN2011} for a generic class of systems, and this classification was later extended to all semitoric systems in~\cite{PPT-nonsimple}.
The original classification and its generalization make use of slightly different invariants, but in both cases the invariant is a polygon decorated with certain marked points and cuts, where each cut is also labeled with extra data (typically a Taylor series and integer).

We discuss the invariants in detail, and we explain how to construct them from a given semitoric system. After stating the classification results, we move on to explaining a technique that can be used to actually produce completely explicit examples of such systems.
The idea is to use a family of integrable systems, depending on a single parameter, to start with a well-understood system (typically toric) and deform it through certain controlled bifurcations (the Hamiltonian-Hopf bifurcation) and obtain the desired semitoric system. This procedure is developed and laid out in~\cite{LFPal,LFPal-SF2}, and applied to produce certain desired systems in various other papers, including~\cite{HoPa2017,HohMeu,GuHo}.

These techniques have been extremely fruitful for semitoric systems. The classification is an important result, which has led to numerous other papers, and the construction techniques have been used to produce a wide variety of semitoric systems (and even hypersemitoric systems). 
It is time to expand the theory. 
Semitoric systems have restrictions on dimension (they are four dimensional) and restrictions on the types of singularities which can arise (only elliptic type or focus-focus). This paper gives the motivation and background needed to understand the theory of semitoric systems, and gives an overview of much of the progress and results around semitoric systems.
We also discuss ongoing efforts to generalize these results to higher dimensional settings and other broader classes of integrable systems. The hope is that the success of semitoric systems provides a road map describing how to generalize this theory to include a wider variety of types of singular points and to include systems of higher dimension.

\paragraph{Acknowledgments:} 
I'd like to thank my collaborators,  
from whom I learned many things about integrable systems and symplectic geometry, and without which this paper would not have been possible.
Also, I'd like to thank Yael Karshon and Yohann Le Floch, who gave many helpful comments on an early version of the paper.
Finally, as always I am indebted to CK, for her endless patience and support.


\subsection{Outline}

This paper aims to be useful to graduate students, mathematicians from other fields, and experts in this field. Therefore, depending on your background it might make sense to skip around. In particular, Sections~\ref{sec:symplectic-int-ham} and~\ref{sec:toric} act as a mostly self-contained introduction to integrable systems and torus actions. 
Here is what to expect in each section:

\begin{itemize}
    \item \textbf{Section~\ref{sec:symplectic-int-ham}:} We give a brief introduction to symplectic geometry, integrable systems, and Hamiltonian group actions. This is important background for the rest of the paper, but of course can be skipped depending on the background of the reader. We assume essentially no prerequisites except for basic differential geometry, starting from the beginning.
    \item \textbf{Section~\ref{sec:toric}:} We discuss symplectic toric manifolds, as an important example of integrable systems, and discuss their classification in terms of Delzant polytopes. 
    \item \textbf{Section~\ref{sec:local-class}:} We discuss local classifications of integrable systems: we present the Liouville-Arnold-Mineur theorem for regular fibers (Theorem~\ref{thm:LAM}) and discuss non-degenerate singular points and their classification.
    \item \textbf{Section~\ref{sec:st}:} We discuss semitoric integrable systems, which are a useful class of $4$-dimensional integrable systems and are classified by objects called marked semitoric polygons with certain labels on each marked point. The original classification for a generic class of semitoric systems is due to Pelayo and V\~{u} Ng\d{o}c~\cite{PeVN2009,PeVN2011} and was extended to all semitoric systems in~\cite{PPT-nonsimple}.
    In this section we make a special effort to separately achieve two goals:
    \begin{itemize}
        \item In Section~\ref{sec:st-invariants} we define the invariants of semitoric systems as abstract objects (marked polygons with labels) without any reference to a semitoric system,
        \item In Section~\ref{sec:st-construct} we explain how to construct the invariants from a given semitoric system.
    \end{itemize}
    From here we state the classification results and end the section with a useful example.
    \item \textbf{Section~\ref{sec:st-families}:} We discuss \emph{semitoric families} and their generalizations. These are families of integrable systems with a fixed underlying symplectic manifold and Hamiltonian $S^1$-action, but a second integral which depends on a parameter. As the parameter changes the systems can undergo certain bifurcations and these systems are useful for obtaining examples with certain desired properties.
    \item \textbf{Section~\ref{sec:closing}:} The paper closes with a short section giving a quick introduction to many related and interesting topics which, unfortunately, we were not able to discuss in detail in this paper. We mention several topics, such as
    \begin{itemize}[noitemsep]
        \item[\textbf{\ref{sec:close-big-S1-section}:}] The relationship between integrable systems and Hamiltonian $S^1$-actions (especially concerning classifications);
        \item[\textbf{\ref{sec:close-ATFs-symp-top}:}] Almost toric fibrations and symplectic topology;
        \item[\textbf{\ref{sec:close-quantum}:}] Quantization;
        \item[\textbf{\ref{sec:beyond}:}] Generalizing semitoric systems by: \begin{enumerate}
             \item[\textbf{\ref{sec:hst}:}] Reducing the restrictions on the types of singularities (e.g.~hypersemitoric systems);
             \item[\textbf{\ref{sec:close-cmlx-one}:}] Allowing higher dimensions (complexity one systems).
        \end{enumerate} 
    \end{itemize}
    While our discussion of each topic in this section is brief, we include many references to more detailed accounts.
    We end the paper with a brief discussion of future directions in the field.
\end{itemize}

\noindent Now, it's time to get started.


\section{Symplectic geometry, integrable systems, and Hamiltonian group actions}
\label{sec:symplectic-int-ham}

In this section, we will introduce some important background and motivation for the rest of the paper. 
We move very quickly, only discussing results that will be useful to us, but this brief treatment should suffice as an introduction to the field for someone unfamiliar with it.
This section can be skipped, or skimmed, by experts who are interested in jumping right into the more recent developments, or other readers familiar with this standard background.

\subsection{Symplectic geometry}\label{subsec:symplectic}

Symplectic geometry was originally developed to study classical dynamical systems, but has since expanded into a vast field with deep connections to many areas of math and physics, such as geometric mechanics, algebraic geometry, Floer theory, homological mirror symmetry, and low-dimensional topology. We start with a very quick review of the relevant foundational ideas and notions in symplectic geometry, and cover as much as we can in a few pages, but of course many basic topics are skipped or covered with very little detail. Fortunately, there are many resources to fill these gaps, such as the excellent introductory lecture notes by da Silva~\cite{daSi2008} and the, also excellent, textbook by McDuff-Salamon~\cite{McDSal}.

 Now we give a quick overview of some of the foundational definitions, concepts, and examples in symplectic geometry.
A \emph{symplectic manifold} is a pair $(M,\om)$ such that $M$ is a smooth manifold and $\om$ is a closed, non-degenerate 2-form on $M$.
That is, $d\omega=0$ and for all $p\in M$ if $v\in T_pM$ satisfies $\omega_p(v,w)=0$ for all $w\in T_pM$ then $v=0$. In this case $\omega$ is called a \emph{symplectic form}.
By a version of the Gram-Schmidt process, the existence of such an $\om$ implies that the dimension of $M$ must be even (for a proof, see for instance~\cite{daSi2008}).
Furthermore, the non-degeneracy implies that the wedge of $\frac{1}{2}\dim(M)$ copies of $\om$ is a volume form, so $M$ must be orientable.

The non-degeneracy of $\om$ is equivalent to the fact that for all $p\in M$ the map \begin{align*}\om^\flat_p\colon  T_pM &\to T^*_pM\\  v &\mapsto \om_p(v,\cdot)\end{align*} is a vector space isomorphism (from the definition of non-degenerate it is immediate to show that $\om_p^\flat$ is injective, and this result then follows from the fact that $\dim(T_pM) = \dim(T^*_pM)$).

In the case that $\mathrm{dim}(M)=2$, the notion of a symplectic form coincides with that of a volume form.
An example of a symplectic manifold is $\R^{2n}$ with coordinates $x_1,\ldots, x_n, y_1,\ldots,y_n$ and symplectic form \[\om_0 = \sum_i \mathrm{d}x_i\wedge \mathrm{d}y_i,\]
called the \emph{standard symplectic structure on $\R^{2n}$}.

If $(M,\om)$ and $(M',\om')$ are symplectic manifolds then a diffeomorphism $\phi\colon M\to M'$ is called a \emph{symplectomorphism} if $\phi^* \om' = \om$.
Darboux's theorem states that if $(M,\om)$ is a symplectic manifold of dimension $2n$, then for any $p\in M$ there exists an open neighborhood of $p$ which is symplectomorphic to $\R^{2n}$ with the standard symplectic form $\om_0$. That is, all symplectic manifolds of a given dimension look the same locally: up to symplectomorphism, the only local invariant is dimension.

A submanifold $S$ of $M$ with inclusion map $i\colon S\hookrightarrow M$ is called \emph{isotropic} if $i^*\om = 0$.
Using linear algebra, we can obtain a bound on the dimension of such submanifolds:

\begin{claim}\label{claim:isotroipic-dim}
    If $S\subseteq M$ is isotropic then $\dim(S)\leq \frac{1}{2}\dim(M)$.
\end{claim}

\begin{proof}
    Let $p\in S$, let $V = T_pM$, and let $W=T_pS$. Let \[W^\om = \{v\in V \mid \om_p(v,w)=0 \text{ for all }w\in W\}.\]
    Note that the isotropic condition implies that $W\subseteq W^\om$, so $\dim(W)\leq \dim(W^\om)$. Consider the map $V\to W^*$ given by $v \mapsto \om (v,\cdot)|_W$. Then the kernel of this map is $W^\om$ and its image is $W^*$ (using the fact that $\om$ is non-degenerate), so $\dim(V) = \dim(W^\om) + \dim(W) \geq 2 \dim (W)$.
\end{proof}

A submanifold $L$ of $M$ is called a \emph{Lagrangian submanifold} if it is isotropic and $\mathrm{dim}(L) = \frac{1}{2}\mathrm{dim}(M)$. Lagrangian submanifolds are surprisingly important in symplectic geometry, as we will see in the following sections.

We close the section with one more example. Let $X$ be any manifold and let $\pi \colon T^*X \to X$ be projection. Then the \emph{tautological one form} on $T^*X$ is the one form $\alpha$ defined by $\alpha_{(p,\xi)}(\eta)=\xi(\pi_*\eta)$ for each $p\in X$, $\xi\in T^*_p(X)$, and for $\eta\in T_{(p,\xi)}(T^*_p X)$. Then the two-form $\om = \dd \alpha$ is a symplectic form. If $T^*X$ has local coordinates $(x_i,\xi_i)$, then $\om = \sum_i \dd x_i \wedge \dd \xi_i$.
Thus, $T^*X$ is naturally a symplectic manifold, and moreover the zero section (which is canonically symplectomorphic to $X$) is a Lagrangian submanifold of $T^*X$. 
This example is worked out in full detail in, for instance, Section 2 of~\cite{daSi2008}.

\subsubsection{Hamiltonian vector fields}

Let $f\colon M\to\R$ be any smooth function. 
Then $\mathrm{d}f$ is a one-form on $M$, and since $\om$ is a two-form we can obtain a one-form by plugging a single vector field into $\om$. Due to the fact that $\om$ is non-degenerate, there always exists a unique vector field to plug into $\om$ which produces $\mathrm{d}f$. That is, any function $f$ determines a vector field $\mathcal{X}_f$, called the \emph{Hamiltonian vector field of $f$}, via the equation
\begin{equation}\label{eqn:Xf}
 -\mathrm{d} f = \om(\mathcal{X}_f, \cdot).
\end{equation}
Not all authors include the negative sign above, but it is traditional in physics so we include it here.
Note that Equation~\eqref{eqn:Xf} is analogous to the equation defining the gradient vector field in Riemannian geometry, except for the negative sign and replacing the metric with the symplectic form. For this reason, the Hamiltonian vector field of $f$ is sometimes called the \emph{symplectic gradient of $f$}.

 Note that it is straightforward to check that $f$ is preserved by the flow of $\mathcal{X}_f$:
\[
 \mathcal{L}_{\mathcal{X}_f}f = \mathcal{X}_f f = \mathrm{d}f(\mathcal{X}_f) = -\om(\mathcal{X}_f,\mathcal{X}_f) = 0
\]
and it is similarly straightforward to check, using Cartan's magic formula and the fact that $\om$ is closed, that $\om$ is preserved by the flow of $\mathcal{X}_f$:
\[
 \mathcal{L}_{\mathcal{X}_f}\om = \mathrm{d} (\om(\mathcal{X}_f,\cdot)) + (\mathrm{d} \om)(\mathcal{X}_f,\cdot,\cdot) = \mathrm{d}(-\mathrm{d}f) + 0 = 0.
\]
The flow of $\mathcal{X}_f$, when it exists, thus preserves both $\om$ and $f$, so the dynamics coming from this flow lie on level sets of $f$. Equation~\eqref{eqn:Xf} is a simple version of one of the key ideas in the study of integrable systems: it connects dynamics and group actions (the flow of $\mathcal{X}_f$) with level sets of a function (the level sets of $f$).

\subsection{Integrable systems}
\label{sec:integrable_systems}

Now we continue on towards the definition of integrable systems, and some first consequences.

\subsubsection{Motivation}
We start with a general (and somewhat vague) example, skipping the details and all complications, just to explain the concept behind integrable systems.
Suppose that $(M,\om)$ is a symplectic manifold which represents the phase space of a classical dynamical system, so that each point in $M$ represents a state of the system (positions and momenta), and suppose that the dynamics are generated by the Hamiltonian vector field $\mathcal{X}_H$ of a function $H:M\to \R$.
That is, the physical system's evolution through different states over time, which corresponds to a path in the phase space $M$, is given by the flow of the vector field $\mathcal{X}_H$.
The typical (but not only) situation is that $M=T^*Q$, where $Q$ is the manifold of possible positions of the system (the \emph{configuration space}) and the cotangent vectors of $Q$ are viewed as describing the momenta of the system (in Section~\ref{subsec:symplectic} we described how cotangent bundles inherit a natural symplectic form).
Any function $f\colon M\to\R$ which is preserved by the dynamics of this system, i.e.~$\mathcal{X}_H(f)=0$, is called a \emph{conserved quantity} of the system. In examples, these are often quantities such as components of the momentum or angular momentum of the system.
Note that $\mathcal{X}_H(H)=\om(\mathcal{X}_H,\mathcal{X}_H)=0$, so the function $H$ is a conserved quantity as well, usually identified with the total energy of a state of the system.

Each conserved quantity $f$ also has a Hamiltonian vector field $\mathcal{X}_f$, and 
\[ \mathcal{X}_f(H) = \mathrm{d}H (\mathcal{X}_f) = -\om(\mathcal{X}_H,\mathcal{X}_f) = \om(\mathcal{X}_f,\mathcal{X}_H) = -\mathrm{d}f (\mathcal{X}_H) = -\mathcal{X}_H(f)=0, \]
using the definition of Hamiltonian vector fields and the fact that $f$ is conserved. 
Thus, the flow of $\mathcal{X}_f$ (for however long it exists) preserves both $\om$ and $H$: it is a symmetry of the system. 
Now suppose that $(M,\om,H)$ is a Hamiltonian dynamical system, and suppose that $f$ is a conserved quantity. 
If $g$ is another conserved quantity, it is natural to desire that $g$ is not only preserved by the dynamics (the flow of $\mathcal{X}_H$), but also that it is invariant under the symmetry induced by $f$ (the flow of $\mathcal{X}_f$). 
That is, in order for $g$ to be invariant under the dynamics and symmetries of the system, we want both $\mathcal{X}_H(g)=0$ and $\mathcal{X}_f(g)=0$.

The main idea of integrable systems is to have many such conserved quantities which are all independent of one another.  Furthermore, note that from this point of view the Hamiltonian or energy function $H$ and the conserved quantity $f$ play symmetric roles, so we will not usually designate which function is the Hamiltonian. 
Each additional conserved quantity makes the system easier to understand, since the dynamics is restricted to the joint level sets of all conserved quantities and the entire system is invariant under the flow of the Hamiltonian vector field of each function.

\subsubsection{The definition of an integrable system}

To streamline the following discussion, let us introduce the natural Poisson bracket
 induced by the symplectic form: given two functions $f,g\in C^\infty (M)$, the \emph{Poisson bracket} of $f$ and $g$ is the smooth function on $M$ given by:
\[
 \{f,g\} = \om(\mathcal{X}_f,\mathcal{X}_g).
\]
Note that
\[
 \mathcal{X}_f(g) = -\mathrm{d}g (\mathcal{X}_f) = -\om (\mathcal{X}_g,\mathcal{X}_f) = \{f,g\}
\]
so the condition that $g$ is invariant under the flow of $\mathcal{X}_f$ is equivalent to $\{f,g\}=0$. If $\{f,g\}=0$ then we say that $f$ and $g$ \emph{Poisson commute}.
We say that a collection of functions $f_1,\ldots, f_k\in C^\infty(M)$ are \emph{almost everywhere independent} if $\mathrm{d}_pf_1,\ldots, \mathrm{d}_p f_k\in T^*_pM$ are linearly independent for almost all $p\in M$.

\begin{lemma}\label{lem:max-num-integrals}
    Let $(M,\om)$ be a symplectic manifold of dimension $2n$, and suppose that \[f_1,\ldots,f_k\colon M\to\R\] are smooth functions which pairwise Poisson commute and are linearly independent almost everywhere.
    Then $k\leq n$.
\end{lemma}

\begin{proof}
    Let $p\in M$ be any point for which $\mathrm{d}_pf_1,\ldots, \mathrm{d}_pf_k$ are linearly independent. Note this is equivalent to the fact that $\mathcal{X}_{f_1},\ldots, \mathcal{X}_{f_k}$ are linearly independent at $p$, so they span a $k$-dimensional subspace $V$ of $T_pM$. Furthermore, $\om(\mathcal{X}_{f_i},\mathcal{X}_{f_j}) = \{f_i,f_j\} = 0$ for all $i,j$, so $\om_p$ vanishes on $V$. Since $V$ is thus an isotropic subspace, the result follows by Claim~\ref{claim:isotroipic-dim}.
\end{proof}

With Lemma~\ref{lem:max-num-integrals} in mind, the largest number of such functions possible is $\frac{1}{2}\dim(M)$, and so we give this situation a special name: an integrable system.

\begin{definition}
    A \emph{(completely) integrable system} is a triple $(M,\om,F)$ where $(M,\om)$ is a symplectic manifold of dimension $2n$ and \[F = (f_1,\ldots,f_n)\colon M\to \R^n\] is a smooth function, called the \emph{momentum map}, whose components $f_1,\ldots, f_n\in C^\infty(M)$ Poisson commute and are linearly independent almost everywhere. The points at which the linear independence fails are called the \emph{singular points} of the system.
\end{definition}

We have motivated integrable systems from the point of view of mechanics, but we will see that they appear in many areas of mathematics.

\subsubsection{First consequences}
\label{sec:first-conseq}

In this section, several basic facts about symplectic geometry are established, all with an eye towards integrable systems.

We say that $X$ is a \emph{symplectic vector field} if $\mathcal{L}_X\om = 0$. Note that, since $\om$ is closed,
\[
 \mathcal{L}_X \om = \iota_X \mathrm{d}\om + d(\iota_X \om) = d(\iota_X \om),
\]
where $\iota_X \om = \om (X,\cdot)$ as usual.
 By definition, $X$ is a \emph{Hamiltonian vector field} if $\iota_X\om = -\mathrm{d}f$ for some function $f$. 
 Thus, $X$ is a symplectic vector field if and only if $\iota_X\om$ is closed and $X$ is a Hamiltonian vector field if and only if $\iota_X\om$ is exact.

 \begin{claim}
     Let $X$ and $Y$ be symplectic vector fields. Then $\mathcal{X}_{\om(X,Y)} = [X,Y]$.
 \end{claim}

\begin{proof}
    Recall that $\mathrm{d}(\iota_X\om) = \mathrm{d}(\iota_Y\om) =0$ and note that $\iota_{[X,Y]}(\cdot) = \mathcal{L}_X(\iota_Y(\cdot)) - \iota_Y(\mathcal{L}_X(\cdot))$. 
    Thus,
    \begin{align*}
        \om([X,Y],\cdot) &= \iota_{[X,Y]}\om\\
        &=\mathcal{L}_X(\iota_Y \om)-\iota_Y(\mathcal{L}_X\om)\\
        &=\mathrm{d}(\iota_X \iota_Y \om)+\iota_X(d(\iota_Y\om)) - \iota_Y(\iota_X (\mathrm{d}\om)+\mathrm{d}(\iota_X\om))\\
        &=\mathrm{d}(\iota_X \iota_Y \om)+\iota_X(0) - \iota_Y(\iota_X (0)+0)\\
        &=-\mathrm{d}(\om(X,Y)). \qedhere
    \end{align*}
\end{proof}

Applying the previous claim to Hamiltonian vector fields, we see that for $f,g\in C^\infty(M)$, since $\{f,g\} = \om(\mathcal{X}_f,\mathcal{X}_g)$, we have
\[
 \mathcal{X}_{\{f,g\}} = [\mathcal{X}_f,\mathcal{X}_g].
\]
In particular, if $\{f,g\}=0$, then $[\mathcal{X}_f,\mathcal{X}_g]=0$, which leads to the following corollary.
\begin{corollary}
    If $(M,\om,F)$ is an integrable system and $F=(f_1,\ldots,f_n)$, then $[\mathcal{X}_{f_i},\mathcal{X}_{f_j}]=0$ for all $i,j$.
\end{corollary}
Thus, the flows of each of the $\mathcal{X}_{f_i}$ in an integrable system, whenever they exist, commute.
Recall that a map is called \emph{proper} if the preimage of each compact set is compact. If $(M,\om,F)$ is an integrable system and $F$ is proper, then each fiber of $F$ is compact. 
In particular, if $F$ is proper then each of the vector fields $\mathcal{X}_{f_i}$ are complete, and since the flows commute they generate an action of $\R^n$ on $M$. We will call this the \emph{induced $\R^n$-action on $M$}. 
More precisely, suppose that $F$ is proper and let $\psi^t_i\colon M\to M$ be the time $t$ flow of the vector field $\mathcal{X}_{f_i}$. Then the $\R^n$-action on $M$ is given by
\begin{align}\label{eqn:intsys-groupaction}
    \R^n\times M &\to M\\ \nonumber
    ((t_1,\ldots,t_n),p) &\mapsto \psi_1^{t_1}\circ \ldots \circ \psi_n^{t_n}(p)
\end{align}
This induced action is the key connection between integrable systems and Hamiltonian group actions, which we discuss in Section~\ref{sec:ham-int-interactions}.

As discussed earlier, $\mathcal{X}_f(g) = \{f,g\}$, so the Poisson commuting condition in an integrable system implies that each $f_i$ is invariant under the flow for as long as it exists of each $\mathcal{X}_{f_j}$. In other words, the induced $\R^n$-action stays on level sets of $F$.

\subsubsection{Examples of integrable systems}

Here we give several important examples.

\begin{example}[Harmonic Oscillator]
    Let $M = \C^n$ with the symplectic form $\om = \frac{i}{2}\sum \mathrm{d}z_i\wedge \mathrm{d}\overline{z}_i$ and let $F\colon M\to\R^n$ be given by 
    \[
     F(z_1,\ldots,z_n) = \frac{1}{2}\left(|z_1|^2,\ldots,|z_n|^2\right).
    \]
    Since $\mathrm{d}f_k = \frac{1}{2}(\overline{z}_k\dd z_k + z_k \dd \overline{z}_k)$ we obtain that $\mathcal{X}_{f_k} = i\left(z_k \frac{\partial}{\partial z_k} - \overline{z}_k \frac{\partial}{\partial \overline{z}_k}\right)$. To have a clearer view of the flow of this vector field, we can switch to polar coordinates via $z_k = r_k \mathrm{e}^{i \theta_k}$, in which case $\mathcal{X}_{f_k} = \frac{\partial}{\partial \theta_k}$. Therefore, the $\R^n$-action induced by this integrable system is the map $\R^n\times \C^n \to \C^n$ given by
    \[
     ((t_1,\ldots,t_n),(r_1\mathrm{e}^{i\theta_1},\ldots, r_n\mathrm{e}^{i\theta_n})) \mapsto \left(r_1\mathrm{e}^{i(\theta_1 + t_1)}, \ldots, r_n\mathrm{e}^{i(\theta_n + t_n)}\right).
    \]
    Note that this action is periodic, and therefore descends to an action of the compact group $\R^n/(2\pi \Z)$.
    This system can equivalently be viewed as an integrable system on $\R^{2n}$ with coordinates $(x_1,\ldots,x_n,y_1,\ldots,y_n)$, symplectic form $\sum \dd x_i\wedge \dd y_i$, and momentum map
    \[
     F(x_1,\ldots,x_n,y_1,\ldots,y_n) = \frac{1}{2}(x_1^2+y_1^2,\ldots,x_n^2+y_n^2).
    \]
\end{example}

\begin{example}\label{ex:cot-bundle-torus}
 Consider the cotangent bundle $M=T^*\T^n$ of the $n$-dimensional torus $\T^n := (S^1)^n$. Using the identification from $M$ to $\T^n\times\R^n$, there are coordinate functions \[q_1,\ldots, q_n\colon M\to S^1 \text{ and } p_1,\ldots, p_n\colon M\to\R.\] The standard symplectic form on $M$ is $\om_{T^*\T^n}=\sum_i \mathrm{d} p_i \wedge \mathrm{d} q_i$, and 
 with the projection map $\pi_{\R^n}\colon M\to\R^n$ given by
\[
 \pi_{\R^n}  = (p_1,\ldots,p_n)\colon M \to \R^n
\]
we obtain an integrable system $(T^*\T^n,\om_{T^*\T^n}, \pi_{\R^n})$. 
 The Hamiltonian flows of the functions $p_1,\ldots,p_n$ generate a $\T^n$-action on $M$ (that is, it is a toric integrable system, see Example~\ref{example:toric-int-syst}).
\end{example}

\begin{example}
An important physical example is the simple pendulum (or 2d pendulum), which models a mass swinging on a rigid rod which can only go back and forth (as opposed to Example~\ref{ex:spherical-pen}, in which the pendulum can swing in all directions). A configuration of this system is given by the position of the mass, encoded by the angle of the rod $\theta\in S^1$, so the phase space is $T^*S^1 \cong S^1\times \R$. Taking coordinates $(\theta,p)\in S^1\times \R$, the standard symplectic form is $\dd \theta \wedge \dd p$, the Hamiltonian of the system is the sum of the kinetic and potential energies 
\[H(\theta,p) = \frac{1}{2}p^2 + \cos(\theta),\] 
and the dynamics of the system is given by the flow of the Hamiltonian vector field of $H$. Since $T^*S^1$ is only $2$-dimensional, only one function is needed to form an integrable system, so $(T^*S^1,\dd \theta \wedge \dd p,H)$ is an integrable system.
\end{example}

\begin{example}[Spherical pendulum]\label{ex:spherical-pen}
We will now describe the example of the spherical pendulum, which models the motion of a mass swinging on a rod. The configuration space of this system is $S^2$ (corresponding to the position of the mass).
Let $(\phi,\theta)$ be spherical coordinates on $S^2$. That is, if $S^2\subset \R^3$ is viewed as the unit sphere in $\R^3$ with coordinates $x,y,z$ then 
\[
 x = \sin(\phi)\cos(\theta), \quad y = \sin(\phi)\sin(\theta), \quad z = \cos(\phi).
\]
Let $p_\phi$ and $p_\theta$ be the coordinates along the fibers of $T^*S^2$ induced by $\phi$ and $\theta$. Define $H\colon T^*S^2\to\R$ by
\[
 H(\phi,\theta,p_\phi, p_\theta) = \frac{1}{2}\left(p_\phi^2 + \frac{p_\theta^2}{\sin^2(\phi)}\right) + \cos(\phi).
\]
Note that it looks like there is a singularity when $\sin(\phi)=0$, but this is due to the choice of coordinates and $H$ is in fact smooth.
Define $J\colon T^*S^2\to \R$ by
\[
 J(\phi,\theta,p_\phi,p_\theta) = p_\theta.
\]
Physically, $H$ represents the total energy of a given state (the first term is kinetic and the second is potential) and $J$ represents the angular momentum. Let $\om = \dd \phi \wedge \dd p_\phi + \dd\theta \wedge \dd p_\theta$. Then $(T^*S^2,\om,F)$ is an integrable system, where $F=(J,H)$.

There are two points where $dF=0$, both of which occur in the zero section of $T^*S^2$. There is a stable singularity at the point corresponding to the south pole of the sphere (representing when the pendulum is hanging down at rest) and an unstable singularity at the point corresponding to the north pole of the sphere (representing the position of the pendulum balancing straight upwards). Using terminology that we have not defined yet (see Theorem~\ref{thm:eliasson}), the point at the south pole is an elliptic-elliptic point and the point at the north pole is a focus-focus point.
\end{example}

There are many physical examples of integrable systems. For more examples, see for instance~\cite{arnold89,PeVNsymplthy2011} and the references therein. We close this section with an example that will be important for us in Section~\ref{sec:local-class}.

\begin{example}\label{ex:non-deg-examples}
Here we will describe three examples of singular points that we will later see are used to build a large class of singularities. The idea is that products of the following examples will serve as local models for certain singular points (see Section~\ref{subsec:singularities}).

\begin{itemize}
    \item \textbf{Elliptic model:} Let $M=\R^2$ with coordinates $(x,y)$, let $\om = \dd x \wedge \dd y$, and define $F\colon M\to \R$ by 
    \[
     F(x,y) = x^2 + y^2.
    \]
    Then $(M,\om, F)$ is an integrable system. The flow of $\mathcal{X}_F$ generates an $S^1$-action on $\R^2$ which rotates the plane and has a single fixed point, at the origin.
    \item \textbf{Hyperbolic model:} Let $M = \R^2$ with coordinates $(x,y)$, let $\om = \dd x \wedge \dd y$, and define $F\colon M \to \R$ by
    \[
     F(x,y) = xy.
    \]
    Then $(M,\om,F)$ is an integrable system. The flow of $\mathcal{X}_F$ generates an $\R$-action on $\R^2$ which has exactly one fixed point, at the origin. The fibers of $F$ are not connected, and for any $c\in \R\setminus \{0\}$ the fiber $F^{-1}(c)$ is diffeomorphic to the disjoint union of two copies of $\R$.
    \item \textbf{Focus-focus model:} Let $M = \R^4$ with coordinates $(x_1,y_1,x_2,y_2)$, let $\om = \sum_i \dd x_i \wedge \dd y_i$, and define $F = (f_1,f_2)$ where
    \begin{align*}
        f_1(x_1,y_1,x_2,y_2) &= x_1y_2-x_2y_1,\\
        f_2(x_1,y_1,x_2,y_2) &= x_1y_1+x_2y_2.
    \end{align*}
    Then $(M,\om,F)$ is an integrable system. The flow $\mathcal{X}_{f_1}$ generates an $S^1$-action while $\mathcal{X}_{f_2}$ only generates an $\R$-action. The $(S^1\times\R)$-action has a single fixed point, at the origin. The fiber $F^{-1}(0,0)$ is two planes meeting transversely at a single point, and the $S^1$-action rotates each of these planes.
    \item \textbf{Products:} Taking products of any of the above models, along with the regular model $(x,y)\mapsto y$ can produce a variety of singular points, which are typically described by listing the products. For instance, the elliptic-focus-focus-regular model (sometimes called the EFFR-model) would be the integrable system with $M = \R^8$ with coordinates $(x_1,\ldots,x_4,y_1,\ldots,y_4)$, symplectic form $\om = \sum_i \dd x_i \wedge \dd y_i$, and momentum map $F = (f_1,\ldots,f_4)$ with
    \begin{align*}
        f_1 &= x_1^2+y_1^2,\\
        f_2 &= x_2y_3-x_3y_2,\\
        f_3 &= x_2y_2+x_3y_3,\\
        f_4 &= y_4.
    \end{align*}
    The flows of the Hamiltonian vector fields in this case produce a $(\T^2\times \R^2)$-action.
\end{itemize}
\end{example}

\subsection{Hamiltonian group actions}

In this section we quickly introduce the concepts from the theory of Hamiltonian group actions which will be useful for our study of their interactions with integrable systems.
In Section~\ref{sec:integrable_systems} we started with a function and used it to induce a group action (via the flow of the Hamiltonian vector field), while in this section we will view the group action as the foundational object.

\subsubsection{Motivation}

Let $(M,\om)$ be a symplectic manifold, let $\phi\colon \R\times M\to M$ be an $\R$-action, 
and let $X$ be the vector field on $M$ generated by the $\R$-action in the following sense:
\[
    X_p = \frac{\dd}{\dd t}\Big|_{t=0} \phi(t,p).
\]
Then we say that this action is \emph{symplectic} if $X$ is a symplectic vector field, and we say that this action is \emph{Hamiltonian} if $X$ is a Hamiltonian vector field (see Section~\ref{sec:first-conseq}). 


This notion is compatible with our earlier concept of the Hamiltonian vector field of a function in the following sense:
if $X$ is the vector field generated by the $\R$-action $\phi$, then $\phi$ is Hamiltonian if and only if there exists a smooth function $f\colon M\to \R$ for which $X = \mathcal{X}_f$, where $\mathcal{X}_f$ is the Hamiltonian vector field of $f$ as in Equation~\eqref{eqn:Xf}.

The above situation of a Hamiltonian $G$-action with $G=\R$ can be seen as a motivation for the general theory of Hamiltonian group actions. Speaking very roughly, the idea is that for each vector in the Lie algebra, we get a one-parameter subgroup of $G$, and we ask that the vector field generated by the action of this subgroup be Hamiltonian as above. So for each $v\in \mathfrak{g}$, we want a function $f_v\in C^\infty(M)$. 
Thus, for each $p\in M$ we obtain a map from $\mathfrak{g}\to\R$ taking $v$ to $f_v(p)$, and if we require this assignment to be linear, then we have a map from $M$ to $\mathfrak{g}^*$. 
While this discussion is informal, it hopefully motivates why a map $M\to\mathfrak{g}^*$ is natural in this context. We will now make this discussion precise with a formula which is analogous to Equation~\eqref{eqn:Xf}.

Let $G$ be a group, and assume that $G$ acts smoothly on a symplectic manifold $(M,\om)$. This action is called \emph{Hamiltonian} if there exists a map $\mu \colon M \to \fg^*$, called the \emph{moment map}, with the following properties:
\begin{enumerate}
    \item[(1)] For all $v\in  \mathfrak{g}$, \begin{equation}\label{eqn:general-ham-action}
    \omega(v^{\#},\cdot )=-\dd \langle \mu, v\rangle,
    \end{equation}
    where $v^{\#}$ denotes the vector field generated by the action of the subgroup $\{\mathrm{exp}(tv)\mid t\in\R\}\subseteq G$ and $\langle \cdot,\cdot\rangle$ is the usual duality pairing\footnote{that is, $\langle \mu,v\rangle$ denotes the real-valued function on $M$ given by $p\mapsto  \mu_p(v_p)$.},
    \item[(2)] $\mu$ is equivariant with respect to the $G$ action on $M$ and the adjoint action of $G$ on $\mathfrak{g}^*$.
\end{enumerate}

Let us immediately note that the second condition is not necessary if $G$ is abelian, which is the case that we will be interested in for the remainder of this paper. In the case that $G$ is abelian, $\mu$ being a moment map is equivalent to $\mu$ satisfying (1) above and also satisfying:
\begin{enumerate}
    \item[(2')] $\mu$ is invariant under the $G$-action. That is, the value of $\mu$ is preserved by the action of $G$.
\end{enumerate}

If $G$ acts on $(M,\om)$ with momentum map $\mu$, then we call $(M,\om,G,\mu)$ a \emph{Hamiltonian $G$-space}.

Note that if $G=\R$ and we identify the dual Lie algebra of $\R$ with $\R$ itself, then an $\R$-action being Hamiltonian is the same as the $\R$-action being the flow of $\mathcal{X}_f$ for some function $f\colon M\to\R$, since $f$ plays the role of $\mu$.

\optional{should I put more about Hamiltonian group actions here? Some example?}

\subsubsection{Interactions between integrable systems and Hamiltonian group actions}
\label{sec:ham-int-interactions}

Let $(M,\om,F)$ be an integrable system, and suppose that $F$ is proper. Then we obtain an $\R^n$-action as in Equation~\eqref{eqn:intsys-groupaction}, by taking the flow of the commuting vector fields $\mathcal{X}_{f_1},\ldots, \mathcal{X}_{f_n}$. Recall that these flows commute since the functions Poisson-commute, and that the properness of $F$ implies that these vector fields are complete. 
In fact, this $\R^n$-action is Hamiltonian, taking as the momentum map $\mu = \phi \circ F$ where $\phi\colon \R^n \to \mathrm{Lie}(\R^n)$ is the usual identification of $\R^n$ with its Lie algebra.
Indeed, let $e_1,\ldots,e_n\in\R^n$ be the standard basis, then $\mathcal{X}_{f_i} = \phi(e_i)^{\#}$, so taking $v=\phi(e_i)$, Equation~\eqref{eqn:general-ham-action} becomes $\om(\mathcal{X}_{f_i},\cdot) = -\dd \langle \mu, \phi(e_i)\rangle$
and since $\langle \mu, \phi(e_i)\rangle_p = \langle \phi(F(p)),\phi(e_i)\rangle = \langle F(p),e_i\rangle = f_i(p)$, we recover exactly Equation~\eqref{eqn:Xf} for each $f_i$.

Thus, at least in the case that the momentum map is proper, which is automatic if $M$ is compact, an integrable system can be viewed as a Hamiltonian action of $\R^n$. If the flow of each $\mathcal{X}_i$ is periodic, then this descends to an action of the $n$-torus $\T^n = (S^1)^n$, and in more generality, if all but $c$ of the actions are periodic then we obtain a Hamiltonian action of $\T^{n-c}\times \R^c$.

The case that $c=0$ is extremely well understood, as we will discuss in the next section, and later we will deal with higher values of $c$.

\section{Symplectic toric manifolds}
\label{sec:toric}

This section is devoted to the theory, and especially the classification, of symplectic toric manifolds.
Like the previous section, this section can be skipped or skimmed by experts.
Let $(M,\om)$ be a symplectic manifold of dimension $2n$, and suppose that a torus $\T^k$ acts on $M$ in a Hamiltonian fashion. Let $\mu\colon M\to\mathfrak{t}^*$ denote the moment map, and recall that the $\T^k$-action preserves $\mu$ by assumption. Thus, for any $p\in M$ we have that $\T^k\cdot p \subseteq \mu^{-1}(\mu(p))$, where $\T^k\cdot p$ is the $\T^k$-orbit of $p$. 

Recall that a group action is called \emph{effective} if there is no group element which acts trivially. 
It turns out that if there exists an effective and Hamiltonian action of $\T^k$ on $M$ then $k \leq \frac{1}{2}\dim(M)$.
The integer $c = \frac{1}{2}\dim(M) - k$ is called the \emph{complexity} of the $\T^k$-action.

\begin{definition}\label{def:toric}
    A \emph{symplectic toric manifold} is a Hamiltonian $\T^n$-space $(M,\om,\T^n,\Psi)$ where  $n = \frac{1}{2}\dim(M)$ and the $\T^n$-action on $M$ is effective.
\end{definition}

Note that some authors include the requirement that $M$ is compact in the definition above, but some do not, as we have here.

\begin{example}[Symplectic toric manifolds are integrable systems]\label{example:toric-int-syst}
Let $(M,\om,\T^n,\Psi)$ be a symplectic toric $2n$-manifold, and choose an identification $\phi\colon \mathfrak{t}^* \to \R^n$ induced by an isomorphism $\T^n\cong (S^1)^n$. Let 
\[ F = \phi\circ\Psi\colon M \to \R^n.\]
Then we claim that $(M,\om, F)$ is an integrable system. 
Indeed, the fact that $\T^n$ is abelian can be used to show that the components of $F$ Poisson commute, and the fact that the action is effective can be used to show almost everywhere independence. 
Such integrable systems, in which $M$ is compact and which each $\mathcal{X}_{f_i}$ has periodic flow of the same period, are called \emph{toric integrable systems}.
Conversely, any toric integrable system $(M,\om,F)$ can be associated with a symplectic toric manifold by defining the $\T^n$-action by the flows of the Hamiltonian vector fields and taking $\Psi = \phi^{-1}\circ F$ as the momentum map.
\end{example}

As the above example shows, toric integrable systems and symplectic toric manifolds are two different ways to view the same class of objects.

Let $\T^n$ be a torus and suppose that $(M,\om,\T^n,\mu)$ is an effective Hamiltonian $\T^n$-space.
Then Atiyah~\cite{Atiyah} and Guillemin-Sternberg \cite{GuSt1982} showed that the moment image $\mu(M)$ is a convex polytope in $\mathfrak{t}^*$, obtained as the convex hull of the images of the $\T^n$-fixed points. Following this, Delzant~\cite{De1988} described necessary and sufficient conditions on when a convex polytope $\De$ could appear as the moment image for a symplectic toric manifold, and furthermore showed that such polytopes, now called \emph{Delzant polytopes}, classify symplectic toric manifolds. We will review this classification now. For this exposition, we will view the symplectic toric manifolds as toric integrable systems, to better align with our treatment of semitoric systems (and their generalizations) later in the paper.

\subsection{The toric invariant: Delzant polytopes}

First, we will introduce the invariant of toric integrable systems independently, and in the next section we will discuss how to obtain this invariant from a toric integrable system (it is simply the image of the momentum map).

A \emph{convex polytope} $\De$ is a compact set in $\R^n$ which is the intersection of finitely many half-spaces. Equivalently, it is the convex hull of finitely many points, but the definition with half-spaces will be more useful for us. A half-space in $\R^n$ may be described by a normal vector $u\in\R^n$ and a scalar $\lambda$:
\[
 H_{u,\lambda} = \{x\in \R^n \mid x\cdot u \geq \lambda\}.
\]
The vector $u$ is normal to the boundary of $H_{u,\lambda}$ and points inwards towards $H_{u,\lambda}$. An integer vector $u\in \Z^n$ is called \emph{primitive} if whenever $u=kw$ for $w\in\Z^n$ and $k\in\Z$, then $k=\pm 1$. 

 A vertex $v$ of a polytope is a 0-dimensional face (which can also be described as a point of $\De$ which does not lie in the interior of any line segment contained in $\De$).

\begin{definition}
    A convex polytope $\De$ is \emph{rational} if it is the intersection of finitely many half-spaces of the form $H_{u,\lambda}$ where $u\in\Z^n$ and $\lambda\in\R$, and it is called \emph{simple} if $n$ faces meet at each vertex. A vertex $v$ of a rational convex polytope $\De$ is called \emph{smooth} if there are exactly $n$ faces meeting at $v$, and the  primitive inwards pointing normal vectors $u_1,\ldots,u_n\in\Z^n$ determined by the faces meeting at $v$ span the lattice $\Z^n$.
\end{definition}

Note that the existence of the inwards pointing normal vectors with integer entries is guaranteed by the rationality assumption.
Given $n$ vectors $u_1,\ldots, u_n\in\R^n$ we will use $\det(u_1,\ldots, u_n)$ to denote the determinant of the matrix with columns $u_1,\ldots,u_n$.
The condition of smoothness is equivalent to $\det(u_1,\ldots, u_n)=\pm 1$, where $u_1,\ldots, u_n\in\Z$ are the inwards pointing normal vectors of the faces meeting at the vertex $v$.

With these concepts, we are ready to define the invariant of toric integrable systems.

\begin{definition}
    An $n$-dimensional convex polytope $\De\subset \R^n$ is called a \emph{Delzant polytope} if it is rational, simple, and all vertices are smooth.
\end{definition}

In dimension two, the simplicity assumption is automatic.

\begin{example}\label{def:delzant}
Here are some examples of Delzant polytopes:
\begin{enumerate}[label = (\alph*)]
 \item the convex hull of $(0,0), (a,0), (0,b), (a,b)\in\R^2$ for any $a,b\in\R_{>0}$;
 \item the convex hull of $(0,0), (a,0), (0,a)\in\R^2$ for any $a\in\R_{>0}$ (Figure~\ref{fig:yes-delzant});
 \item let $e_1,\ldots,e_n\in\R^n$ be the standard basis, and let $\mathbf{0}\in\R^n$ denote the zero vector. Then the convex hull of $\mathbf{0}, ae_1,\ldots,ae_n\in\R^n$ is Delzant for any $a\in\R_{>0}$;
 \item the convex hull of $(0,0), (0,b), (a,b), (a+kb,0)\in\R^2$ for any $a,b\in\R_{>0}$ and any $k\in\Z_{>0}$.
\end{enumerate}
For a non-example, consider the convex hull of $(0,0), (2,1), (4,0)\in\R^2$. Then the vertex at $(2,1)$ does not satisfy the smoothness condition, since the primitive inwards pointing normal vectors to the adjacent edges are $(1,-2)$ and $(-1,-2)$, which do not span $\Z^2$ (their determinant is $-4$, not $\pm 1$). See Figure~\ref{fig:not-delzant}.
\end{example}

\begin{figure}
\begin{center}

\begin{subfigure}[b]{.4\linewidth}
\centering
\begin{tikzpicture}[scale=1.5]
\filldraw[draw=black, fill=gray!60] (0,0) 
  -- (0,2)
  -- (2,0)
  -- cycle;
\end{tikzpicture}
\caption{This polygon is Delzant.}
\label{fig:yes-delzant}
\end{subfigure}\qquad
\begin{subfigure}[b]{.4\linewidth}
\centering
\begin{tikzpicture}[scale=1.5]
\filldraw[draw=black, fill=gray!60] (0,0) 
  -- (2,1)
  -- (4,0)
  -- cycle;
\end{tikzpicture}
\caption{This polygon is \textbf{not} Delzant since the top vertex is not smooth.}
\label{fig:not-delzant}
  \end{subfigure}

\caption{}
\label{fig:two-triangles}
\end{center}
\end{figure}

\subsection{The Delzant classification}
\label{sec:delz-construction}

Let $(M,\om,F)$ be an $n$-dimensional compact toric integrable system.
Then, as a special case of results of Atiyah~\cite{Atiyah} and Guillemin-Sternberg~\cite{GuSt1982}, $F(M)\subset \R^n$ is a convex set, and by Delzant \cite{De1988}, it is a Delzant polytope (Definition~\ref{def:delzant}).
We say two toric integrable systems $(M,\om,F)$, $(M',\om',F')$ are \emph{isomorphic} if there is a symplectomorphism $\phi\colon M\to M'$ such that $\phi^* F' = F$.
Let $[(M,\om,F)]$ denote the isomorphism class of the toric system $(M,\om,F)$.

Given any Delzant polytope $\De$, it is possible to construct a compact toric integrable system $(M,\om,F)$ such that $F(M)=\De$. 
This construction is by an explicit algorithm involving symplectic reduction by a torus action on $\C^d$ where $d$ is the number of faces of $\De$; we describe this process in Example~\ref{ex:delzant}. 
Due to the existence of such a construction, the map taking a toric integrable system to its moment image is surjective onto the set of Delzant polytopes, and Delzant further showed that, if the toric systems are considered up to isomorphism, it is injective.

\begin{theorem}[Toric classification]\label{thm:toric-classification} 
The map taking $[(M,\om,F)]$ to $F(M)$ is a bijection from isomorphism classes of compact toric integrable systems to Delzant polytopes.
That is, compact toric integrable systems, up to isomorphism, are classified by Delzant polytopes.
\end{theorem}

\optional{might want to say something about isomorphism classes here?}

This classification is extremely useful:
it opens up the possibility to  work directly with the polytopes and obtain results about toric integrable systems, and toric systems are a key class of examples in several areas, including integrable systems, Hamiltonian group actions, mirror symmetry, etc.

\optional{remark that mentions isomorphisms could go here - LOOK AT JASON'S PAPER!}

Theorem~\ref{thm:toric-classification} is more than just an abstract bijection though, both directions of the map are relatively accessible to actually compute. Given an integrable system, the corresponding polytope is simply the image of its momentum map, which is furthermore equal to the convex hull of the images of the fixed points of the associated torus actions. Also, given a Delzant polytope $\De$, the corresponding toric integrable system (including the manifold, symplectic form, and torus action) can be obtained via the following algorithm, due to Delzant~\cite{De1988}. The details of the construction are also shown in~\cite[Section 29]{daSi2008}.

\optional{\textbf{ACTUALLY IMPORTANT:} maybe need to deal with the $2\pi$ here?}

\begin{example}[The Delzant construction]\label{ex:delzant}
 Let 
  \begin{equation}\label{eqn:Delzant-intersection}
   \De = \bigcap_{i=1}^d \{x\in \R^n \mid x\cdot u_i \geq \lambda_i\}
  \end{equation}
 be a compact $n$-dimensional Delzant polytope. Note that in our convention the vectors $u_1,\ldots, u_d$ are the inwards pointing normal vectors of the facets of $\De$, and assume that $d$ is minimal (so that none of the half spaces are redundant). Now we define a map 
 \[
  \rho\colon \R^d \to \R^n
 \]
 defined by $\rho(e_i) = u_i$, where $e_1,\ldots,e_d$ is the standard basis of $\R^d$. It turns out that $\rho$ descends to a map from $\R^d/\Z^d$ to $\R^n/\Z^n$. Let $N=\mathrm{ker}(\rho)$, and note that $N$ is a Lie subgroup of $(\R/\Z)^d$.

 Let $\mathfrak{t} := \mathrm{Lie}((\R/\Z)^d)$ and $\mathfrak{n} := \mathrm{Lie}(N)$.
 Consider the standard action of $(\R/\Z)^d$ on $\C^d$, and let $\Phi\colon \C^d\to \mathfrak{t}^*$ be the choice of moment map which has $\Phi(0) = (\lambda_1,\ldots,\lambda_d)$. 
 The inclusion $N\hookrightarrow (\R/\Z)^d$ induces an inclusion map $i\colon \mathfrak{n}\hookrightarrow \mathfrak{t}$ with dual $i^*\colon  \mathfrak{t}^* \rightarrow \mathfrak{n}^*$, and this combined with the moment map $\Phi$ induces a Hamiltonian action of $N$ on $\C^d$ with a moment map $\Phi_N = i^*\circ \Phi$. Now we define a symplectic manifold by symplectic reduction of $\C^d$ by this action of $N$ at the level $\Phi_N=0$. That is, $M_\De = \Phi_N^{-1}(0)/N$ with a symplectic form $\omega_\Delta$ the one guaranteed by the Marsden-Weinstein-Meyer Theorem~\cite{marsden-weinstein,meyer}. It can be checked that the requirements of the Marsden-Weinstein-Meyer Theorem are satisfied, and $(M_\Delta,\om_\De)$ is therefore a smooth symplectic manifold of dimension $2n$.
 Furthermore, the torus $T_\Delta := (\R/\Z)^d/N$ has dimension $n$ and its action descends to an effective Hamiltonian action on $M_\Delta$ with momentum map $\mu_\Delta$. Thus, $(M_\Delta,\om_\Delta,T_\De,\mu_\Delta)$ is a symplectic toric manifold. Finally, it can be verified that $\mu_\Delta (M_\De) = \De$, as desired.
\end{example}

The construction given in Example~\ref{ex:delzant} provides a technique to explicitly construct symplectic toric manifolds from the associated Delzant polygons.
Let's now show how it works on a specific example.

\begin{example}[The Delzant triangle]
Let $\De$ be the triangle in $\R^2$ which has vertices at $(0,0)$, $(\pi,0)$ and $(0,\pi)$. Then it can be written in the form of Equation~\eqref{eqn:Delzant-intersection} by taking 
\[
 u_1 = (1,0),\,u_2 = (0,1),\, u_3 = (-1,-1),\, \lambda = (0,0,-\pi).
\]
Thus, the map $\rho\colon (S^1)^3\to (S^1)^2$ is given by $\rho(a,b,c) = (a-c, b-c)$ which has kernel
\[
 N = \{(a,a,a)\in (S^1)^3 \mid a\in S^1\} \cong S^1.
\]
The action of $N$ on $\C^3$ is Hamiltonian with momentum map $\Phi_N(z_1,z_2,z_3) = -\pi (|z_1|^2+|z_2|^2+|z_3|^2) + \pi$.
Thus,
\[
 M_\De = \Phi_N^{-1}(0)/N = \{z\in\C^3 \mid |z|^2 = 1\}/S^1.
\]
where the $S^1$-action on $\Phi_N^{-1}(0)$ is given by $\alpha \cdot (z_1,z_2,z_3) = (\alpha z_1, \alpha z_2, \alpha z_3)$.
Thus, $M_\Delta$ is $\mathbb{CP}^2$ and it can be seen that it inherits a multiple of the Fubini-Study symplectic form. The quotient $(S^1)^3/N$ is diffeomorphic to a 2-torus, can be identified with $\{(a,b,c)\in(S^1)^3 \mid c=0\}$, and acts on $M_\De$ by rotating the first two coordinates, which is the standard $2$-torus action on $\mathbb{CP}^2$.

\optional{could try to say what the multiple is}
    
\end{example}

\begin{remark}
The toric classification, Theorem~\ref{thm:toric-classification}, is restricted to \emph{compact} toric systems, but Karshon and Lerman~\cite{karshon-lerman} have extended this result to the case of non-compact toric systems by including additional invariants. 
\optional{It would be good to explain this in a few sentences}
\end{remark}

\subsubsection{Looking beyond the toric case}
Let $\mathcal{M}_T$ denote the set of isomorphism classes of toric integrable systems of dimension $2n$  
and let $\mathbf{D}$ denote the set of Delzant polytopes of dimension $n$. 
Then Theorem~\ref{thm:toric-classification} states that the map
\begin{equation}\label{eqn:toric_map}
 [(M,\om,F)] \mapsto F(M)
\end{equation}
produces a bijection from $\mathcal{M}_T$ to $\mathbf{D}$, where $[(M,\om,F)]$ denotes the isomorphism class of the toric system $(M,\om,F)$.
The algorithm described in Example~\ref{ex:delzant} shows how to compute the inverse of this map.

From here, there are two nice things about this situation that we would like to generalize beyond toric systems:
\begin{enumerate}
    \item the map given in Equation~\eqref{eqn:toric_map} is a bijection. This is the statement that toric systems are classified up to isomorphism by their moment image, and allows for using the geometry of polytopes to understand toric systems. 
    \item Delzant's construction shows how to obtain an explicit system from the invariant. That is, the construction produces a known symplectic manifold, in this case obtained as a quotient of $\C^d$, and the Hamiltonian for the torus action is an explicit function on $M_\De$.
\end{enumerate}
In the case of semitoric systems, discussed in Section~\ref{sec:st}, the first item above has been achieved: the semitoric classification theorem of Pelayo and V\~{u} Ng\d{o}c~\cite{PeVN2009,PeVN2011} generalizes the toric classification in dimension four for a generic class of semitoric systems, and this was further generalized to all semitoric systems in~\cite{PPT-nonsimple}. 
We discuss these classifications in Section~\ref{sec:st}. The second item above, obtaining a technique to explicitly and globally write out a semitoric system from the invariants, has proven to be more difficult. 
In Section~\ref{sec:st-families}, we will discuss a technique developed jointly with Y.~Le Floch to obtain explicit examples of semitoric systems with some, but not all, of the invariants prescribed~\cite{LFPal,LFPal-SF2}. The techniques discussed in Section~\ref{sec:st-families} have proven fruitful, and been used to obtain many new explicit examples of semitoric systems, such as in~\cite{HoPa2017,HohMeu}.
In Section~\ref{sec:beyond}, we discuss hypersemitoric and complexity one integrable systems, each of which is a generalization of semitoric integrable systems which has not yet been classified.

\section{Pointwise and local classifications in integrable systems}
\label{sec:local-class}

Let $(M,\om,F=(f_1,\ldots,f_n))$ be an integrable system and let $p\in M$. In this section we will discuss classifications, up to the appropriate isomorphism, of behavior of the system at $p$ (pointwise classifications) and in a neighborhood of $p$ (local classifications). A point $p$ is called \emph{regular} if $\mathrm{rank}(\dd F_p)=n$ and called \emph{singular} otherwise, and these will be the two cases we consider.

\subsection{Regular points: the Liouville-Arnold-Mineur theorem}
\label{sec:actionangle}

Recall Example~\ref{ex:cot-bundle-torus} of a toric integrable system $(T^*\T^n, \om_{T^*\T^n}, \pi_{\R^n})$ on the cotangent bundle of the torus, $T^*\T^n\cong \T^n\times \R^n$, where $\pi_{\R^n}$ is projection onto the $\R^n$ factor.
The following theorem states that the neighborhood of any compact regular fiber can be modeled by a neighborhood of the zero section in this system.

\begin{theorem}[{Liouville-Arnold-Mineur Theorem~\cite{arnold89}}]\label{thm:LAM}
Let $(M,\om,F)$ be an integrable system, let $c\in F(M)$, and let $\Lambda$ be a connected component of $F^{-1}(c)$. If $\Lambda$ is compact and all points in $\Lambda$ are regular points of $F$, then there exists open neighborhoods $U\subseteq M$ of $\Lambda$ and $V\subseteq T^*\T^n$ of the zero section, such that
$(U,\om,F|_U)$ and $(V,\om_{T^*\T^n},(\pi_{\R^n})|_V)$ are isomorphic integrable systems in the following sense: there exists a symplectomorphism $\phi\colon U \to V$ and a local diffeomorphism $g\colon \R^n\to\R^n$ such that $g\circ F|_U = (\pi_{\R^n})|_V \circ \phi$. That is, the following diagram commutes:
\[
\begin{tikzcd}[column sep = tiny]
M \arrow[r,phantom,"\supseteq"]& U \arrow[rrr, "\phi"] \arrow[d, "{F}"'] & & & V  \arrow[d, "\pi_{\R^n}"] &T^*\T^n \arrow[l,phantom,"\subseteq"]\\
& \mathbb{R}^n \arrow[rrr, "g"] & & & \mathbb{R}^n &
\end{tikzcd}
\]
\end{theorem}

The above statement is restricted to compact fibers, but the full statement of the theorem is actually somewhat more general, dealing also with non-compact fibers. Many of the systems that we will work with require at least that $F$ is proper, so this version of the theorem is already very helpful.

There exist several proofs of Theorem~\ref{thm:LAM} in the literature. For instance, there is a recent treatment in~\cite{VNSepe}, and there are also proofs, often of slightly different statements, in~\cite{mineur,arnold89,GS-book, bates-sni-actionangle, Dui88,HZ-book}.
There is a generalization of this theorem for singularities in~\cite{zung96}.

In the situation of Theorem~\ref{thm:LAM}, the $p_i$ and $q_i$ coordinates on $T^*\T^n$ induce coordinates in a neighborhood of $F^{-1}(c)$, which are called \emph{action} and \emph{angle} coordinates, respectively.

For the next proposition, let us use action-angle coordinates $(q_1,\ldots,q_n,p_1,\ldots,p_n)$ viewed as functions on a neighborhood of $F^{-1}(c)$ in $M$; technically, these are obtained by pulling back the $p_i$ and $q_i$ from the model $T^*\T^n$ to $M$ via the symplectomorphism $\phi$. The following is also as discussed in~\cite{arnold89}.

\begin{proposition}\label{prop:reg-action}
    Let $c\in F(M)$ be a regular value and denote $\Lambda_c :=F^{-1}(c)$, and suppose that $\Lambda_c$ is compact and connected and $(q_1,\ldots,q_n,p_1,\ldots,p_n)$ are a set of action-angle coordinates around $\Lambda_c$. Let $c'\in F(M)$ be any value sufficiently close to $c$ so that the action-angle coordinates are still defined on $\Lambda_{c'}$. Let $[\gamma_j^{c'}]\in H_1(\Lambda_{c'})$ denote the cycle formed by fixing all angle variables $q_i$ for $i\neq j$ and letting $q_i$ go from $0$ to $2\pi$. Then, for any $x'\in F^{-1}(c')$, the following \emph{action integral} can be used to compute $p_j$:
    \[
      p_j(x') = \frac{1}{2\pi}\oint_{\gamma_j^{c'}}\alpha
    \]
    where $\alpha$ is any primitive of $\omega$ defined in a neighborhood of $\Lambda_c$ which includes $\Lambda_{c'}$. That is, $\omega = \dd \alpha$.
\end{proposition}

\begin{proof}
 In the model $T^*\T^n$ a primitive of $\omega$ is given by $\alpha = \sum_i p_i \dd q_i$. Since $p_j$ is constant along $\gamma_j$, we compute
 \[
  \oint_{\gamma_j^{c'}} \alpha = \int_0^{2\pi}p_j \dd q_j = 2\pi p_j (x').
 \]
 Thus, the equation holds for any primitive of $\omega$, since any such one-form would differ from $\alpha$ by an exact form.
 Since the desired formula holds in the model, pulling back by $\phi$ we obtain that the formula also holds locally in $M$.
\end{proof}

In Proposition~\ref{prop:reg-action} the cycles $\gamma_j$ are determined by the action-angle coordinates: in the theorem we describe them in terms of the angle coordinates, but they can also be described by following the flow of $\mathcal{X}_{p_j}$. This may make the proposition seem not very useful, since we must first know the action-angle coordinates before we can compute them, but an important implication of this proposition is that \emph{there exist} cycles $\gamma_1^c,\ldots,\gamma_n^c$ which are a basis of $\Lambda_c$ and for which the action variables can be computed by integrating over these cycles. In fact, any such choice of cycles will produce a choice of action coordinates.

Action variables are an important invariant of integrable systems, and in Section~\ref{sec:affine} we will see that they produce what is called an integral affine structure on the base of a fibration induced by $F$.
In Section~\ref{sec:construct-labels} we will discuss situations in which such action integrals are generalized to certain singular points, called focus-focus points, though the actions obtained in that case are singular (they blow up at the focus-focus point, and also exhibit monodromy).
Nevertheless, V\~{u} Ng\d{o}c~\cite{VN2003} was able to use these singular actions to obtain a semilocal invariant around focus-focus singularities.

\optional{cite open questions paper bolsinov section}

\subsection{Singularities in integrable systems}
\label{subsec:singularities}

The structure near singular points is extremely rich compared with the structure near regular points. In this section, we will actually only deal with what are called \emph{non-degenerate} singular points, which is a generic condition similar to Morse non-degeneracy, and in this situation there are several cases. We will start with a pointwise classification (stated in terms of the Hessian of $F$) and then discuss the local version. We will see that these classifications are similar to the classification of critical points of Morse functions by their Morse index: non-degenerate singular points can be written as products of only four different types of factors (elliptic, hyperbolic, focus-focus, and regular).

For more details on the following, see the discussion in the book by Bolsinov and Fomenko~\cite{Bolsinov-Fomenko}.
Let $p$ be a singular point, then $\mathrm{rank}(\mathrm{d}F(p))$ is called the \emph{rank of $p$} and denoted by $\text{rank}(p)$.
The Hessians of $f_1,\ldots,f_n$ at $p$ lie in the space of quadratic forms on $T_pM$, denoted here by $Q(T_pM)$, which can be equipped with a Lie algebra structure isomorphic to $\mathfrak{sp}(2n,\R)$.

To start, suppose that $p$ is a singular point of rank zero. Then, since $f_1,\ldots,f_n$ Poisson-commute, their Hessians commute, and thus span an abelian subalgebra of $Q(T_pM)$.
Then $p$ is called \emph{non-degenerate} if the span of the Hessians is a Cartan subalgebra\footnote{A Cartan subalgebra of a Lie algebra is a maximal abelian subalgebra of semisimple elements.} of $Q(T_pM)$.
If $p$ is a singular point with $r:=\mathrm{rank}(p)>0$, then there is also a notion of non-degeneracy: roughly, the Hessians of $f_1,\ldots,f_n$ descend to the quotient of $T_pM$ by a space related to the non-singular part of the momentum map, and taking the span of the images of the Hessians in this space yields a subalgebra of the space of quadratic forms on $\R^{2(n-r)}$. As in the rank zero case, we then say that $p$ is non-degenerate if this subalgebra is Cartan.
For details, see~\cite{Bolsinov-Fomenko}, or for a quick overview in the case of four dimensions see~\cite[Section 2.1]{LFPal-SF2}.

Williamson~\cite{Williamson} classified all Cartan subalgebras of $\mathfrak{sp}(2n,\R)\cong Q(T_pM)$, which in turn implies a classification of non-degenerate singular points.
Let $\om_{\R^{2n}}$ denote the standard symplectic form on $\R^{2n}$.

\begin{theorem}\label{thm:williamson}
Let $Q(2r,\R)$ denote the set of quadratic forms on $\R^{2r}$, with Lie algebra structure isomorphic to $\mathfrak{sp}(2r,\R)$. 
Then for any Cartan subalgebra $\mathcal{C}\subset Q(2r,\R)$ there exist coordinates $(x_1,\ldots,x_r,y_1,\ldots,y_r)$ such that $\omega_{\R^{2r}} = \sum_{i=1}^r \mathrm{d}x_i \wedge \mathrm{d}y_i$, and such that $\mathcal{C}$ has a basis of the form $q_1,\ldots,q_r$ where for each $i\in\{1,\ldots, r\}$ one of the following holds:
\begin{enumerate}
    \item $q_i = x_i^2+y_i^2$ (elliptic block);
    \item $q_i = x_i y_i$ (hyperbolic block)
    \item $q_i = x_i y_{i+1}- x_{i+1}y_i$ and $q_{i+1} = x_i y_i+ x_{i+1} y_{i+1}$ (focus-focus block).
\end{enumerate}
\end{theorem}

Let $\mathbf{k}_e$ denote the number of elliptic blocks, $\mathbf{k}_h$ denote the number of hyperbolic blocks, and $\mathbf{k}_{ff}$ denote the number of focus-focus blocks, and notice that $\mathbf{k}_e + \mathbf{k}_h + 2\mathbf{k}_{ff} = r$.
We consider a singular point of rank $r$ in a $2n$-dimensional integrable system to have $n-r$ \emph{regular blocks}.

\begin{definition}\label{def:will-type}
Let $p\in M$ be a non-degenerate singular point of an $n$-dimensional integrable system $(M,\om,F)$ with $\mathrm{rank}(p)=r$, and let $\mathcal{C}\subset Q(2r,\R)$ denote the Cartan subalgebra associated with $p$. 
The \emph{Williamson type} of $p$ is the triple $(\mathbf{k}_e,\mathbf{k}_h, \mathbf{k}_{ff})$ associated to $\mathcal{C}$ by Theorem~\ref{thm:williamson}.
\end{definition}

Singular points are often discussed by listing the blocks. For instance, an elliptic-elliptic-regular point would be a rank 2 singular point in a 6-dimensional integrable system for which the corresponding Cartan subalgebra has Williamson type $(\mathbf{k}_e,\mathbf{k}_h, \mathbf{k}_{ff}) = (2,0,0)$.

Notice that the Williamson classification of a singular point $p$ is a completely pointwise notion, only concerned with the Hessians acting on $T_pM$. 
It is natural to wonder if there is an analogue of the Morse lemma in this case, i.e.~a local version of this classification. 
That is, do there exist symplectic coordinates on the \emph{manifold $M$} (as opposed to $T_pM$) for which the fibration of the integrable system near $p$ can be put in a normal form depending on the Williamson type of $p$?

To answer this question in the affirmative, we have the following statement:

\begin{theorem}[Eliasson normal form~\cite{Eliasson-thesis,Eli90,miranda-zung}]\label{thm:eliasson}
    Let $p\in M$ be a non-degenerate singular point of the $2n$-dimensional integrable system $(M,\om,F)$. 
    Then there exist local coordinates $(x_1,\ldots, x_n,y_1,\ldots,y_n)$ on an open neighborhood $U$ of $p$ and a map $Q = (q_1,\ldots,q_n)\colon U\to\R^n$ whose components $q_i$ each satisfy one of the following:
    \begin{enumerate}
    \item $q_i = x_i^2+y_i^2$ (elliptic block);
    \item $q_i = x_i y_i$ (hyperbolic block)
    \item $q_i = x_i y_{i+1}- x_{i+1}y_i$ and $q_{i+1} = x_i y_i+ x_{i+1} y_{i+1}$ (focus-focus block),
    \item $q_i = y_i$ (regular block),
\end{enumerate}
such that $p$ corresponds to $(x,y)=(0,0)$, and $\{q_i,f_j\}=0$ for all $i,j$. Moreover, if none of the $q_i$ are hyperbolic, there exists a local diffeomorphism $g\colon \R^n\to\R^n$ with $g(F(p))=0$ and such that $g\circ F|_U =  Q$.
\end{theorem}

To my knowledge, there does not exist a complete proof of this theorem anywhere in the literature. 
It was originally proved only in the analytic case by Vey~\cite{Vey}, and since then various special cases have been proved in the smooth case: completely elliptic in all dimensions~\cite{DufMol,Eli90}, focus-focus in dimension four~\cite{VNWac,Cha}, and in all cases (hyperbolic and elliptic) in dimension two~\cite{ColVey}.
There is also an equivariant version~\cite{miranda-zung}, which is proved using Theorem~\ref{thm:eliasson}.
This is also discussed around Theorem 2.1 in \cite{LFVN} and in \cite[Remark 4.16]{VNSepe}.

\begin{remark}
    Though there is not a complete proof of Theorem~\ref{thm:eliasson} in the smooth category appearing anywhere in the literature, in most cases there is a straightforward way to avoid explicit dependence on this theorem while remaining rigorous. Namely, one can take the existence of the local coordinates of Theorem~\ref{thm:eliasson} to be the \emph{definition} of what it means for a point to be non-degenerate, instead of using the definition presented above about Cartan subalgebras. 
    While this workaround is somewhat unsatisfactory, it provides a rigorous and practical way to proceed until a complete proof of Theorem~\ref{thm:eliasson} becomes available.    
\end{remark}

Recall that in Example~\ref{ex:non-deg-examples} we discussed each of these local models as examples of integrable systems on either $\R^2$ or $\R^4$. Theorem~\ref{thm:eliasson} says that all non-degenerate singular points are locally isomorphic to one of those models.

\begin{example}[Morse functions]
    Let $\Sigma$ be any surface equipped with an area form $\om$. Recall that area forms and symplectic forms are equivalent on surfaces. Let $f\colon M\to \R$ be any Morse function. Then $(M,\om,f)$ is a non-degenerate integrable system. The singularities of this integrable system are the critical points of $f$. The critical points with Morse index 0 or 2 are elliptic singularities, and those with Morse index 1 are hyperbolic singularities.
\end{example}

\begin{example}[Products]
   Suppose that $(M_i,\om_i,F_i)$ is a non-degenerate integrable system of dimension $n_i$ for $i=1,2$.
   Let $F_{12}\colon M_1\times M_2\to\R^{n_1+n_2}$ be given by $F_{12}(p_1,p_2) = (F_1(p_1),F_2(p_2))$. Then 
   \[
    (M_1\times M_2, \om_1 \oplus \om_2,F_{12})
   \]
   is also a non-degenerate integrable system.
\end{example}

\begin{example}
Suppose that $(M,\om,F)$ is a toric integrable system, in the sense of Definition~\ref{def:toric}.  Then all singular points are elliptic. That is, each singular point of rank $k$ has $k$ regular blocks and $n-k$ elliptic blocks.
\end{example}

\begin{remark}\label{rmk:elliptic-actions}
Singular points that have only elliptic and regular blocks, like those that appear in toric integrable systems, also possess action integrals. That is, the components of the function $g$ from Theorem~\ref{thm:eliasson} can be computed as integrals over cycles, as in Proposition~\ref{prop:reg-action}. See~\cite{miranda-zung}.
\end{remark}

\subsection{Integral affine structures and integrable systems}
\label{sec:affine}

The action-angle coordinates of Section~\ref{sec:actionangle} give a semi-local understanding of what is happening around regular points of the system, but they fall short of being a global invariant. There can be obstructions to obtaining global action-angle coordinates for all regular points simultaneously; this was investigated in the seminal paper by Duistermaat~\cite{Dui88}, in which he introduced the concept of the integral affine structure induced by an integrable system. Informally, this encodes the way that the local action-angle coordinates, and specifically the actions, fit together.

Let $(M,\om,F)$ be a $2n$-dimensional integrable system,
and define an equivalence relation on $M$ by $p\sim q$ if and only if $p$ and $q$ are in the same connected component of the same fiber of $F$.
Then we define \emph{the base of the singular Lagrangian fibration induced by $F$} to be $B=M/\sim$.
This terminology comes from the fact that the regular fibers of $M\to M/\sim$ are Lagrangian submanifolds, which quickly follows from a dimension count and the fact that $(\mathcal{X}_{f_1})_p,\ldots,(\mathcal{X}_{f_n})_p$ span the tangent plane of the fiber containing $p$ at $p$, and $\om(\mathcal{X}_{f_i},\mathcal{X}_{f_j}) = \{f_i,f_j\}=0$. The other fibers are either too low of dimension to be Lagrangian (i.e.~they are isotropic) or they are not smooth submanifolds, or both.

Let $B_r\subset B$ denote the set of those points $[p]\in B$ such that $q$ is a regular point of $F$ for all $q\in [p]$, and we call the points in $B_r$ \emph{regular fibers}.
In the case that the fibers of $F$ are compact, the set $B_r$ inherits what is called an integral affine structure.
Let
$\aff = \R^n \rtimes GL(n,\Z)$ denote the group of integral affine transformations.

\begin{definition}
    An \emph{integral affine structure} on an $n$-dimensional manifold $N$ is a maximal atlas $\mathcal{A} = \{(U_\alpha,\phi_\alpha)\}_{\alpha \in \mathcal{I}}$ on $N$ such that for each $\alpha,\alpha'\in\mathcal{I}$, the transition map $\phi_\alpha \circ \phi_{\alpha'}^{-1}$ is equal to a restriction of an element of $ \aff$ to its domain.
    The pair $(N,\mathcal{A})$ is called an \emph{integral affine manifold}.
\end{definition}
In the case that the fibers of $F$ are connected, we may associate $B$ with $F(M)$, and $B_r$ with the regular values in the image of $F$. In this case, it is often said that the regular values of $F$ inherit an integral affine structure from the integrable system.
This is the scenario in several important classes of integrable systems, such as toric systems (see Section~\ref{sec:toric}) and semitoric systems (see Section~\ref{sec:st}).
On the other hand, many important examples of integrable systems have disconnected fibers, such as many systems which include singular points with hyperbolic types, and in particular the hypersemitoric systems introduced and studied in~\cite{HP-extend} (which we discuss in Section~\ref{sec:hst}).

If $(N,\mathcal{A})$ is an integral affine manifold, then for any loop $\gamma\colon S^1 \to N$ and point $p\in \gamma(S^1)$, the parallel transport of any vector $X\in T_pM$ along the loop yields another vector in $T_pM$, related to the original one by an integral affine transformation. 
This induces a map from $T_pM$ to itself, and for each choice of coordinates around $p$ we obtain an element of $\aff$.
This construction only depends on the homotopy class of $\gamma$. Let $\pi_1(N,p)$ denote the fundamental group with base point $p$.

\begin{definition}\label{def:monodromy}
    Let $(N,\mathcal{A})$ be an integral affine manifold and $p\in N$. 
    For each integral affine chart around $p$, the \emph{monodromy map at $p$} is the map
    \[
     \mathfrak{M}_p\colon \pi_1(N,p) \to \aff
    \]
    described above.
\end{definition}

The monodromy is an important property of an integral affine structure, since non-trivial monodromy obstructs the existence of global action-angle variables~\cite{Dui88}.
We will see that the natural integral affine structure on the base of a semitoric integrable system has nontrivial monodromy due to the existence of focus-focus points~\cite{Zou,Mat96,Zung97} (and this implies the existence of quantum monodromy in the associated quantum integrable system \cite{cushman-duist-pendulum,SaZh1999}).
Monodromy is explored more in~\cite{MaBrEf-monodromy}.
The following simple observation will be important for us later: if $N$ is simply connected, then the monodromy of any affine structure on $N$ is trivial.

Note that $\R^n$ has a natural integral-affine structure: consider the global chart $U = \R^n$ with $\phi\colon \R^n\to\R^n$ where $\phi$ is the identity map, and take the atlas to be all charts compatible with this one. 
That is, the standard integral affine structure on $\R^n$ is given by the atlas \[\mathcal{A}_{\R^n} = \{(U,\phi) \mid U\subset \R^n \text{ open and there exists } \psi \in \aff \text{ such that }\phi=\psi|_U\}.\]
Furthermore, any submanifold of $\R^n$ inherits an affine structure from $\mathcal{A}_{\R^n}$. By using the action-angle coordinates to compute the affine structure of a toric system, one can prove the following:

\begin{proposition}\label{prop:toric-integralaffine}
 Let $(M,\om,F)$ be a compact toric integrable system, and let $B_r \subset F(M) \subset \R^n$. Then the integral affine structure that $B_r$ inherits as a subset of $\R^n$ and the one that it inherits as the base of the integrable system $(M,\om,F)$ are equal.
\end{proposition}

\optional{do this proof}

Recall that compact toric integrable systems are classified by their moment image, which is a Delzant polytope. Proposition~\ref{prop:toric-integralaffine} tells us that the moment image also holds the data of the integral affine structure of $B_r$; essentially, in this case we can ``see'' the integral affine structure of $B_r$ by simply looking at $F(M)$. When moving to more complicated situations, such as semitoric systems, we will see that the integral affine structure plays a key role in the classification, and a polytope (or polygon) is a useful way to encode it. Recall further, from Proposition~\ref{prop:reg-action}, that action coordinates around a regular point can be computed via integrating a primitive of the symplectic form over certain cycles, and obtaining coordinates from integrating over such cycles is also something we will generalize to further situations.

\begin{remark}\label{rmk:int-affine-lattice}
    Suppose that the fibers of $F|_{F^{-1}(B_r)}\colon F^{-1}(B_r) \to B_r$ are connected. Then, the integral affine structure on $B_r$ can also be encoded as a lattice in $T^* B_r$, which is equivalent information as an integral affine atlas on $B_r$. 
    For instance this is the perspective taken in~\cite{PPT-nonsimple}.
    In the case that the fibers of $F$ are connected, the period lattice is formed in the following way: let $b\in B_r\subset \R^n$. Then each $\beta\in T^*_b\R^n$ determines a vector field $\mathcal{X}_\beta$ on $F^{-1}(b)$ by flowing along $\mathcal{X}_\beta$ for time 1. Now, define $\mathcal{L}_b\subset T^*_b\R^n$ to be such that $2\pi\mathcal{L}$ is the isotropy group for this action. It turns out that each such $\mathcal{L}_b$ is a lattice of full rank in $T^*_b\R^n$. Then the period lattice on $B_r$ is then $\mathcal{L} = \sqcup_{b\in B_r}\,\mathcal{L}_b$ (c.f.~\cite{Dui88}). 
    This construction can also be extended from $B_r$ to all of $B$, but in that case there may be values of $b$ such that $\mathcal{L}_b$ is not full rank in $T^*_b\R^n$.
\end{remark}

\begin{remark}\label{rmk:int-affine-question}
    Proposition~\ref{prop:toric-integralaffine} and Theorem~\ref{thm:toric-classification} together suggest that for toric integrable systems the integral affine structure on $B_r$ determines the toric system (up to the appropriate notions of isomorphism), and indeed, this is the case\footnote{Toric integrable systems up to symplectomorphisms which preserve the fibration (but not necessarily the momentum map itself) are classified by Delzant polytopes up to integral affine transformations (translations and $\mathrm{GL}(n,\Z)$), which is the same as the data of the integral affine structure of the polytope.}. 
    It is natural now to wonder in what cases an integral affine structure can determine the associated integrable system. 
    To attack this problem, one obstacle is that a notion of a \emph{singular integral affine structure} needs to be rigorously developed on all of $B$, extending the usual integral affine structure on $B_r$.
    This is particularly difficult in the case that the fibers of $F$ are not connected, in which case $B$ cannot be identified with $F(M)\subset \R^n$.
    Once such a structure is defined for a class of integrable systems, a key open question is: under what conditions does $B$, equipped with this structure, determine the integrable system up to fiberwise symplectomorphism?
    This is a delicate question: there are examples of integrable systems where a reasonable definition of singular integral affine structure exists and in which $B$ with this structure is known to not determine the associated integrable system. Nevertheless, it remains possible that the system can be recovered in a broad class of cases. 
    Even partial results in this direction, understanding specific classes of system that are or are not determined by the singular integral affine manifold $B$, or understanding what properties are encoded in $B$, would be of high interest.
    For further discussion on this question, see~\cite[Question 2.1]{open-questions}.
\end{remark}

\begin{remark}
    In addition to the monodromy of the integral affine structure on the base, Duistermaat introduced another invariant of a regular Lagrangian fibration in~\cite{Dui88}, a Chern class (which automatically vanishes if the symplectic form is exact, for instance on a cotangent bundle).
    There is a thorough description of Duistermaat's Chern class in~\cite[Remark 1.7]{karshon-lerman}.
    It is implicit in~\cite{Dui88} that the base of a regular Lagrangian fibration (with its structure as an integral affine manifold) and this Chern class together classify the fibration (this is explained in detail in~\cite[Example 2]{Mol}). Recently, this was extended from purely regular fibrations to those which also include singularities with elliptic blocks, in which case the base is an integral affine manifold with corners, see~\cite{Mol}.
\end{remark}

\section{Semitoric integrable systems}
\label{sec:st}

The toric classification, Theorem~\ref{thm:toric-classification}, has had a wide-ranging impact on many fields of mathematics, and there are many different ways that one could attempt to generalize this result. 
From the point of view of integrable systems, a natural next step is to consider a system for which some, but not all, of the Hamiltonian flows of the components of the momentum map generate a torus action. 
In this section, we will concentrate on an important class of 4-dimensional systems for which one of the two integrals generates a periodic flow and for which the singularities of the system are relatively nice.
We use the convention of identifying $S^1$ with $\R/(2\pi \Z)$, so the flow of a vector field generating an effective $S^1$-action is equivalent to it having a periodic flow with period $2\pi$.

The following important class of systems was first introduced by V\~{u} Ng\d{o}c~\cite{VN2007} and Pelayo-V\~{u} Ng\d{o}c~\cite{PeVN2009}.
Recall that a map is called \emph{proper} if the preimage of every compact set is compact.

\begin{definition}
    A \emph{semitoric integrable system} (sometimes called a \emph{semitoric system} for brevity) is a 4-dimensional integrable system $(M,\om,F=(J,H))$ such that 
    \begin{enumerate}
        \item $J\colon M\to\R$ is proper,
        \item the flow of $\mathcal{X}_J$ generates an effective $S^1$-action on $M$,
        \item all singularities of $F$ are non-degenerate, and do not include any hyperbolic blocks (as in Theorem~\ref{thm:eliasson}).
    \end{enumerate}
    A semitoric system is called \emph{simple} if it also satisfies the following generic condition:
\begin{itemize}
    \item[4.] there is at most one singularity of focus-focus type in each level set of $J$.
\end{itemize}
\end{definition}

Note that item (1) is automatic if $M$ is compact, item (2) is equivalent to requiring that $(M,\om,J)$ is a Hamiltonian $S^1$-space in the sense of Karshon~\cite{karshon}, and item (3) is equivalent to requiring that all singularities which arise in a semitoric system are either elliptic-elliptic, focus-focus, or elliptic-regular, as in Theorem~\ref{thm:eliasson}.
We will call any fiber of $F$ which consists of only regular a focus-focus points a \emph{focus-focus fiber}, and we will see that understanding these fibers, and the impact of their existence on the global structure of the system (due to the monodromy around such fibers) is the key to obtaining the classification of semitoric systems.

\begin{definition}\label{def:st-isomorphism}
Two semitoric systems $(M,\om,F=(J,H))$ and $(M',\om',F'=(J',H'))$ are called \emph{isomorphic} if there exists a pair $(\Phi,g)$ where $\Phi\colon M\to M'$ is a symplectomorphism such that $g\colon \R^2 \times \R^2$ is a local diffeomorphism satisfying $g\circ F = F' \circ \Phi$, $g(x,y) = (x,g_2(x,y))$, and $\frac{\partial g_2}{\partial y}>0$. The pair $(\Phi,g)$ is called an \emph{isomorphism of semitoric systems}.
\end{definition}

Note that if $(\Phi,g)$ is a semitoric isomorphism, then $\Phi^*J' = J$, so $\Phi$ is an $S^1$-equivariant symplectomorphism. Furthermore, notice that a semitoric isomorphism is a special case of an isomorphism of integrable systems, with the extra restrictions that $\Phi$ preserve the moment map for the $S^1$-action and the orientation requirement that $\frac{\partial g_2}{\partial y}>0$ (to be an isomorphism of integrable systems, this value would only have to be nonzero).

Semitoric systems are a natural class of integrable systems, which appears in several important examples in physics, such as the coupled angular momenta system~\cite{SaZh1999,LFP}. Furthermore, by replacing item (1) above with the weaker requirement that the joint map $F=(J,H)$ is proper, even more physical systems fit the definition, such as the familiar spherical pendulum (Example~\ref{ex:spherical-pen}). Such systems are called \emph{proper semitoric} and are discussed in~\cite{PeRaVN2015}.
Semitoric systems, and their generalization almost-toric fibrations, also appear in various aspects of symplectic topology, see Section~\ref{sec:close-ATFs-symp-top}.

\optional{if we cite some SYZ stuff put it here}

Now we introduce a foundational example of a semitoric system.

\begin{example}[{Coupled angular momenta system~\cite{SaZh1999,LFP}}]\label{ex:coupledspins}
Let $\om_{S^2}$ denote the standard symplectic form on $S^2$ (the one for which the total volume is $4\pi$).
We will now define an integrable system dependent on parameters $0<R_1<R_2$ and $t\in [0,1]$.
Let $(M,\om_{R_1,R_2})$ be the symplectic manifold with $M = S^2\times S^2$ and $\om_{R_1,R_2} = R_1\om_{S^2} \oplus R_2\om_{S^2}$.
We view $M$ as the product of unit spheres in $\R^3\times \R^3$ with coordinates $(x_1,y_1,z_1,x_2, y_2, z_2)$.
Finally, let $F_t = (J,H_t)$ where
\[
J = R_1 z_1 + R_2 z_2 \qquad H_t = (1-t)z_1 + t (x_1 x_2 + y_1 y_2 + z_1 z_2).
\]
Then $(M,\om_{R_1,R_2},F_t)$ is an integrable system and this family is known as the \emph{coupled angular momentum system}.
Let $p = (0,0,1, 0, 0, -1)$.
Given any fixed $0<R_1<R_2$ then there exist $t^-, t^+\in (0,1)$ such that
\begin{itemize}
    \item when $0\leq t <  t^-$, the system is \textbf{semitoric with no focus-focus points} and $p$ is singular of elliptic-elliptic type,
    \item when $t=t^-$, the system \textbf{has a single degenerate singular point}, which occurs at $p$, and all other singular points are of elliptic-regular or elliptic-elliptic type,
    \item when $t^-<t<t^+$, the system is \textbf{semitoric with exactly one focus-focus point}, which occurs at $p$,
    \item when $t=t^+$, the system \textbf{has a single degenerate singular point}, which occurs at $p$, and all other singular points are of elliptic-regular or elliptic-elliptic type,
    \item when $t^+<t<1$, the system is \textbf{semitoric with no focus-focus points} and $p$ is singular of elliptic-elliptic type.
\end{itemize}
In particular, as $t$ increases from 0 to 1, the point $p$ goes from being elliptic-elliptic, to degenerate, to focus-focus, to degenerate again, and eventually back to elliptic-elliptic.
During this time, the point $F_t(p)$ ``travels'' from the upper boundary to the lower boundary of the image $F_t(M)$, see Figure~\ref{fig:momentummapimage}.
Furthermore, the system is always semitoric, except for when $t=t^\pm$, at which time the point $p$ is degenerate. 
\end{example}

The coupled angular momenta system is an important example in physics, and in fact was originally introduced by physicists~\cite{SaZh1999}, since the associated quantum system exhibits monodromy.
The system is already interesting from the classical side though, since the integral affine structure on the base of the singular Lagrangian fibration induced by this system (see Section~\ref{sec:affine}) also exhibits monodromy\footnote{of course, it is no accident that both the quantum spectrum and the affine structure have the same property (monodromy), as these two structures are closely related. The monodromy in both settings comes from the existence of a focus-focus point.}. This example was the main motivation for forming a general theory of integrable systems with a fixed $S^1$-action and a second integral that varies with a parameter~\cite{LFPal,LFPal-SF2}, see Section~\ref{sec:st-families} for a discussion of semitoric families and their generalizations.
Also, in~\cite{HoPa2017} the authors describe a generalization of this system, see Section~\ref{sec:generalized-spins}.

\begin{remark}
    In general, the bifurcation that an integrable system undergoes when a singular point switches between being of focus-focus and elliptic-elliptic type is called a \emph{Hamiltonian-Hopf bifurcation}.
    In the context of almost toric fibrations, which are a generalization of semitoric systems due to Symington~\cite{Sy2003}, this is also known as a \emph{nodal trade}. We discuss almost toric fibrations briefly in Section~\ref{sec:ATFs}.
\end{remark}

While the definition of semitoric systems is broad enough that they represent a substantial generalization of toric systems, the definition is also restrictive enough that semitoric systems are relatively well-behaved. In particular, they admit a classification in terms of marked labeled polygons which generalizes the classification for four-dimensional toric systems in terms of polygons.
The plan for this section is to first describe the invariants of a semitoric system as an abstract object, and then describe how to construct these invariants from a given system. More specifically:
\begin{itemize}[nosep]
    \item In Section~\ref{sec:st-invariants} we describe the semitoric invariants abstractly (in both the simple and general cases);
    \item In Section~\ref{sec:st-construct} we give an overview of how the invariants are constructed from a given system;
    \item In Section~\ref{sec:st-classify} we state and discuss the Pelayo-V\~{u} Ng\d{o}c classification theorem~\cite{PeVN2009,PeVN2011} of simple semitoric systems, and it's generalization to all semitoric systems, simple or not~\cite{PPT-nonsimple};
    \item In Section~\ref{sec:generalized-spins} we explain the generalization of the coupled angular momenta system from~\cite{HoPa2017}.
\end{itemize}

\subsection{The semitoric invariants}
\label{sec:st-invariants}

In this section, we will describe the invariants of a semitoric system, independently of the system itself. In the next section we will discuss how these invariants are actually obtained from a given semitoric system.

We will package all of the invariants together, in what is called a marked labeled polygon. 
The invariants naturally break into two parts: the marked semitoric polygon, and the labels on the marked points of this polygon.

\subsubsection{Marked semitoric polygons}
\label{sec:marked-st-poly}

We will start with an imprecise description of marked semitoric polygons, then we will describe them formally, and finally we will describe the natural equivalence relation on these objects.

\paragraph{An informal description.}
Roughly speaking, a marked semitoric polygon is a rational convex polyhedron (if it is compact, this means that it is a rational convex polygon) with a finite number of marked points in the interior, and a ray called a cut coming from each marked point traveling either directly up or directly down. The points where the cuts meet the boundary are all vertices, and these vertices have to satisfy special conditions depending on the number of cuts hitting the boundary at that point (there can be more than one), while all other vertices of the polygon satisfy the same conditions as the Delzant case. See Figure~\ref{fig:marked-st-poly}.

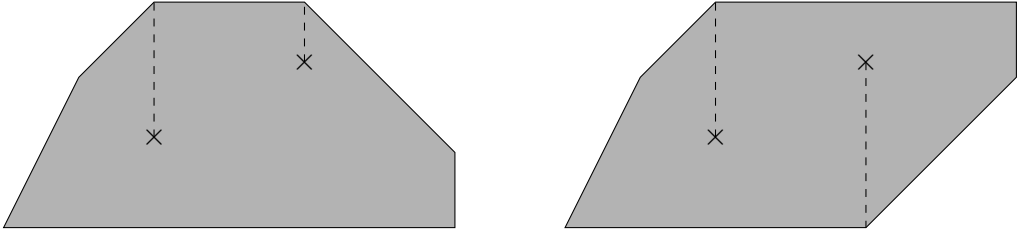
\begin{figure}
\begin{center}
\begin{tikzpicture}[scale=2.0]
\begin{scope}[xscale=-1]
\filldraw[draw=black, fill=gray!60] (0,0) node[anchor=north,color=black]{}
  -- (0,0.5) node[anchor=south,color=black]{}
  -- (1,1.5) node[anchor=south,color=black]{}
  -- (2,1.5) node[anchor=south,color=black]{}
  -- (2.5,1) node[anchor=north,color=black]{}
  -- (3,0) node[anchor=north,color=black]{}
  -- cycle;
\draw [dashed] (1,1.1) -- (1,1.5);
\draw (1,1.1) node[] {$\times$};
\draw [dashed] (2,0.6) -- (2,1.5);
\draw (2,0.6) node[] {$\times$};  
\begin{scope}[xshift = -220pt]
\filldraw[draw=black, fill=gray!60] (4,1) node[anchor=north,color=black]{}
  -- (4,1.5) node[anchor=south,color=black]{}
  -- (6,1.5) node[anchor=south,color=black]{}
  -- (6.5,1) node[anchor=north,color=black]{}
  -- (7,0) node[anchor=north,color=black]{}
  -- (5,0) node[anchor=south,color=black]{}
  -- cycle;
\draw [dashed] (5,1.1) -- (5,0);
\draw (5,1.1) node[] {$\times$};
\draw [dashed] (6,0.6) -- (6,1.5);
\draw (6,0.6) node[] {$\times$};  
\end{scope}
\end{scope}
\end{tikzpicture}
\end{center}
\caption{Two representatives of the same marked semitoric polygon, related by changing the cut direction on the second marked point. In this example, all vertices are either Delzant or fake (there are no hidden corners in this example).}
\label{fig:marked-st-poly}
\end{figure}

\paragraph{The formal definition.}
We will now make this description more precise. First of all, in keeping with the original definition of the invariant in~\cite{PeVN2009}, we will use the term \emph{marked semitoric polygon} to refer to the entire (infinite) equivalence class of polygons, so we first describe a single representative.
Recall that a convex rational polyhedron is a set in $\R^2$ which is locally the intersection of a finite number of closed half-planes whose boundaries admit normal vectors with integer entries. In the case that the set is compact, this definition is equivalent to the definition of a rational convex polygon.

Let
\begin{equation}\label{eqn:T}
    T = \begin{pmatrix}
        1 & 0 \\ 1 & 1
    \end{pmatrix}.
\end{equation}
\begin{definition}
    Let $\De$ be a convex rational polyhedron, and let $p$ be a vertex of $\De$. Let $v_1,v_2\in\Z^2$ be the primitive integral vectors directing the edges emanating from $p$, ordered so that $\det(v_1,v_2)>0$. Then we say that:
    \begin{enumerate}
        \item $p$ satisfies the \emph{Delzant condition} if $\det(v_1,v_2)=1$,
        \item $p$ satisfies the \emph{fake corner condition for $k$ cuts} if $\det(v_1,T^k v_2)=0$, and
        \item $p$ satisfies the \emph{hidden corner condition for $k$ cuts} if $\det(v_1,T^k v_2)=1$.
    \end{enumerate}
\end{definition}

\begin{remark}
The definition of this invariant in \cite[Definition 4.4]{PeVN2009} only includes the case that $k=1$, but here we also include the case of $k>1$ since this corresponds to non-simple semitoric systems (which are associated to semitoric polygons with more than one marked point in a single vertical line). Pelayo and V\~{u} Ng\d{o}c's classification of simple semitoric systems was extended to the non-simple case in~\cite{PPT-nonsimple}, and the original construction of a polygon from a semitoric system in~\cite{VN2007} already included the non-simple case.
The complete invariant in the non-simple case is significantly more complicated than the complete invariant in the simple case, but the underlying marked semitoric polygons are essentially the same in both cases, so we present a general definition here.
\end{remark}

For $i=1,2$, let $\proj_i\colon \R^2\to\R$ be the projection onto the $i^{\text{th}}$ coordinate,
and denote $\Z_2 := \{-1,1\}$.
Given $c\in\R^2$ and $\epsilon\in\Z_2$, we denote by $L_c^\epsilon$ the ray starting at $c$ and going up if $\epsilon=1$ and down if $\epsilon=-1$. That is,
$L_c^\epsilon = \{(x,y)\in\R^2 \mid x=\proj_1(c) \text{ and }\epsilon y\geq \epsilon \,\proj_2(c)\}$.

\begin{definition}\label{def:st-poly-repr}
    A \emph{semitoric polygon representative} is a triple $(\De,\vec{c},\vec{\epsilon}\,)$ where 
    \begin{enumerate}
        \item $\De\subset \R^2$ is a convex, rational polygon,
        \item $\vec{c} = (c_1,\ldots, c_m)$ satisfies $c_k\in\mathrm{int}(\De)$ for $k\in\{1,\ldots,m\}$, and the $c_i$ are in lexicographic\footnote{lexicographic order means that $(x_1,y_1)<(x_2,y_2)$ if and only if either $x_1<x_2$ or $x_1=x_2$ and $y_1<y_2$.} order,
        \item $\vec{\epsilon} = (\epsilon_1,\ldots,\epsilon_m)\in(\Z_2)^m$,
    \end{enumerate}
    and such that:
    \begin{enumerate}
        \item\label{item:finite-height} for each $x_0\in \proj_1(\De)$, the set $\De \cap \proj_1^{-1}(x_0)$ is compact.
        \item each point in the intersection of $\partial \De$ and $\cup_i L_{c_i}^{\epsilon_i}$ is a vertex of $\De$, and satisfies either the fake corner condition for $k$ cuts or the hidden corner condition for $k$ cuts, where $k$ is the number of $i$ such that the vertex is contained in $L_{c_i}^{\epsilon_i}$.
        \item all other vertices of $\De$ satisfy the Delzant condition.
    \end{enumerate}
    Furthermore, $(\De,\vec{c},\vec{\epsilon}\,)$ is said to be \emph{simple} if it satisfies one additional condition:
    \begin{itemize}
        \item[4.] There is at most one marked point in each vertical line $\proj_1^{-1}(x)$ for $x\in\R$. 
    \end{itemize}
\end{definition}

The $c_i$ are called the \emph{marked points} and the $L_{c_i}^{\epsilon_i}$ are called the \emph{cuts}. We represent the marked points by "$\times$" and the cuts by dotted lines.
Note that item~\eqref{item:finite-height} is automatic if $\De$ is compact.

The motivations behind the conditions on the corners are easier to understand after seeing how the marked polygon transforms. The idea is that there is an operation which changes the direction of the cuts, and the fake and hidden corners are designed so that by moving all cuts away from the given vertex the point which remains is either not a vertex anymore (for the fake vertices) or satisfies the Delzant condition (for the hidden vertices). We will describe the operations on these polygons now.

\paragraph{The equivalence relation.}
There are two natural equivalence relations on semitoric polygon representatives, motivated by the fact that when constructing the invariant from a semitoric system, as described in Section~\ref{sec:construct-poly}, there are certain choices involved which can produce different results. For the invariant of a given system to be well-defined, we consider such outputs to be equivalent.

There are two types of equivalence: a global skewing or translation of the polygon, and an operation which changes the direction of a cut (and adapts the polygon appropriately).
We will define a group action which encodes these.
Let $\mathcal{T}\subset \text{Aff}(\R^2)$ denote those integral affine transformations which preserve the $x$-component, which is the subgroup generated by $T$ and vertical translations of $\R^2$. For $j\in\R$ define $\mathfrak{t}_j\colon \R^2\to\R^2$ by 
\[
 \mathfrak{t}_j (x,y) = \begin{cases}
     (x,y+x-j) & \text{if }x\geq j,\\
     (x,y) & \text{otherwise.}
 \end{cases}
\]
Then $\mathfrak{t}_j$ is equivalent to applying the identity map to the set $x\leq j$ and $T$, relative to coordinates taking the origin anywhere on the line $x=j$, for $x\geq j$, see Figure~\ref{fig:marked-st-poly}.

Let $(\De,\vec{c},\vec{\epsilon}\,)$ be a semitoric polygon representative and let $m\geq 0$ denote the number of marked points.
Then $(\tau,\vec{\epsilon}\,')\in\mathcal{T}\times (\Z_2)^m$ acts on $(\De,\vec{c},\vec{\epsilon}\,)$ by
\begin{equation}\label{eqn:stpoly-action} (\tau,\vec{\epsilon}\,')\cdot(\De,\vec{c},\vec{\epsilon}\,) = \big( \sigma(\De), \sigma(\vec{c}), (\epsilon_1'\epsilon_1,\ldots, \epsilon_m'\epsilon_m)\big)
\end{equation}
where $\sigma = \tau\circ \mathfrak{t}^{u_1}_{\proj_1(c_1)}\circ\mathfrak{t}^{u_m}_{\proj_1(c_m)}$ and $u_k = \epsilon_k (1-\epsilon_k')/2$ for $k\in\{1,\ldots, m\}$.

Considering Equation~\eqref{eqn:stpoly-action}, the pair $(\tau,\vec{\epsilon}\,')$ acts by globally applying $\tau$ to the polygon and marked points, and by switching the $k^{\mathrm{th}}$ cut direction if and only if $\epsilon_k' = -1$. For each cut direction that is switched, the portion of the polygon to the right of that cut is skewed either upwards or downwards. 

\begin{definition}\label{def:marked-st-poly}
    A \emph{marked semitoric polygon} is the orbit of a semitoric polygon representative $(\De,\vec{c},\vec{\varepsilon}\,)$ (as in Definition~\ref{def:st-poly-repr}) under the action of $\mathcal{T}\times (\Z_2)^m$ given in Equation~\eqref{eqn:stpoly-action}. We denote the orbit by $[\De,\vec{c},\vec{\varepsilon}\,]$.
\end{definition}

In \cite[Lemma 4.2]{PeVN2011}, Pelayo and V\~{u} Ng\d{o}c showed that if $[\De,\vec{c},\vec{\varepsilon}\,]$ is a marked semitoric polygon then each element of $[\De,\vec{c},\vec{\varepsilon}\,]$ also satisfies the conditions given in Definition~\ref{def:st-poly-repr} to be a semitoric polygon representative. Furthermore, note that $[\De,\vec{c},\vec{\varepsilon}\,]$ is completely determined by any single representative, and therefore it is typical to provide a single polygon to represent the entire family.

\paragraph{Examples of marked semitoric polygons.}
First of all, any Delzant polygon satisfies the conditions to be a semitoric polygon representative with no markings (though the equivalence relation on marked semitoric polygon representatives identifies different Delzant polygons). We now explain a few more examples.

\begin{example}\label{ex:st-poly-triangle}
The polygon from Figure~\ref{fig:not-delzant}, which was an example of a polygon which is not Delzant, can be made to be a marked semitoric polygon representative by adding a marked point with an upwards cut underneath the top vertex, as in Figure~\ref{fig:st-poly-triangle}. Note that this is a fake corner: when changing the cut direction a straight line (with no vertex) is revealed.
\end{example}

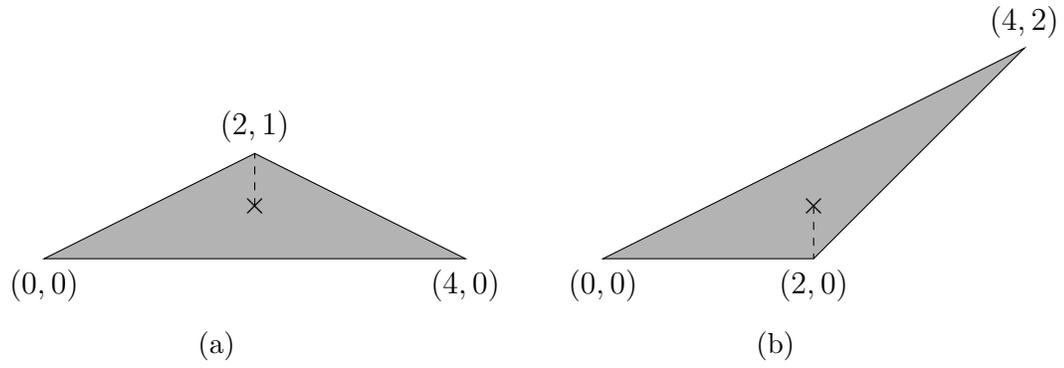
\begin{figure}
\begin{center}
\begin{subfigure}[b]{.35\linewidth}
\centering
\begin{tikzpicture}[scale=1.4]
\filldraw[draw=black, fill=gray!60] (0,0) 
  -- (2,1)
  -- (4,0)
  -- cycle;
\draw (2,0.5) node {$\times$}; 
\draw [dashed] (2,0.5) -- (2,1); 
\draw (0,0) node[below] {$(0,0)$};
\draw (4,0) node[below] {$(4,0)$};
\draw (2,1) node[above] {$(2,1)$};
\end{tikzpicture}
\caption{}
\label{fig:triangle-with-cut}
\end{subfigure}\qquad\qquad
\begin{subfigure}[b]{.35\linewidth}
\centering
\begin{tikzpicture}[scale=1.4]
\filldraw[draw=black, fill=gray!60] (0,0) 
  -- (4,2)
  -- (2,0)
  -- cycle;
\draw (2,0.5) node {$\times$}; 
\draw [dashed] (2,0.5) -- (2,0); 
\draw (0,0) node[below] {$(0,0)$};
\draw (4,2) node[above] {$(4,2)$};
\draw (2,0) node[below] {$(2,0)$};
\end{tikzpicture}
\caption{}
  \end{subfigure}
\caption{Two representatives of the same marked semitoric polygon, related by a change in the cut direction. Both representatives include a single fake corner and two Delzant corners.}
\label{fig:st-poly-triangle}
\end{center}
\end{figure}

\begin{example}
    For an example with a hidden corner, consider the marked semitoric polygon representatives shown in Figure~\ref{fig:st-poly-hidden}.
    The representative on the left includes a hidden corner (since moving the cut reveals a Delzant corner) and two Delzant corners, while the representative on the right includes three Delzant corners and one fake corner.
\end{example}

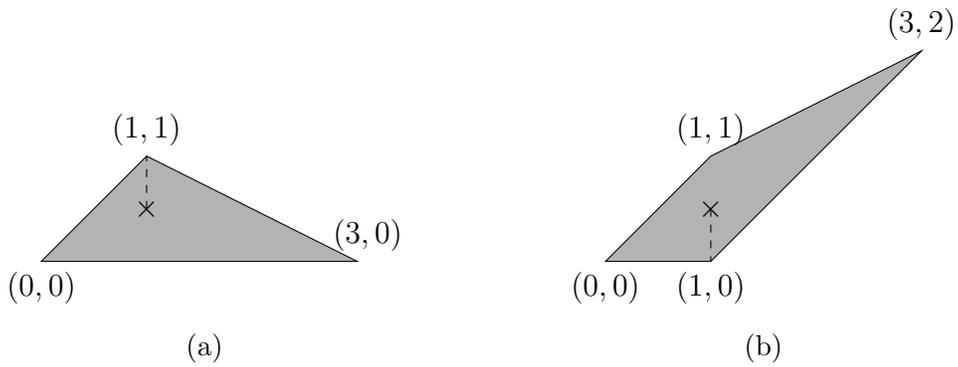
\begin{figure}
\begin{center}
\begin{subfigure}[b]{.35\linewidth}
\centering
\begin{tikzpicture}[scale=1.4]
\filldraw[draw=black, fill=gray!60] (0,0) 
  -- (1,1)
  -- (3,0)
  -- cycle;
\draw (1,0.5) node {$\times$}; 
\draw [dashed] (1,0.5) -- (1,1); 
\draw (0,0) node[below] {$(0,0)$};
\draw (1,1) node[above] {$(1,1)$};
\draw (3.1,0) node[above] {$(3,0)$};
\end{tikzpicture}
\caption{}
\end{subfigure}\qquad\qquad
\begin{subfigure}[b]{.35\linewidth}
\centering
\begin{tikzpicture}[scale=1.4]
\filldraw[draw=black, fill=gray!60] (0,0) 
  -- (1,1)
  -- (3,2)
  -- (1,0)
  -- cycle;
\draw (1,0.5) node {$\times$}; 
\draw [dashed] (1,0.5) -- (1,0); 
\draw (0,0) node[below] {$(0,0)$};
\draw (1,1) node[above] {$(1,1)$};
\draw (3,2) node[above] {$(3,2)$};
\draw (1,0) node[below] {$(1,0)$};
\end{tikzpicture}
\caption{}
  \end{subfigure}
\caption{Two representatives of the same marked semitoric polygon, related by a change in the cut direction.}
\label{fig:st-poly-hidden}
\end{center}
\end{figure}

\begin{example}
    For an example of a marked semitoric polygon with two marked points consider the example shown in Figure~\ref{fig:st-poly-2FF}. 
    One representative of this system is the marked polygon with vertices $(0,0),(1,1),(2,1),(3,0)$, marked points $c_1,c_2\in\mathrm{int}(\De)$ with $\text{proj}_1(c_1)=1$ and $\text{proj}_2(c_2)=2$. While the $y$-components of the $c_i$ are also part of the invariant, we have not indicated them here. Since there are two marked points, there are two cuts which can be changed between up and down. The four representatives shown in the figure are not only related by cuts, though. For instance, to get from the top left figure to the bottom left figure, the left cut is changed from up to down and then the entire polygon is acted on by $T^{-1}$. The situation is similar for getting from the top right representative to the bottom left one.
\end{example}

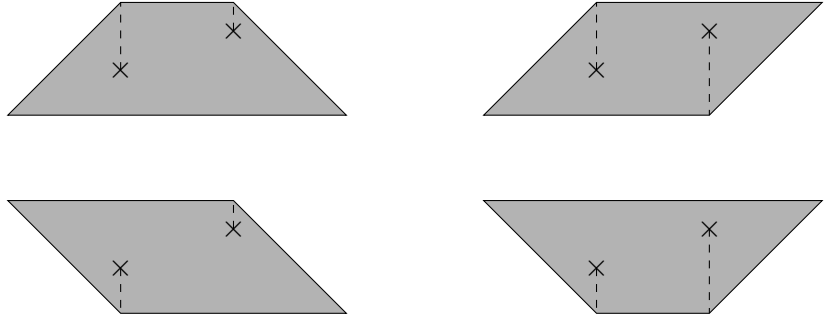
\begin{figure}
\begin{center}
\begin{tikzpicture}[scale=1.5]
\filldraw[draw=black, fill=gray!60] (0,0) 
  -- (1,1)
  -- (2,1)
  -- (3,0)
  -- cycle;
\draw (1,0.4) node {$\times$}; 
\draw [dashed] (1,0.4) -- (1,1); 
\draw (2,0.75) node {$\times$}; 
\draw [dashed] (2,0.75) -- (2,1); 
\begin{scope}[xshift = 120pt]
\filldraw[draw=black, fill=gray!60] (0,0) 
  -- (1,1)
  -- (3,1)
  -- (2,0)
  -- cycle;
\draw (1,0.4) node {$\times$}; 
\draw [dashed] (1,0.4) -- (1,1); 
\draw (2,0.75) node {$\times$}; 
\draw [dashed] (2,0.75) -- (2,0); 
\end{scope}
\begin{scope}[yshift = -50pt]
\filldraw[draw=black, fill=gray!60] (0,1) 
  -- (2,1)
  -- (3,0)
  -- (1,0)
  -- cycle;
\draw (1,0.4) node {$\times$}; 
\draw [dashed] (1,0.4) -- (1,0); 
\draw (2,0.75) node {$\times$}; 
\draw [dashed] (2,0.75) -- (2,1); 
\end{scope}
\begin{scope}[xshift = 120pt, yshift=-50pt]
\filldraw[draw=black, fill=gray!60] (0,1) 
  -- (3,1)
  -- (2,0)
  -- (1,0)
  -- cycle;
\draw (1,0.4) node {$\times$}; 
\draw [dashed] (1,0.4) -- (1,0); 
\draw (2,0.75) node {$\times$}; 
\draw [dashed] (2,0.75) -- (2,0); 
\end{scope}
\end{tikzpicture}
\caption{Four representatives of the same semitoric polygon, which has two focus-focus points. The labels on the vertices are not shown to simplify the figure, but each edge of each polygon has slope either zero, one, or negative one.}
\label{fig:st-poly-2FF}
\end{center}
\end{figure}

\begin{example}
In Figure~\ref{fig:st-poly-non-simple} we show an example of a marked semitoric polygon with two marked points in the same vertical line. In this situation the cuts can overlap and make the figure more complicated. In the top left figure, for instance, there are two cuts incident on the top boundary of the polygon, and this is why the slope decreases by two here (instead of decreasing by one, like when there is one cut). In the top right representative, the two cuts face towards each other and go through the marked points. It is also possible to have the two marked points completely coincide, but that is not what occurs in this example since they are at different heights.
\end{example}

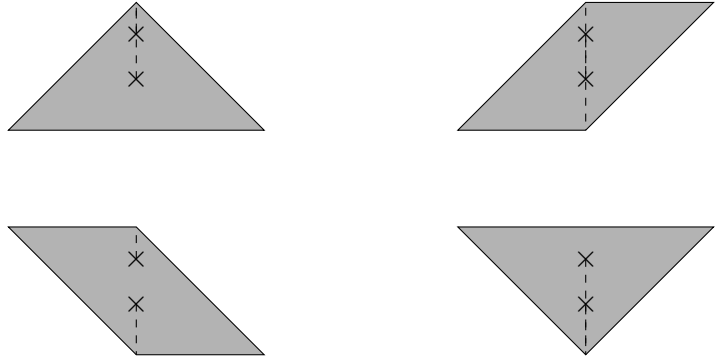
\begin{figure}
\begin{center}
\begin{tikzpicture}[scale=1.7]
\filldraw[draw=black, fill=gray!60] (0,0) 
  -- (1,1)
  -- (2,0)
  -- cycle;
\draw (1,0.4) node {$\times$}; 
\draw [dashed] (1,0.4) -- (1,1); 
\draw (1,0.75) node {$\times$}; 
\draw [dashed] (1,0.75) -- (1,1); 
\begin{scope}[xshift = 100pt]
\filldraw[draw=black, fill=gray!60] (0,0) 
  -- (1,1)
  -- (2,1)
  -- (1,0)
  -- cycle;
\draw (1,0.4) node {$\times$}; 
\draw [dashed] (1,0.4) -- (1,1); 
\draw (1,0.75) node {$\times$}; 
\draw [dashed] (1,0.75) -- (1,0); 
\end{scope}
\begin{scope}[yshift = -50pt]
\filldraw[draw=black, fill=gray!60] (0,1) 
  -- (1,1)
  -- (2,0)
  -- (1,0)
  -- cycle;
\draw (1,0.4) node {$\times$}; 
\draw [dashed] (1,0.4) -- (1,0); 
\draw (1,0.75) node {$\times$}; 
\draw [dashed] (1,0.75) -- (1,1); 
\end{scope}
\begin{scope}[xshift = 100pt, yshift=-50pt]
\filldraw[draw=black, fill=gray!60] (0,1) 
  -- (2,1)
  -- (1,0)
  -- cycle;
\draw (1,0.4) node {$\times$}; 
\draw [dashed] (1,0.4) -- (1,0); 
\draw (1,0.75) node {$\times$}; 
\draw [dashed] (1,0.75) -- (1,0); 
\end{scope}
\end{tikzpicture}
\caption{Four representatives of the same semitoric polygon, which has two marked points which occur in the same vertical line. Again, the vertex labels have been omitted, but each edge of each polygon has slope either zero, one, or negative one.}
\label{fig:st-poly-non-simple}
\end{center}
\end{figure}

The paper~\cite{packingIGL} contains a concise and accessible treatment of marked semitoric polygons (and Delzant polygons in dimension four), and furthermore includes many examples of marked semitoric polygons.

\subsubsection{The labels: Taylor series type invariants}
\label{sec:invariants-labels}

The marked semitoric polygon is an important invariant of semitoric systems, but it is not a complete invariant. Each marked point on the polygon must also be labeled with extra information, coming from the semi-local invariants (i.e.~invariants of a neighborhood of the fiber) of focus-focus points in integrable systems introduced in~\cite{VN2007}. If two or more of the marked points are equal, then the label invariants are more complicated, as described in~\cite{PT}, and in that case the invariants of all marked points which are equal interact and must be treated together as a single object.
We will describe the labels in the single-pinched case and the multi-pinched case separately.

\paragraph{The labels in the simple case.}
In fact, the description in this section applies to the labels for any marked polygon for which all of the marked points $c_k$ are distinct, which is slightly more general than the simpleness condition given in Definition~\ref{def:st-poly-repr} (which does not allow more than one marked point to lie in the same vertical line).

In the case that all marked points are distinct, each marked point corresponds to a single focus-focus point of the semitoric system, which lies in a once-pinched torus as in Figure~\ref{fig:pinchedtorus}. The invariants of such a fiber were given by V\~{u} Ng\d{o}c~\cite{VN2003}. We give an overview of how this invariant is obtained from the system in Section~\ref{sec:construct-labels}, and in this section we simply describe the invariant itself.

\begin{figure}
    \centering
    \includegraphics[width=0.3\linewidth]{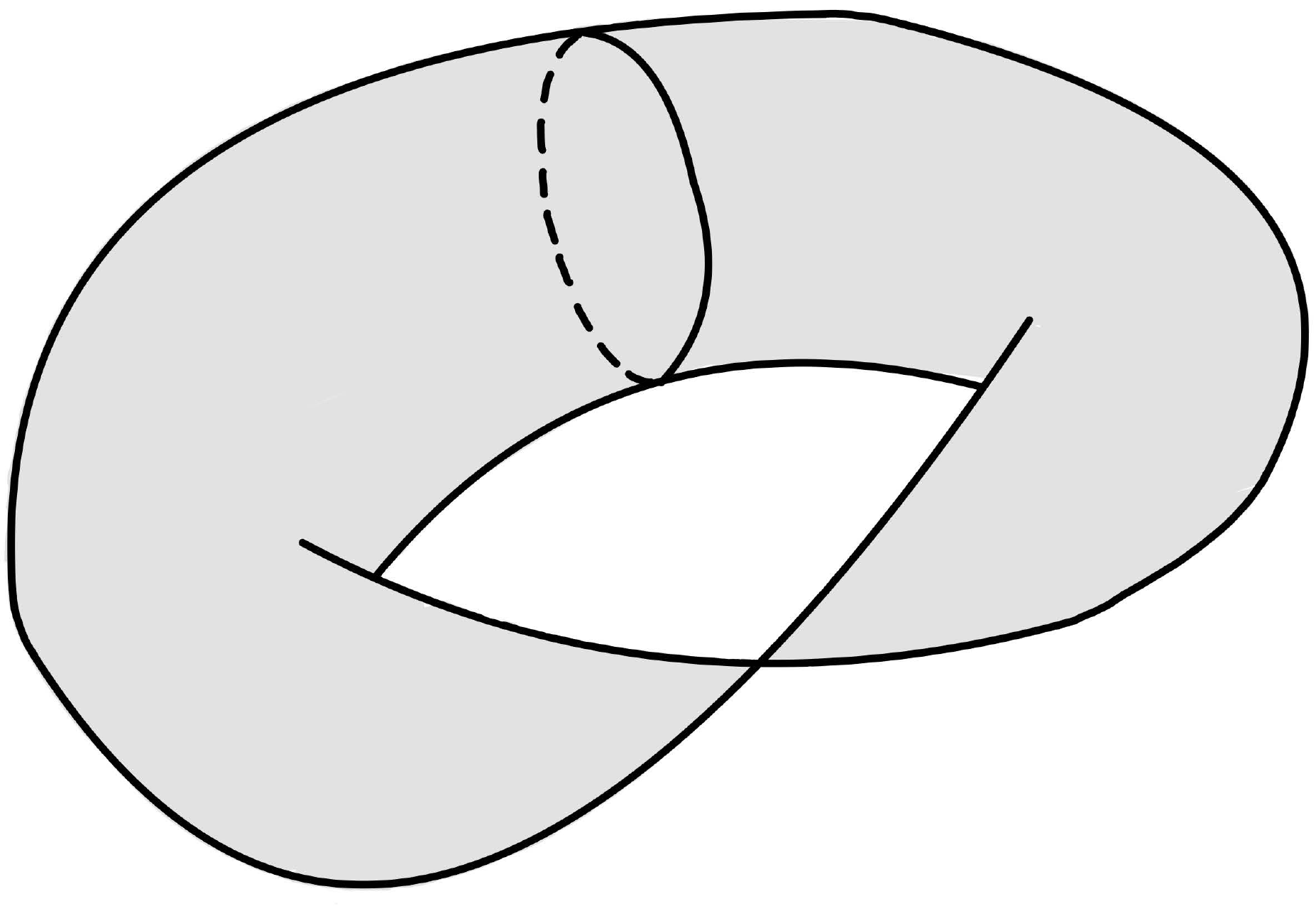}
    \caption{A compact fiber containing a single focus-focus point: a pinched torus.}
    \label{fig:pinchedtorus}
\end{figure}

Let $\R[[X,Y]]$ denote the set of all power series in $X$ and $Y$ with real coefficients, and let $\R_0[[X,Y]]\subset \R[[X,Y]]$ denote the subset of series which have zero constant term. Then, the \emph{Taylor series invariant} is a collection of Taylor series 
\[
\vec{S}^\infty = (S^\infty_1,\ldots, S^\infty_m) \in \Big(\R_0[[X,Y]]/(2\pi X \Z)\Big)^m
\]
and for each $k\in\{1,\ldots, m\}$ we say that $S^\infty_k$ is the Taylor series invariant of $c_k$.
Note that the labels can be any possible Taylor series with zero constant term, and note that the coefficient of $X$ in each Taylor series is only determined up to integer multiples of $2\pi$. 

These Taylor series completely determine the structure around the fiber of the focus-focus point, but there is an additional invariant known as the \emph{twisting index invariant} which encodes how the structure of each focus-focus fiber lies with respect to the rest of the integrable system. This is encoded as an integer for each marked point. The added complication with this invariant is that it depends on the choice of semitoric polygon, so it is more accurate to say that the twisting index is an assignment of a tuple of integers \[\vec{\kappa} = (\kappa_1,\ldots,\kappa_m)\in\Z^m,\]
one for each marked point, to each semitoric polygon representative $(\De,\vec{c},
\vec{\epsilon}\,)$. Thus, each marked point $c_k$ is labeled with a Taylor series $S^\infty_k$ (only defined up to $2\pi X\Z$) and an integer $\kappa_k$.
We obtain a collection called a \emph{marked labeled semitoric polygon representative},
\[ \Big(\big(\De,\vec{c},\vec{\epsilon}\,\big), \vec{S}^\infty, \vec{\kappa} \Big).
\]

As with the marked polygon, the values of the twisting index on a single representative determine the values on all representatives, since their transformation under the group action is prescribed by
\begin{equation}\label{eqn:action-k}
    (\tau,\vec{\epsilon}\,')\cdot \vec{\kappa} = \left( \kappa_k + n + \sum_{i=1}^k u_i\right)_{k=1}^m,
\end{equation}
where again $u_k = \epsilon_k (1-\epsilon_k')/2$ and here $n$ is the integer such that $\tau$ is a translation composed with $T^n$. On the other hand, the action of $\mathcal{T}\times (\Z_2)^m$ does not change the Taylor series invariants at all, since these are the invariants of the focus-focus points themselves and have nothing to do with the choice of polygon. 

We obtain that $\mathcal{T}\times (\Z_2)^m$ acts on a marked labeled semitoric polygon representative by
\begin{equation}\label{eqn:complete-group-action}
    (\tau,\vec{\epsilon}\,')\cdot \Big(\big(\De,\vec{c},\vec{\epsilon}\,\big), \vec{S}^\infty, \vec{\kappa} \Big) = \Big( (\tau,\vec{\epsilon}\,')\cdot\big(\De,\vec{c},\vec{\epsilon}\,\big), \vec{S}^\infty, (\tau,\vec{\epsilon}\,')\cdot \vec{\kappa} \Big).
\end{equation}
Here the action on $(\De,\vec{c},\vec{\epsilon}\,')$ is as given in Equation~\eqref{eqn:stpoly-action} and the action on $\vec{\kappa}$ is given by Equation~\eqref{eqn:action-k}.
We denote the orbit of $\big((\De,\vec{c},\vec{\epsilon}\,), \vec{S}^\infty, \vec{\kappa} \big)$ under this group action by $\big[(\De,\vec{c},\vec{\epsilon}\,), \vec{S}^\infty, \vec{\kappa} \big]$.

\begin{definition}\label{def:Y}
Let $\mathbf{Y}$ be the set of all orbits $\big[(\De,\vec{c},\vec{\epsilon}\,), \vec{S}^\infty, \vec{\kappa} \big]$ such that
\begin{enumerate}
    \item $[\De,\vec{c},\vec{\epsilon}\,]$ is a marked semitoric polygon (as in Definition~\ref{def:marked-st-poly}),
    \item Each $c_1,\ldots,c_m$ is distinct,
    \item $S^\infty_1,\dots,S^\infty_m\in\R_0[[X,Y]]/(2\pi X \Z)$
    \item $\kappa_1,\ldots,\kappa_m\in\Z$.
\end{enumerate}
We call $\mathbf{Y}$ the set of \emph{semitoric ingredients}. 
\end{definition}

Note that $\mathbf{Y}$ is slightly larger than the set of all sets of invariants satisfying the simplicity assumption, since a marked polygon portion of an element of $\mathbf{Y}$ can have multiple marked points in the same vertical line, as long as the points do not coincide.
We will see that there is a natural bijection from semitoric systems with at most one focus-focus point per fiber of $F$ and $\mathbf{Y}$ (Theorem~\ref{thm:PVN}), so we say that such systems are classified by $\mathbf{Y}$.

\paragraph{Packaging the Taylor series and twisting index together.}
In the case that all marked points are distinct, we have now seen that each marked point is labeled by an equivalence class of a Taylor series in $\R_0[[X,Y]]/(2\pi X)$ and an integer.

For any equivalence class $S^\infty \in \R[[X,Y]]/(2\pi X)$ let $\hat{S}^\infty\in\R[[X,Y]]$ be the representative of $S^\infty$ for which the coefficient of $X$ is taken to be in $[0,2\pi)$. Now, considering the labels $S^\infty_k$ and $\kappa_k$ of a marked point $c_k$ in a marked labeled polygon, we can encode these two pieces of information together as
\[
\widetilde{S}_k^\infty := \hat{S}^\infty_k + 2\pi \kappa_k X \in\R_0[[X,Y]].
\]
Note that both $S^\infty_k$ and $\kappa_k$ can be recovered from $\widetilde{S}_k^\infty$. Also, notice that while $S^\infty_k$ is not affected by the action of the group $\mathcal{T}\times (\Z_2)^m$ as in Equation~\eqref{eqn:complete-group-action}, the impact of that group action on $\kappa_k$ means that the group acts on $\widetilde{S}_k^\infty$ by
\begin{equation}\label{eqn:action-tilde-S}
 (\tau,\vec{\epsilon}\,')\cdot \widetilde{S}^\infty_k = \widetilde{S}^\infty_k + 2\pi X \left(n + \sum_{i=1}^k u_i \right),    
\end{equation}
where $n$ is the integer such that $\tau$ is a translation composed with $T^n$ and $u_k = \epsilon_k (1-\epsilon_k')/2$. 
For ease of notation, we omit the vector symbol and simply let $\widetilde{S}^\infty = (\widetilde{S}^\infty_1,\ldots,\widetilde{S}^\infty_m)$.
With Equation~\eqref{eqn:action-tilde-S}, we can define an action of $\mathcal{T}\times (\Z_2)^m$ on the set of semitoric polygon representatives with $m$ distinct marked points and one Taylor series labeling each marked point by 
\begin{equation}\label{eqn:complete-group-action-tilde}
    (\tau,\vec{\epsilon}\,')\cdot \Big(\big(\De,\vec{c},\vec{\epsilon}\,\big), \widetilde{S}^\infty \Big) = \Big( (\tau,\vec{\epsilon}\,')\cdot\big(\De,\vec{c},\vec{\epsilon}\,\big), (\tau,\vec{\epsilon}\,')\cdot \widetilde{S}^\infty \Big).
\end{equation}
Similarly to above, we denote the orbit of $\big((\De,\vec{c},\vec{\epsilon}\,), \widetilde{S}^\infty \big)$ under this action by $
\big[(\De,\vec{c},\vec{\epsilon}\,), \widetilde{S}^\infty \big]
$. This produces an invariant of semitoric systems which is clearly equivalent to the one appearing in Definition~\ref{def:Y}.

\begin{definition}\label{def:Ytilde}
Let $\widetilde{\mathbf{Y}}$ be the set of all orbits $\big[(\De,\vec{c},\vec{\epsilon}\,), \widetilde{S}^\infty \big]$ such that
\begin{enumerate}
    \item $[\De,\vec{c},\vec{\epsilon}\,]$ is a marked semitoric polygon,
    \item Each $c_1,\ldots,c_m$ is distinct,
    \item $S^\infty_1,\dots,S^\infty_m\in\R_0[[X,Y]].$
\end{enumerate}
\end{definition}

We emphasize that $S_k^\infty$ lives in the quotient $\R_0[[X,Y]] / (2\pi X \Z)$, while $\widetilde{S}_k^\infty$ is a representative in $\R_0[[X,Y]]$, adjusted by the twisting index.
There is a natural bijection from $\mathbf{Y}$ to $\widetilde{\mathbf{Y}}$ given by \[\big[(\De,\vec{c},\vec{\epsilon}\,), \vec{S}^\infty, \vec{\kappa} \big]\mapsto \big[(\De,\vec{c},\vec{\epsilon}\,), \widetilde{S}^\infty \big]\]
where $\widetilde{S}^\infty_k$ is formed by taking the representative of $S^\infty_k$ whose $X$-coefficient lies in $[0,2\pi)$ and adding $2\pi \kappa_k X$.
Soon we will see that $\mathbf{Y}$ classifies semitoric systems, and this bijection means that we can equally well use $\widetilde{\mathbf{Y}}$ instead of $\mathbf{Y}$ for the classification.

\begin{remark}
 The classification of simple semitoric systems was originally stated~\cite{PeVN2009,PeVN2011} in terms of \emph{five symplectic invariants}: (1) the number of focus-focus points invariant, (2) the (unmarked) semitoric polygon invariant, (3) the height invariant, (4) the Taylor series invariant, and the (5) twisting index invariant.
 We have combined the first three of these invariants into the object we have called the marked semitoric polygon. For $\mathbf{Y}$ we have left (4) and (5) separate, but in $\widetilde{\mathbf{Y}}$ we have combined them.
\end{remark}

\paragraph{The labels in the non-simple case.} In the case that some of the marked points in a semitoric polygon coincide, which corresponds to a semitoric system which has more than one focus-focus point in a single fiber of the momentum map, the labels corresponding to those two marked points interact and are replaced by a more complicated object.
This semilocal invariant of fibers containing multiple focus-focus points was introduced in~\cite{PT} and used in the extension of the classification of simple semitoric systems to all semitoric systems, simple or not~\cite{PPT-nonsimple}.
While focus-focus fibers in simple semitoric systems are always once-pinched tori, as in Figure~\ref{fig:pinchedtorus}, in non-simple systems it is possible to have many focus-focus points in the same fiber, leading to tori with multiple pinches, as in Figure~\ref{fig:multipinched}.

\begin{figure}
    \centering
    \includegraphics[width=0.25\linewidth, angle = 100]{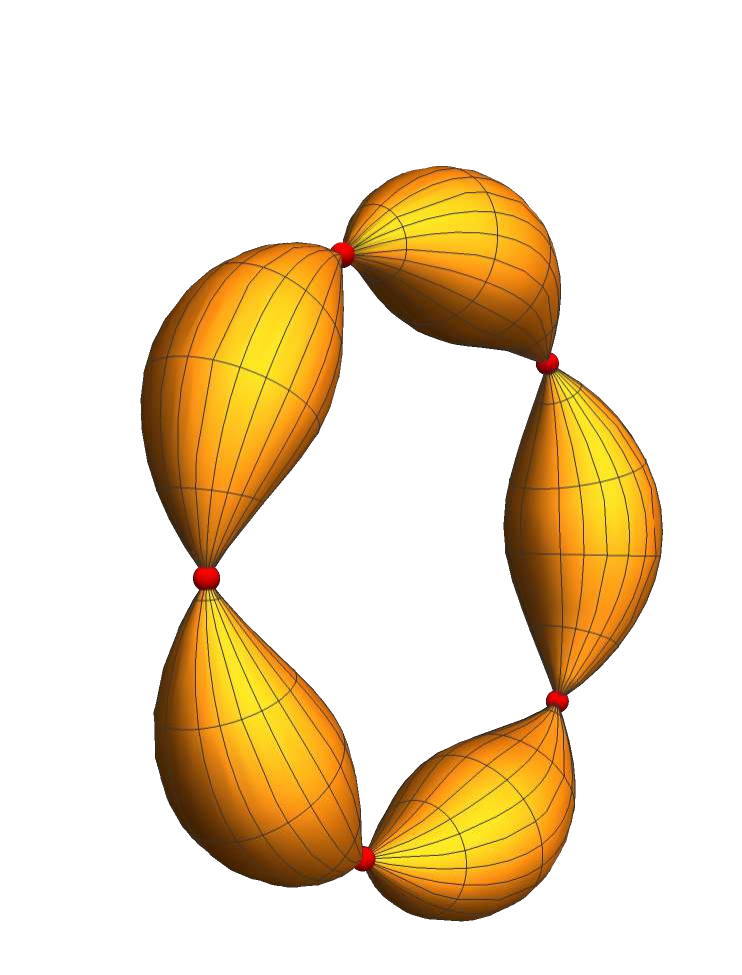}
    \caption{A fiber which includes five focus-focus points, and therefore is topologically a torus pinched five times.}
    \label{fig:multipinched}
\end{figure}

Here we will explain how to construct a single object which includes the information of both the Taylor series and twisting index type invariants for the non-simple case, since in this scenario the twisting index naturally interacts with the Taylor series.
We will use the notation of~\cite{PT} and~\cite{PPT-nonsimple}.
Let $n\in\Z_{>0}$, which will represent the number of focus-focus points in the fiber in question, and let $\Z_n = \Z/(n\Z)$. As above, we use $\R[[X,Y]]$ to denote Taylor series in $X,Y$, and now let $\R_+[[X,Y]]$ denote those Taylor series with no constant term and for which the coefficient of the linear $Y$ term is positive. Then we consider objects of the form $(\tilde{\mathsf{s}}_\mu,\mathsf{g}_{\mu,\nu})_{\mu,\nu\in\Z_n}$ such that: 
\begin{equation}\label{eqn:non-simple-cond}
    \begin{cases}
    \tilde{\mathsf{s}}_\mu\in\R[[X,Y]],\, \mathsf{g}_{\mu,\nu}\in\R_+[[X,Y]]\\[2pt] \tilde{\mathsf{s}}_\mu(X,Y)=\tilde{\mathsf{s}}_\nu(X,\mathsf{g}_{\mu,\nu}(X,Y))\\[2pt]
    \mathsf{g}_{\mu,\mu}(X,Y)=Y\\[2pt]
    \mathsf{g}_{\mu,\sigma}(X,Y) = \mathsf{g}_{\nu,\sigma}(X,\mathsf{g}_{\mu,\nu}(X,Y)).
    \end{cases}
\end{equation}
We include the tilde here on the $\tilde{\mathsf{s}}_\mu$ to indicate that, as above, the information of the twisting index is also included. 

Due to the relations in Equation~\eqref{eqn:non-simple-cond}, these Taylor series are over determined. For instance, \[\tilde{\mathsf{s}}_0, \mathsf{g}_{0,1},\mathsf{g}_{1,2},\ldots, \mathsf{g}_{n-2,n-1}\] determine the entire list. That is, given any choice of $\tilde{\mathsf{s}}_0\in\R[[X,Y]]$ and $\mathsf{g}_{\mu,\mu+1}\in\R_+[[X,Y]]$ for $\mu\in\Z_n\setminus\{n-1\}$, there exists exactly one way to extend to a collection $(\tilde{\mathsf{s}}_\mu,\mathsf{g}_{\mu,\nu})_{\mu,\nu\in\Z_n}$ 
which satisfies Equation~\eqref{eqn:non-simple-cond}.
Note that in the case that $n=1$ this means that this entire object reduces to the choice of $\tilde{\mathsf{s}}_0$, a single Taylor series. Removing the constant term from that Taylor series we obtain the information of the Taylor series $S^\infty$ from the simple case. We include the constant term here to agree with the conventions in~\cite{PT,PPT-nonsimple}, but it is not actually extra information once we apply these labels to the marked semitoric polygon, since we will end up requiring that the constant term agree with the height of the marked point.

Now we must describe the complete invariant in the non-simple case. In~\cite{PPT-nonsimple} the authors give two equivalent classifications, and here we will explain the invariant described in Section 5 of that paper (see~\cite[Theorem 5.3]{PPT-nonsimple}), since it more closely resembles our presentation of the classification for simple systems.

Let $(\De,\vec{c},\vec{\epsilon}\,)$ be a semitoric polygon representative, as in Definition~\ref{def:st-poly-repr}. In this case, we allow any number of the $c_i$ to coincide. Now let:
\begin{itemize}
    \item $m$ be the number of marked points,
    \item $v$ be the number of distinct marked points, i.e.~$v=|\{c_i\}_{i=1}^m|$,
    \item let $\{c_j'\}_{j=1}^v$ be the set of distinct values, so $\{c_j'\}_{j=1}^v=\{c_i\}_{i=1}^m$, the $c'_j$ are distinct, and we assume that the $c_j'$ are also in lexicographic order,
    \item for each $j$, let $m_j$ denote the size of $\{c_i \mid c_i = c_j'\}$. In other words, $m_j$ is the number of marked points which coincide at the $j^{\text{th}}$ value.
\end{itemize}
Note that $m = \sum_j m_j$.
Now, for each $j$, let $(\tilde{\mathsf{s}}_\mu^j,\mathsf{g}^j_{\mu,\nu})_{\mu,\nu\in\Z_{m_j}}$ be a tuple satisfying~\eqref{eqn:non-simple-cond}, and denote by $[\tilde{\mathsf{s}}_\mu^j,\mathsf{g}^j_{\mu,\nu}]_{\mu,\nu\in\Z_{m_j}}$ the orbit of $(\tilde{\mathsf{s}}_\mu^j,\mathsf{g}^j_{\mu,\nu})_{\mu,\nu\in\Z_{m_j}}$ under cyclic permutation of the indices. Then we call
\[
\Big((\De,\vec{c},\vec{\epsilon}\,),\big([\tilde{\mathsf{s}}_\mu^j,\mathsf{g}^j_{\mu,\nu}]_{\mu,\nu\in\Z_{m_j}}\big)_{j=1}^v\Big)
\]
a \emph{generalized marked polygon representative}. As before, we will need to take a quotient by a group action to obtain a well-defined invariant.
The group $\mathcal{T}\times (\Z_2)^m$ acts on generalized marked polygon representatives by
\begin{equation}\label{eqn:non-simple-group}
(\tau,\vec{\epsilon}\,')\cdot\Big((\De,\vec{c},\vec{\epsilon}\,),\big([\tilde{\mathsf{s}}_\mu^j,\mathsf{g}^j_{\mu,\nu}]_{\mu,\nu\in\Z_{m_j}}\big)_{j=1}^v\Big)=\Big((\tau,\vec{\epsilon}\,')\cdot(\De,\vec{c},\vec{\epsilon}\,),\big([(\tau,\vec{\epsilon}\,')\cdot \tilde{\mathsf{s}}_\mu^j,\mathsf{g}^j_{\mu,\nu}]_{\mu,\nu\in\Z_{m_j}}\big)_{j=1}^v\Big)
\end{equation}
where the action on the semitoric polygon representative is as in Equation~\eqref{eqn:stpoly-action} and the action of $(\tau,\vec{\epsilon}\,')$ on the Taylor series $\tilde{\mathsf{s}}_\mu^j$ is as follows:
let $a_j = \pi_1(c_j')$, so $a_j$ is the $x$-coordinate of the $j^{\text{th}}$ focus-focus value. Let $S_{\vec{\epsilon}\,'}^{j_0}=\{j \mid \varepsilon'_{j}=-1 \text{ and }a_j \leq a_{j_0}\}$. Then the action on the $j_0^{\text{th}}$ tuple of Taylor series is given by
\begin{align}\label{eqn:non-simple-ts-action}
( T^k+(0,b),\vec{\epsilon}\,')\cdot (\tilde{\mathsf{s}}_\mu^{j_0}) =
\tilde{\mathsf{s}}_\mu^{j_0} + (2\pi X + 2\pi a_{j_0}) k + 2\pi b + \sum_{j\in S_{\vec{\epsilon}\,'}}\left(2\pi X + 2\pi(a_j-a_{j_0})\right).
\end{align}
Here we use $T^k+(0,b)$ to denote the element of $\mathcal{T}$ which acts by $T^k$ and then translates the $y$-component by $b\in \R$. 

\begin{remark}
Note that the group action in Equation~\eqref{eqn:non-simple-ts-action} only changes the constant term and the coefficient of $X$ in the series $\widetilde{\mathsf{s}}_\mu^{j_0}$, and it only changes the coefficient of $X$ by multiples of $2\pi$. Thus, removing the constant term and taking the quotient by $2\pi X \Z$ would yield an object which is not effected by this group action.
\end{remark}

\begin{definition}\label{def:Z}
Let $\widetilde{\bf{Z}}$ be the set of all orbits $\Big[(\De,\vec{c},\vec{\epsilon}\,),\big([\tilde{\mathsf{s}}_\mu^j,\mathsf{g}^j_{\mu,\nu}]_{\mu,\nu\in\Z_{m_j}}\big)_{j=1}^v\Big]$ under the action of $\mathcal{T}\times (\Z_2)^m$ from Equation~\eqref{eqn:non-simple-group} such that
\begin{enumerate}
    \item $[\De,\vec{c},\vec{\epsilon}\,]$ is a marked semitoric polygon (see Definition~\ref{def:marked-st-poly}) where $m$, $v$, $m_j$, and $c_j
    $ are as defined above,
    \item $(\tilde{\mathsf{s}}_\mu^j,\mathsf{g}^j_{\mu,\nu})_{\mu,\nu\in\Z_{m_j}}$ satisfies the conditions of Equation~\eqref{eqn:non-simple-cond} for each $j=1,\ldots, v$.
    \item for each $j$, the $y$-value of $c_j'$ is equal to the constant term of $\tilde{\mathsf{s}}_\mu^j$ for all $\mu\in\Z_{m_j}$.
\end{enumerate}
We call $\tilde{\bf{Z}}$ the \emph{set of generalized semitoric ingredients}.
\end{definition}

Note that all of the marked points are distinct if and only if $m_j=1$ for all $j$, in which case $v=m$. As discussed above, if $m_j=1$ then the tuple of Taylor series reduces to the data of a single Taylor series, and in this case all of this information becomes equivalent to an element of $\bf{Y}$ (or equivalently, $\widetilde{\bf{Y}}$). That is, there is a natural injection $\bf{Y}\hookrightarrow \widetilde{\bf{Z}}$.

\begin{remark}
    In $\Tilde{\bf{Z}}$ we package the information of the twisting index in with the Taylor series, similar to how we did it above for $\Tilde{\bf{Y}}$ (which is why $\Tilde{\bf{Z}}$ also has a tilde). It is also possible to extract this information as a collection of integers for each focus-focus fiber, more similar to $\bf{Y}$, as described in~\cite[Proposition 6.2]{PPT-nonsimple}, but in the non-simple case it is much more natural to simply package these invariants together, so that is what we have done.
\end{remark}

\paragraph{Conclusion of Section~\ref{sec:st-invariants}}
In this section we have described three sets:
\begin{itemize}
    \item $\bf{Y}$, from Definition~\ref{def:Y}, is the collection of semitoric ingredients for semitoric systems with distinct marked points;
    \item $\widetilde{\bf{Y}}$, from Definition~\ref{def:Ytilde}, is similar to $\bf{Y}$ except that the twisting index and Taylor series invariants are packaged together. There is a natural bijection between $\bf{Y}$ and $\widetilde{\bf{Y}}$;
    \item $\widetilde{\bf{Z}}$, from Definition~\ref{def:Z}, is the collection of semitoric ingredients for general semitoric systems, covering both the cases of simple and non-simple. It naturally extends $\bf{Y}$ and $\widetilde{\bf{Y}}$.
\end{itemize}
So far, we have only abstractly defined the invariants. In the next section we will describe how to obtain these objects from a given semitoric system.

\subsection{Constructing the semitoric invariants from a  system}
\label{sec:st-construct}

For the duration of this section, we fix a semitoric system $(M,\om,F=(J,H))$, and we will outline the constructions of the semitoric invariants of this system. These invariants appeared in several papers~\cite{VN2003,VN2007} before being combined, along with another invariant, into the classification of simple semitoric systems by Pelayo and V\~{u} Ng\d{o}c~\cite{PeVN2009,PeVN2011}, which was later adapted into a classification of all semitoric systems~\cite{PPT-nonsimple}.

The marked semitoric polygon invariant is constructed in the same way if the system is simple or non-simple, and indeed the original construction of this object by V\~{u} Ng\d{o}c~\cite{VN2007} applied to both the simple and non-simple cases. The labels, on the other hand, differ between the simple and non-simple cases. In the simple case, each marked point is labeled by a Taylor series, as in~\cite{VN2003}, and an twisting index (which is an integer), as in~\cite{PeVN2009}. As discussed above, these two objects can be combined into a single Taylor series for each marked point (see Definition~\ref{def:Ytilde}). On the other hand, in the non-simple case, if several marked points coincide, corresponding to the case of several focus-focus singular points in a single fiber of the system, then they are labeled by a collection of Taylor series, as in~\cite{PT}, and the twisting index can be incorporated into these labels as in~\cite{PPT-nonsimple}.

\subsubsection{Constructing the marked semitoric polygon}
\label{sec:construct-poly}

As discussed above, for a compact toric integrable system, the image of the momentum map is a polytope whose dimension is half of the dimension of the system; if the system is four-dimensional it is a polygon. In a semitoric system $(M,\om,F)$, the image of the momentum map $F(M)\subset \R^2$ is generally not a polygon. In this section, we will describe a technique due to V\~{u} Ng\d{o}c~\cite{VN2007} to use the underlying integral affine structure of $F(M)$ to obtain a polygon. Certain choices made during the construction will affect the resulting polygon, so for the output to be well-defined we take an equivalence relation. The resulting family of polygons is a marked semitoric polygon, as in Definition~\ref{def:marked-st-poly}. We show the general idea in Figure~\ref{fig:cuts}.

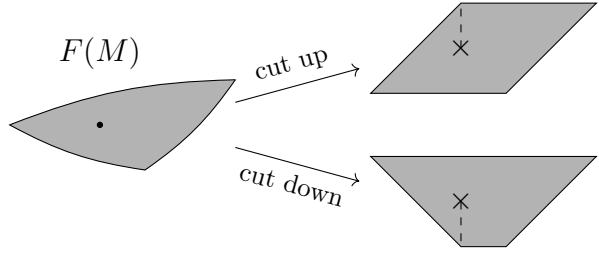
\begin{figure}
\centering
 \begin{tikzpicture}[scale=.6]
\filldraw[draw=gray!60,fill=gray!60] (0,0) node[anchor=north,color=black]{}
  -- (5,1) node[anchor=south,color=black]{}
  -- (3,-1) node[anchor=south,color=black]{}
  -- cycle;  
\draw[bend left=10, fill = gray!60] (0,0)to (5,1);
\draw[bend left=10, fill = gray!60] (5,1) to (3,-1);
\draw[bend left=10, fill = gray!60] (3,-1) to (0,0);
\fill[black] (2,0) circle (.07);

\draw (2,1) node[above] {$F(M)$};

\draw[->] (5.0, .5) -- node[above,sloped]{\footnotesize{cut up}} ++ (2.75, .75);
\draw[->] (5.0, -.5) -- node[below,sloped]{\footnotesize{cut down}} ++ (2.75, -.75);

\begin{scope}[xshift = 8cm, yshift = 1.7cm]
\filldraw[draw=black,fill=gray!60] (0,-1) node[anchor=north,color=black]{}
  -- (2,1) node[anchor=south,color=black]{}
  -- (5,1) node[anchor=south,color=black]{}
  -- (3,-1) node[anchor=south,color=black]{}
  -- cycle;
\draw (2,0) node {$\times$};
\draw[dashed] (2,0)--(2,1);
\end{scope}

\begin{scope}[xshift = 8cm, yshift = -1.7cm]
\filldraw[draw=black,fill=gray!60] (0,1) node[anchor=north,color=black]{}
  -- (5,1) node[anchor=south,color=black]{}
  -- (3,-1) node[anchor=south,color=black]{}
  -- (2,-1) node[anchor=south,color=black]{}
  -- cycle;
\draw (2,0) node {$\times$};
\draw[dashed] (2,0)--(2,-1);
\end{scope}
 \end{tikzpicture}
 \captionsetup{width=.9\linewidth}
 \caption{Straightening the affine structure of $F(M)$ to obtain semitoric polygon representatives.}
  \label{fig:cuts}
\end{figure}

The fibers of the momentum map in a semitoric system are connected~\cite{VN2007}, so we may associate the image $F(M)$ with the base of the fibration $B = M/\sim$ as in Section~\ref{sec:affine}.
Let $B_r\subset B = F(M)$ be the set of regular values of $F$.
In Section~\ref{sec:affine}, we described how $B_r$ inherits an integral affine structure, which typically does not agree with the integral affine structure of $\R^2$. 
To uncover the integral affine structure of $B_r$, and ``see'' it as in the toric case (as in Section \ref{sec:affine}), it would be nice to have a map $B_r\to\R^2$ which has the property that the pull back of the integral affine structure on $\R^2$ is equal to the integral affine structure on $B_r$. If the system has at least one focus-focus singular point, then there is monodromy (Definition~\ref{def:monodromy}) in the affine structure of $B_r$ which obstructs the existence of any such map.

Let $p_1,\ldots, p_m\in M$ denote the focus-focus points of $F$.
Recall that focus-focus values are the only critical values which can occur in the interior of the image of the momentum map for a semitoric system. Thus, $B_r = \mathrm{int}(B)\setminus \{F(p_1),\ldots,F(p_m)\}$.
The reason that this set can admit an integral affine structure with monodromy is that it is not simply-connected, and V\~{u} Ng\d{o}c's idea in this situation is to remove certain ``cuts'' from this set to obtain a simply connected subset, and therefore the integral affine structure restricted to this subset necessarily has trivial monodromy.

For $c\in\R^2$ and $\epsilon = \{-1,1\}$, as above we let $L_c^\epsilon$ denote the vertical ray starting at $c$ which goes up if $\epsilon=1$ and down if $\epsilon=-1$.  For a choice of $\vec{\epsilon}\in\{-1,1\}^m$, we consider the set of regular points with these cuts removed:
\[
 \widetilde{B}_r^{\vec{\epsilon}} := B_r\setminus\left( L^{\epsilon_1}_{F(p_1)} \cup \cdots \cup L^{\epsilon_m}_{F(p_m)} \right).
\]
In~\cite{VN2007}, V\~{u} Ng\d{o}c shows that for each $\vec{\epsilon}\in(\Z_2)^m$, there exists a map $g_{\vec{\epsilon}}\colon B \to \R^2$, which preserves the first coordinate and is a homeomorphism onto its image, which is a rational convex polyhedron, such that $g_{\vec{\epsilon}}|_{\widetilde{B}_r^{\vec{\epsilon}}}\colon \widetilde{B}_r^{\vec{\epsilon}} \to \R^2$ is a diffeomorphism onto its image that sends the integral affine structure of $\widetilde{B}_r^{\vec{\epsilon}}$ to the standard integral affine structure of $\R^2$. 
We will call such a map $g_{\vec{\epsilon}}$ a \emph{straightening map}\footnote{It is also sometimes called a \emph{cartographic homeomorphism}.}.

Let $\widetilde{M}^{\vec{\epsilon}} = F^{-1}(\widetilde{B}^{\vec{\epsilon}}_r\,)$.
An important observation here is that $( g_{\vec{\epsilon}}\circ F)|_{\widetilde{M}^{\vec{\epsilon}}}$ is the moment map for a Hamiltonian $2$-torus action on $\widetilde{M}^\epsilon$. This follows immediately from the fact that the pull-back of the standard integral affine structure on $\R^2$ by $g$ is equal to the integral affine structure on $B_r$.

We now have the following situation:
\[
\begin{tikzcd}[column sep = tiny]
M \arrow[r,phantom,"\supseteq"] \arrow[d,"F"']& \Tilde{M}^{\vec{\epsilon}}  \arrow[d, "{F}"'] \arrow[rrrd,"( g_{\vec{\epsilon}}\circ F)|_{\widetilde{M}^{\vec{\epsilon}}}"] & & &  & & & \\
F(M) \arrow[r,phantom,"\supseteq"]& \Tilde{B}_r^{\vec{\epsilon}} \arrow[rrr, "g_{\vec{\epsilon}}"'] & & & \R^2 && & 
\end{tikzcd}
\]

Understanding the monodromy of the integral affine structure around the image of a focus-focus point allowed V\~{u} Ng\d{o}c~\cite{VN2007} to understand the impact of changing cuts between up and down. 
Combining this with the observation that $( g_{\vec{\epsilon}}\circ F)|_{\widetilde{M}^{\vec{\epsilon}}}$ is the moment map for a Hamiltonian $2$-torus action on $\widetilde{M}^\epsilon$, Pelayo and V\~{u} Ng\d{o}c~\cite{PeVN2011} were able to determine certain restrictions on which polygons are possible (the restrictions in Definition~\ref{def:st-poly-repr}). 
Furthermore, in~\cite{PeVN2011} they also show that any such polygon can be obtained in this way from some semitoric system.
We summarize all of these results here:

\begin{theorem}[The marked semitoric polygon invariant]\label{thm:marked-poly}
    Let $(M,\om,F)$ be a semitoric system with $m\in\Z_{\geq 0}$ focus-focus singular points $p_1,\ldots,p_m\in M$. 
    Then:
    \begin{enumerate}
        \item for each $\vec{\epsilon}\in (\Z_2)^m$, there exists a straightening map $g_{\vec{\epsilon}}\colon F(M) \to \R^2$ such that:
    \begin{enumerate}
        \item the image $\De := (g_{\vec{\epsilon}}\circ F) (M)\subset \R^2$ is a convex rational polyhedron,
        \item $g_{\vec{\epsilon}}$ is a homeomorphism,
        \item $g_{\vec{\epsilon}}$ preserves the first coordinate, in the sense that there exists a $g_{\vec{\epsilon}}^{(2)}\colon \R^2\to\R$ such that $g_{\vec{\epsilon}}(x,y) = (x, g_{\vec{\epsilon}}^{(2)}(x,y))$,
        \item $g_{\vec{\epsilon}}|_{\widetilde{B}_{\vec{\epsilon}}}$ is a diffeomorphism onto its image such that the pull-back of the standard integral affine structure on $\R^2$ equals the integral affine structure on $\widetilde{B}_{\vec{\epsilon}}$.
    \end{enumerate}
    \item letting $c_k = g_{\vec{\epsilon}}\circ F(p_k)$ for $k\in\{1,\ldots, m\}$ and $\vec{c} = (c_1,\ldots, c_m)$, the triple $(\De,\vec{c},\vec{\epsilon}\,)$ is a semitoric polygon representative, as in Definition~\ref{def:st-poly-repr}.
    \item the set of all $(\De,\vec{c},\vec{\epsilon}\,)$ obtained in this way for some choice of $\vec{\epsilon}\in(\Z_2)^m$ and some $g_{\vec{\epsilon}}$ is exactly equal to the orbit $[\De,\vec{c},\vec{\epsilon}\,]$ under the action from Equation~\eqref{eqn:stpoly-action}, which is thus called \emph{the marked semitoric polygon invariant of $(M,\om,F)$}.
    \item If $(M,\om,F)$ and $(M',\om',F')$ are two semitoric systems which are isomorphic, then their marked semitoric polygon invariants are equal.
    \end{enumerate}
    Conversely, given any marked semitoric polygon $[\De,\vec{c},\vec{\epsilon}\,]$ as in Definition~\ref{def:marked-st-poly}, there exists a semitoric system $(M,\om,F)$ whose marked semitoric polygon invariant is $[\De,\vec{c},\vec{\epsilon}\,]$.
\end{theorem}

The above theorem can be interpreted as producing a surjection from the set of isomorphism classes of semitoric systems to the set of marked semitoric polygons. This surjection can be written as
\begin{equation}\label{eqn:M-to-De}
 (M,\om,F) \mapsto \Big[ g_{\vec{\epsilon}}\circ F(M), \big(g_{\vec{\epsilon}}\circ F(p_1),\ldots, g_{\vec{\epsilon}}\circ F(p_m)\big), \vec{\epsilon}\,\Big]
\end{equation}
where $\vec{\epsilon}$ is any choice of  $\vec{\epsilon}\in (\Z_2)^m$ and $g_{\vec{\epsilon}}$ is any choice of corresponding straightening map. The output of the map in Equation~\eqref{eqn:M-to-De} does not depend on these choices since it is an equivalence class of objects related by the action described in Equation~\eqref{eqn:stpoly-action}, and  that action is designed to compensate exactly for the impact of these choices on the resulting invariant.

Theorem~\ref{thm:marked-poly} combines results from several sources~\cite{VN2003,PeVN2009,PeVN2011,PPT-nonsimple}. In particular, the existence and properties of $g_{\vec{\epsilon}}$ was proved in~\cite{VN2007} for both simple and non-simple systems, the properties of the resulting polygons and the existence of a system for each such polygon was proven in~\cite{PeVN2009,PeVN2011} for simple semitoric systems, and these results were extended to non-simple semitoric systems in~\cite{PPT-nonsimple}.

\begin{remark}
The results in~\cite{VN2003,PeVN2009,PeVN2011} were actually only for unmarked semitoric polygons, which only record the $x$-component of the marked points, but to produce a more unified approach here we include the marked points on the polygons. In the original classification, the ``height'' of the marked points (vertical distance from the bottom of the polygon) was included as a separate invariant, which here is packaged together with the polygon to produce a marked polygon.
\end{remark}

\begin{remark}
    A similar object was introduced by Symington~\cite{Sy2003} for \emph{almost-toric fibrations}, which generalize semitoric integrable systems. In the polygons introduced by Symington, the cuts are not all constrained to be vertical, and indeed can show up at any rational slope. 
    See Section~\ref{sec:ATFs} and in particular Figure~\ref{fig:atf}.
    One way to think about it is that the slope of the cut is determined by the local $S^1$-action around the corresponding focus-focus point, and the fact that all cuts are vertical in a semitoric polygon is from the fact that we have prescribed that the first component of the momentum map generate a global $S^1$-action.
\end{remark}

\subsubsection{Constructing the labels}
\label{sec:construct-labels}

As mentioned several times above, while the marked semitoric polygon is an important invariant of a semitoric system, it is not a complete invariant. More precisely, for any choice of marked semitoric polygon which has at least one marked point, there are uncountably infinitely many non-isomorphic semitoric systems whose polygon invariant is the given polygon. Most of this extra freedom comes from the non-trivial invariants of an integrable system in the neighborhood of a fiber containing a focus-focus point; these are the \emph{Taylor series invariants} as discussed in~\cite{VN2003} for the case of one focus-focus point in the fiber and~\cite{PT} for the case of multiple focus-focus points in the same fiber.
Moreover, there is an additional invariant, called the \emph{twisting index}, representing a degree of freedom in how the neighborhood around the focus-focus fiber is glued into the semitoric system, as discussed in~\cite{PeVN2009,PeVN2011,AHP-twist} for the case of one focus-focus point in the fiber and~\cite{PPT-nonsimple} for the case of multiple focus-focus points in each fiber.

The invariants for systems with at most one focus-focus point in each fiber of the momentum map $F$ are a bit easier to work with compared with those for systems with arbitrary numbers of focus-focus points in each fiber, so we will deal with these two cases separately.

\paragraph{The Taylor series label for systems with at most one focus-focus point in each fiber.}Recall that in Proposition~\ref{prop:reg-action} we showed how to compute action variables near regular fibers by integrating a primitive of the symplectic form over a cycle, and in Remark~\ref{rmk:elliptic-actions} mentioned that due to Miranda-Zung~\cite{miranda-zung} this can also be accomplished around singular points with only regular and elliptic blocks. Now we want to do something similar for focus-focus points, using integrals over cycles to obtain invariants, but we will see that there are several difficulties that must be overcome.

For this subsection, suppose that $(M,\om,F=(J,H))$ has the property that there is at most one focus-focus singular point in each fiber of $F$. Note that this is slightly more general than the usual definition of simplicity, which requires at most one focus-focus point in each fiber of $J$.
For the remainder of this section, fix a focus-focus point $p\in M$.
Since we have assumed at most one focus-focus point in each fiber, $F^{-1}(F(p))$ is topologically a pinched torus, as in Figure~\ref{fig:pinchedtorus}, where $p$ is the pinched point.

Now, following~\cite{VN2003}, we outline how to construct the Taylor series invariant, which is an invariant of a neighborhood of a fiber containing a single focus-focus singular point.
Due to Theorem~\ref{thm:eliasson}, there are neighborhoods $U\subset M$ of $p$ and $V\subset \R^4$ of the origin, such that there exists a symplectomorphism $\phi\colon U \to V$ and a local diffeomorphism $g\colon \R^2\to\R^2$ such that $g \circ F = Q \circ \phi$
where
\[
 Q(x_1,x_2,y_1,y_2) = (x_1 y_2- x_2y_1, x_1 y_1+ x_2 y_2),
\]
so we have the following commutative diagram:
\[
\begin{tikzcd}[column sep = tiny]
M \arrow[r,phantom,"\supseteq"]& U \arrow[rrr, "\phi"] \arrow[d, "{F}"'] & & & V  \arrow[d, "Q"] &\R^4 \arrow[l,phantom,"\subseteq"]\\
& \mathbb{R}^2 \arrow[rrr, "g"] & & & \mathbb{R}^2 &
\end{tikzcd}
\]

Furthermore, since we are working in a semitoric system, as opposed to a general integrable system, there are extra conditions that we can place on $(\phi,g)$ to obtain a well-defined invariant in the end (V\~{u} Ng\d{o}c and Sepe~\cite{VNSepe} discuss how the freedom in choosing $(\phi,g)$ in general causes the Taylor series invariant to not be well-defined, but how orientation conventions can fix this issue in the semitoric case).
Recall that the first component $J$ of the map $F$ generates an effective $S^1$-action, so we may choose $(\phi,g)$ such that $g(x,y) = (x,g_2(x,y))$ where $\frac{\partial g_2}{\partial y}>0$.

\begin{remark}
    In fact, $g$ will always be of the form $g(x,y) = (\pm x, g_2(x,y))$ with $\frac{\partial g_2}{\partial y}\neq 0$, so to make sure that the invariant is well-defined we are essentially specifying the orientation of each component to both be positive. In~\cite[Section 5.2]{AHP-twist} the authors discuss how changing these signs impacts the invariant.
\end{remark}

This local form around the focus-focus point completely describes the behavior of the system in a neighborhood of $p$, but does not give all of the information about the behavior of the system in a neighborhood of the fiber $F^{-1}(F(p))$.
We will now see that there is another invariant in a neighborhood of $F^{-1}(F(p))$, which essentially measures the time it takes to flow around each level set of the torus in a direction complementary to the orbits of the $S^1$-action. 

By shrinking $U$ and $V$ if necessary, we may assume that both $U$ and $V$ are contractible.
Let $W = F^{-1}(F(U))$ and let
\begin{equation}\label{eqn:Phi}
    \Phi := g \circ F \colon W \to Q(V).
\end{equation}
Since $g$ preserves the first component, the first component of $\Phi$ is $J$, so we write $\Phi = (J,\Phi_2)$.

Identify $\R^2$ with $\C$ in the usual way, and let $\Lambda_z = \Phi^{-1}(z)$.
If $z\in \Phi(U)\setminus \{0\}$, then $\Lambda_z$ is a regular fiber, which is thus a 2-torus by the Liouville-Arnold-Mineur Theorem, Theorem~\ref{thm:LAM}.
For non-zero $z$, let $x\in \Lambda_z$ and note that the orbit of $x$ under the flow of $\mathcal{X}_J$ determines a loop on the torus, and the homology class of this loop is independent of the choice of $x$; we denote it by $[\gamma_J^z]\in H_1(\Lambda_z)$.
An important observation now is that the monodromy of the torus bundle for non-zero $z$ obstructs the existence of a smoothly varying choice of a complementary cycle in $H_1(\Lambda_z)$. We will overcome this difficulty by ``cutting'' through $\Phi(U)$ to produce a simply connected base (similar to the cutting we needed to perform to produce the marked semitoric polygon invariant).

Fix a choice of $\epsilon \in \Z_2$ and let $L_\epsilon$ be the ray starting at $F(p)$ which goes up if $\epsilon=1$ and down if $\epsilon=-1$.
Let $\widetilde{W} := F^{-1}(F(W)\setminus L_\epsilon)$.
Now $\Phi|_{\widetilde{W}}\colon \widetilde{W}\to g(F(W)\setminus L_\epsilon)$ is a $2$-torus bundle over a contractible set, and is thus a trivial bundle. We may choose an identification between $\widetilde{W}$ and $\Phi(\widetilde{W})\times (S^1)^2$ such that for each $z\in \Phi(\widetilde{W})$, the cycle $[\gamma_J^z]$ corresponds to the first basis element of $H_1(S^1\times S^1)$, and then for each $z\in \Phi(\widetilde{W})$ we define $[\gamma_2^z]\in H_1(\Lambda_z)$ to be the cycle corresponding to the second basis element.
Thus, we have obtained a basis $\{[\gamma_J^z],[\gamma_2^2]\}$ of $H_1(\Lambda_z)$ which varies continuously with $z\in \Phi(\widetilde{W})$. There is a freedom in the choice or orientation of $[\gamma_2^z]$, which we will fix later.

Let $\alpha$ be an $S^1$-invariant one-form on $W$ which is a primitive of $\omega$, so $\omega=\dd \alpha$ on $W$. Note that if $x\in\widetilde{W}$ then the integral of $\alpha$ around $[\gamma_J^{\Phi(x)}]$ is well defined (i.e.~independent of the choice of representative).
Now we claim
\[
 J(x)= J(p)+\frac{1}{2\pi}\oint_{\gamma_J^{\Phi(x)}}\alpha
\]
for all $x\in \widetilde{W}$.

\begin{proof}
 Let $\mathcal{X}$ denote the Hamiltonian vector field of $J$, and recall that $\mathcal{X}$ generates an $S^1$-action (with period $2\pi$). Since $\alpha$ is $S^1$ invariant, we know that $\mathcal{L}_{\mathcal{X}}\alpha =0$. Then, Cartan's formula and the fact that $\iota_\mathcal{X}\omega = -\dd J$ tell us that
 \[
   0=\mathcal{L}_\mathcal{X}\alpha = \dd(\iota_\mathcal{X}\alpha) + \iota_\mathcal{X}(\dd \alpha) = \dd(\alpha(\mathcal{X})) + \iota_\mathcal{X}\omega = \dd(\alpha(\mathcal{X}) - J)
 \]
 so $x\mapsto\alpha(\mathcal{X})_x-J(x)$ is a constant function. Since $p$ is a focus-focus point, $\mathcal{X}_p=0$ so $\alpha(\mathcal{X})=0$, so necessarily  $\alpha(\mathcal{X})_x-J(x)=-J(p)$. Now, let $x\in M$ and let $\gamma\colon S^1\to M$ be a parameterization of $S^1\cdot x$, the $S^1$-orbit through $x$. Then
 \[
  \int_\gamma \alpha = \int_{S^1}\gamma^*(\alpha) = \int_0^{2\pi}\alpha_{\gamma(\theta)}(\dot{\gamma}(\theta))\dd\theta = \int_0^{2\pi}(\alpha(\mathcal{X}))_{\gamma(\theta)}\dd\theta.
 \]
 Since $\alpha(\mathcal{X})_x = -J(p)+J(x)$, the equation above implies that
 \[
  \int_\gamma \alpha = \int_0^{2\pi}(-J(p)+J(\gamma(\theta)))\dd\theta = 2\pi(-J(p)+J(x))
 \]
 since $J$ is constant on the orbit $\gamma$.
\end{proof}

In fact, in $U$ we can give a formula for such a primitive, $\alpha = \phi^* \left(\sum x_i\dd y_i\right)$. Now we have one of the two action variables, which is actually just the one that was already given ($J$), and we would like to form the other by integrating over the complementary cycle.
Define $I\colon \Phi(\widetilde{W})\to\R$ by
\[
I(z) := \frac{1}{2\pi}\oint_{\gamma_2^z}\alpha.
\]
Again, this doesn't depend on the choice of representative of $[\gamma_2^z]$.
Now we will fix the orientation of $[\gamma_2^z]$ to be such that the derivative of $I$ with respect to the imaginary part of its input is strictly positive. 
We now have the following situation:
\[
\begin{tikzcd}[column sep = tiny]
M \arrow[r,phantom,"\supseteq"]& \widetilde{W}  \arrow[d, "{F}"'] \arrow[rrrd,"\Phi"] & & &  & & & \\
& \mathbb{R}^2 \arrow[rrr, "g"] & & & \Phi(\widetilde{W}) \arrow[d,phantom,"\subseteq"{rotate=-90}]   \arrow[rrr,"I"] && &\R\\&&&&\C&&& 
\end{tikzcd}
\]
The function $I$ cannot be smoothly extended from $\Phi(\widetilde{W})$ to all of $\Phi(W)$, but it can be continuously extended.

Speaking informally, there are two reasons why $I$ cannot be extended to a smooth function on $\Phi(W)$:
\begin{enumerate}[nosep]
    \item the value of $I$ jumps along the cut $L_\epsilon$,
    \item $I$ blows up as $z\to 0$,
\end{enumerate}
and both of these issues also exist for the complex logarithm. It turns out that a logarithm can be used to subtract them away. More precisely, let $\log$ denote a determination of the complex logarithm with branch cut along $L^\varepsilon$ and, for $z\in\Phi(\widetilde{W})$, define
\begin{equation}\label{eqn:S}
S(z):=2\pi I(z) -2\pi I(0) +\text{Im}(z\log(z)-z).
\end{equation}
Then V\~{u} Ng\d{o}c~\cite{VN2003} showed that $S(z)$ can be extended smoothly to a function on all of $\Phi(W)$, which we will also call $S$. This function $S(z)$ is sometimes called the desingularized or regularized action (see~\cite{PeVNsymplthy2011}, for instance).

Write $z=X+iY$ and view $S$ as a function of $X$ and $Y$, so $S(X,Y) = S(X+iY)$.
The function $S$ is not unique for several reasons:
\begin{itemize}
    \item $I$ depends on the choice of cycle $[\gamma_2^z]$ which was chosen to form basis with $[\gamma_J^z]$, so for any $k\in\Z$ the choice $[\gamma_2^z]+k[\gamma_J^z]$ will also work. Choosing a different such class changes $S$ by an integer multiple of $X$;
    \item $I$ depends on the choice of determination of log, which also is only defined up to the addition of $X\Z$;
    \item $S$ also depends on the choice of chart for the normal form of the focus-focus point ($U$, $\phi$, $V$, and $g$ from Theorem~\ref{thm:eliasson}). Different charts yield functions $S$ which differ by the addition of a smooth function for which all derivatives vanish when $z=0$, a so-called \emph{flat function}.
\end{itemize}

 Then the first two items above imply that $S$ is only well-defined up to the addition of $2\pi X \Z$, and the last item implies that the function $S$ itself is not an invariant - instead we have to consider its Taylor series. 
Let $S^\infty$ denote the Taylor series of $S$ at the origin, and note that Equation~\eqref{eqn:S} is designed so that $S(0)=0$.
Let $\R_0[[X,Y]]$ denote the set of power series in the variables $X$ and $Y$ with zero constant term.
Taking the Taylor series removes the impact of flat functions, and therefore taking the quotient by $2\pi X \Z$ is enough to make $S^\infty$ well-defined.
That is, by~\cite{VN2003},
\begin{equation}\label{eqn:taylor-series-invt}
S^\infty\in\R_0[[X,Y]]/(2\pi X\Z)
\end{equation}
is independent of all choices, and is called the \emph{Taylor series invariant} of the focus-focus point $p$.

Moreover, V\~{u} Ng\d{o}c showed that the semilocal structure near a focus-focus point is classified by this invariant~\cite{VN2003}. More precisely: 
\begin{itemize}
    \item for $i\in\{1,2\}$, let $(M_i,\om_i,F_i)$ be a 4-dimensional integrable system and let $p_i\in M_i$ be a focus-focus point with the property that it is the only focus-focus point in its fiber of $F_i$. Then there exists a fiber-preserving symplectomorphism from a neighborhood of $F^{-1}_1(F_1(p_1))$ to a neighborhood of $F^{-1}_2(F_2(p_2))$ if and only if the Taylor series invariant of $p_1$ is equal to the Taylor series invariant of $p_2$;
    \item Let $S^\infty\in\R_0[[X,Y]]/(2\pi X \Z)$. Then there exists an integrable system with a focus-focus point with $S^\infty$ as its Taylor series invariant.
\end{itemize}
In other words, this construction produces a bijection from the isomorphism class of semi-local neighborhoods of focus-focus points to $\R_0[[X,Y]]/(2\pi X \Z)$.

\begin{remark}
    There is another construction of the Taylor series invariant which is also due to V\~{u} Ng\d{o}c, in terms of the return time of certain Hamiltonian flows. Roughly, the idea is to choose a fiber $\Lambda_z$ near the focus-focus value $F(p)$, choose a point $x\in\Lambda_z$, and determine two values:
    \begin{itemize}[nosep]
        \item let $\tau_2(z)\in\R_{>0}$ be the time it takes to follow the flow of the Hamiltonian vector field of the second component of $\Phi=(J,\Phi_2)$ from $x$ until it returns to the orbit $S^1\cdot x$;
        \item let $\tau_1(z)\in\R/(2\pi \Z)$ denote the time it takes to follow the flow from that point back to $x$.
    \end{itemize}
    So here we are following the flow of $\mathcal{X}_{\Phi_2}$ and then the flow of $J$, producing a closed path from $p$ back to $p$. It turns out that the functions $\tau_1(z)$ and $\tau_2(z)$ are independent of the choice of $x$, and furthermore V\~{u} Ng\d{o}c~\cite{VN2003} showed that 
    \[
     \sigma_1(z):=\tau_1(z)-\text{Im}(\log z)\quad \text{ and }\quad  \sigma_2(z) := \tau_2(z)+\text{Re}(\log z)
    \]
    can be extended to smooth real-valued functions around the origin, taking any determination of $\log$ and choosing a lift of $\tau_1$ which is discontinuous along the same branch. It can be shown that $\sigma= \sigma_1\dd X + \sigma_2 \dd Y$ is exact, and we can therefore define $S$ to be the unique smooth function such that $\dd S = \sigma$ and $S(0)=0$. Now, again we take the Taylor series of this function to obtain $S^\infty$, and though there are choices in this process it produces a well defined element of $\R_0[[X,Y]]/(2\pi X \Z)$.
\end{remark}

\paragraph{The twisting index for systems with at most one focus-focus point in each fiber.} The twisting index invariant first appeared in~\cite{PeVN2009}, and it is related to how the semi-local neighborhood of a focus-focus point can be glued into a larger integrable system - in some sense it measures the interaction between the Taylor series invariant which encodes the semilocal structure and the polygon invariant which describes the global structure. 
There are now several (equivalent) ways to define this invariant, as discussed in~\cite{AHP-twist}.

The original definition, due to Pelayo and V\~{u} Ng\d{o}c~\cite{PeVN2009}, is in terms of a difference of momentum maps near a focus-focus point: one momentum map related to the polygon invariant (Section~\ref{sec:construct-poly}) and the other one a certain local preferred momentum map around the focus-focus point. In~\cite{AHP-twist}, the authors show that there are three other equivalent ways to describe the twisting index: as a comparison of actions, Taylor series, or cycles. In each case, the twisting index appears by comparing one object which is a preferred local object around the focus-focus point against an object which is obtained from the choice of polygon invariant. In this paper, we will define the twisting index as a comparison of cycles.

Fix a focus-focus point $p\in M$, and suppose that there are no other focus-focus points in the fiber $F^{-1}(F(p))$.
Let $m$ denote the number of focus-focus points of the system.
Let $[\De,\vec{c},\vec{\epsilon}\,]$ be a semitoric polygon associated to the system, so $\De = g_{\vec{\epsilon}}\circ F (M)$ for some choice of cuts $\vec{\epsilon}\in\{\pm 1\}^m$ and straightening map $g_{\vec{\epsilon}}$ as in Theorem~\ref{thm:marked-poly}.
Now, let \[\Psi^{\De} := (g_{\vec{\epsilon}}\circ F)|_{\widetilde{M}^{\vec{\epsilon}}}\colon \widetilde{M}^{\vec{\epsilon}}\to\R^2,\] where $\widetilde{M}_{\vec{\epsilon}}$ is $M$ with the preimages of the cuts removed, as discussed in Section~\ref{sec:construct-poly}. Let $\Psi^\De = (\Psi_1^\De,\Psi_2^\De)$.
Then, $\Psi^\De$ generates a Hamiltonian 2-torus action on $\widetilde{M}_{\vec{\epsilon}}$, and in particular each $\Psi_i^\De$ generates an $S^1$-action. Observe that $\Psi_1^\De=J$. For any regular fiber of $F|_{\widetilde{M}_{\vec{\epsilon}}}$, which is necessarily a torus, the orbits of the flows of $\Psi_1^\De$ and $\Psi_2^\De$ form a basis of $H_1$ of the fiber. The idea now is to compare this basis with a preferred one defined near the focus-focus point $p$.

Let $W$, $Q$, $V$, $\widetilde{W}$, and $\Phi \colon W\to Q(V)$ be as defined in the discussion around Equation~\eqref{eqn:Phi}. Recall that $\Phi = (J, \Phi_2)$ and the Hamiltonian flow of $J$ generates an $S^1$-action while the Hamiltonian flow of $\Phi_2$ in general just generates an $\R$-action. Let $z\in\Phi(\widetilde{W})\subseteq\C$ and consider the fiber $\Lambda_z=\Phi^{-1}(z)$, which is a torus. Let $x\in\Lambda_z$ and define a loop $\gamma_\text{pref}^z$ by performing the following steps:
\begin{enumerate}[noitemsep]
    \item start at $x$,
    \item follow the flow of $\mathcal{X}_{\Phi_2}$ until the first time that it encounters the orbit $S^1\cdot x$ (where the $S^1$-action is the one generated by the flow of $\mathcal{X}_J$),
    \item follow the flow of $\mathcal{X}_J$ until $x$ is encountered again.
\end{enumerate}
It turns out that $[\gamma_\text{pref}^z]\in H_1(\Lambda_z)$ is independent of the choice of $x$. Similarly, since the Hamiltonian flow of $\Psi_2^\De$ induces an $S^1$-action on $\Lambda_z$, given any $x\in\Lambda_z$ the orbit of the flow of $\mathcal{X}_{\Phi_2^\De}$ acting on $x$ determines an element of $H_1(\Lambda_z)$ which is independent of the choice of $x$. We will denote this orbit by $[\gamma_\De^z]$. Finally, let $[\gamma_J^z]\in H_1(\Lambda_z)$ be the cycle determined by the flow of $\mathcal{X}_J$.

Thus, we have obtained $[\gamma_\text{pref}^z],[\gamma_\De^z],[\gamma_J^z]\in H_1(\Lambda_z)$ such that:
\begin{enumerate}
    \item[(1)] $[\gamma_\text{pref}^z]$ is constructed from the local structure around the focus-focus fiber $F^{-1}(p)$,
    \item[(2)] $[\gamma_\De^z]$ is constructed from the choice of polygon $\De$,
    \item[(3)] $[\gamma_J^z]$ is determined by the given $S^1$-action generated by $J$,
    \item[(4)] both $\{[\gamma_J],[\gamma_\text{pref}^z]\}$ and $\{[\gamma_J],[\gamma_\De^z]\}$ form a basis of $H_1(\Lambda_z)\cong \Z^2$.
\end{enumerate}
By item (4), $[\gamma_\text{pref}^z]$ and $[\gamma_\De^z]$ must differ by an integer multiple of $[\gamma_J^z]$, and this integer is the final invariant of semitoric systems: the twisting index. It can be shown that this integer does not depend on the choice of $z$.

The \emph{twisting index} of the focus-focus point $p$ is the integer $\kappa^\De\in\Z$ such that
\begin{equation}\label{eqn:twisting-index-invt}
 [\gamma_\De^z] - [\gamma_\text{pref}^z]=\kappa^\De[\gamma_J^z]
\end{equation}
for all $z\in\Phi(\widetilde{W})$.
Again, this is not the original definition of the twisting index, but it was shown to be equivalent in~\cite{AHP-twist}.
Note that Equation~\eqref{eqn:twisting-index-invt} compares a intrinsic property of the system (the preferred cycle) to a cycle that comes from the choice of polytope.

\begin{remark}
Note that $[\gamma_\text{pref}^z]$ cannot be formed by considering the flow of the given Hamiltonian $H$, since in that case it would not be invariant under isomorphisms of integrable systems.
Instead, we must use the flow of $\Phi_2$, which comes from the local momentum map from the normal form around the focus-focus point, and in that case it turns out that the twisting index and Taylor series invariant are well-defined.
\end{remark}

\begin{remark}
    In Definition~\ref{def:Ytilde} in Section~\ref{sec:st-invariants} we describe an alternative way to encode the invariants of a semitoric system, in the set $\widetilde{\bf{Y}}$, in which we combine the twisting index and Taylor series invariants together. Given that in the present section we have now described how to construct each of these invariants separately, we can now simply use the map described after Definition~\ref{def:Ytilde} to obtain an element of $\widetilde{\bf{Y}}$. On the other hand, it is also possible to construct the Taylor series $\widetilde{S}^\infty_k$ directly from the system, which then naturally includes both the information of the Taylor series and twisting index invariants. This approach of packaging the twisting index with the Taylor series is used in~\cite{PPT-nonsimple,jaume-thesis,LFVN,AHP-twist}.
\end{remark}

Above we have obtained the Taylor series invariant (Equation~\eqref{eqn:taylor-series-invt}) and twisting index invariant (Equation~\eqref{eqn:twisting-index-invt}) for each focus-focus point of the system. The Taylor series and twisting index invariants for the system itself are simply the collection of these for each focus-focus point.

\paragraph{The labels for general systems.}
Now suppose that $(M,\om,F)$ is any semitoric system, so that it may have more than one focus-focus point in a single fiber of $F$. Pelayo and Tang~\cite{PT} showed how to generalize V\~{u} Ng\d{o}c's construction of the Taylor series invariant for a fiber containing exactly one focus-focus point (from~\cite{VN2003}) to the case of any finite number of focus-focus points in the same fiber.
In~\cite{PPT-nonsimple}, this was used to extend the classification of simple semitoric systems by Pelayo and V\~{u} Ng\d{o}c~\cite{PeVN2009,PeVN2011} to all semitoric systems, simple or not.

We will not describe this general construction in detail here, instead we will discuss the idea and refer the reader to \cite{PPT-nonsimple} for the details.
The construction of the marked semitoric polygon is the same as in the non-simple case, and for any fiber of $F$ with exactly one focus-focus point, the construction of the Taylor series is also the same as the non-simple case. The difference comes when constructing the Taylor series invariants for a fiber of $F$ with more than one focus-focus point, which is topologically a multi-pinched torus (such as the one shown in Figure~\ref{fig:multipinched}). 

First a note on the indexing. Given a focus-focus fiber which contains $m$ focus-focus points, there is no way to determine which is a preferred one to start the numbering on, and this is why the indexing is always only up to cyclic permutations, i.e.~we number them by elements of $\Z_m$ and quotient by cyclic permutation.

In the single-pinched case, the invariant is constructed by extending the local coordinates of the normal form of the focus-focus point (as in Theorem~\ref{thm:eliasson}) to a neighborhood of the entire fiber, and then examining the behavior of the integral of the cycles determined by the Hamiltonian flows of the components of the momentum map $\Phi$, see Equation~\eqref{eqn:Phi}. In the multi-pinched case, there is a normal form (from Theorem~\ref{thm:eliasson}) for \emph{each} focus-focus point on the fiber, and each of these gives different coordinates.
Now, V\~{u} Ng\d{o}c's idea (explained in~\cite{VN2003} and executed in~\cite{PT}) is to compare all of these different local coordinates around the focus-focus points. This is why we obtain a collection of Taylor series instead of a single Taylor series. The series $\tilde{\mathsf{s}}_\mu$ are constructed in a similar way to the single-pinched case, by fixing a single focus-focus point and considering an integral along a path around the fiber determined by the local normal form of that point, while the Taylor series $\mathsf{g}_{\mu,\nu}$ are related to comparing the normal form around the $\mu^{\text{th}}$ focus-focus point with the normal form around the $\nu^{\text{th}}$ focus-focus point.

\optional{could maybe say more here? Maybe not}

\paragraph{Conclusion of Section~\ref{sec:st-construct}}
In this section we have described how to construct invariants from a given semitoric system. Part of the results of~\cite{PeVN2009,PeVN2011} is that the constructions for simple systems in this subsection actually yield an element of $\bf{Y}$. That is, the objects constructed here satisfy all of the conditions to be a semitoric ingredient. Similarly, one of the results of~\cite{PPT-nonsimple} is that the constructions for non-simple systems yield an element of $\Tilde{\bf{Z}}$, that is, they are generalized semitoric ingredients.

\subsection{The classification theorems}
\label{sec:st-classify}

In Section~\ref{sec:st-invariants} we abstractly described the invariants of a semitoric system, in both the simple and non-simple cases, and in Section~\ref{sec:st-construct} we described how to obtain these invariants from a given system.
In this section, we state the results of~\cite{PeVN2009,PeVN2011} and~\cite{PPT-nonsimple}, which state that these invariants actually classify semitoric systems.

Recall the collections of invariants that we described in Section~\ref{sec:st-invariants}:
\begin{itemize}[noitemsep]
    \item $\bf{Y}$, from Definition~\ref{def:Y}. Elements of $\bf{Y}$ are equivalence classes of marked semitoric polygons with distinct marked points, where each marked point is labeled by a single Taylor series and integer;
    \item $\widetilde{\bf{Y}}$, from Definition~\ref{def:Ytilde}. Elements of $\widetilde{\bf{Y}}$ are equivalence classes of marked semitoric polygons with distinct marked points, where each marked point is labeled by a single Taylor series (which includes the information of the twisting index in one of its linear terms);
    \item $\Tilde{\bf{Z}}$, from Definition~\ref{def:Z}. Elements of $\widetilde{\bf{Z}}$ are equivalence classes of marked semitoric polygons, whose marked points may or may not be distinct, where each marked value (which may include several marked points) is labeled by a collection of Taylor series.
\end{itemize}

There is a natural bijection from $\bf{Y}$ to $\widetilde{\bf{Y}}$ (they are two different ways to encode the invariants of a simple semitoric system), and there are natural injections from $\bf{Y}$ and $\widetilde{\bf{Y}}$ into $\widetilde{\bf{Z}}$ (since $\widetilde{\bf{Z}}$ naturally extends the invariants to also include non-simple systems).

Let $\mathcal{M}_{\text{ST}}$ denote the set of isomorphism classes of semitoric systems. Let $\mathcal{M}_{\text{ST}}^{\text{distinct}}\subset  \mathcal{M}_{\text{ST}}$ denote the isomorphism classes of those systems for which each fiber of the momentum map $F$ contains at most one focus-focus point.

\begin{theorem}[Classification of $\mathcal{M}_{\text{ST}}^{\text{distinct}}$~\cite{PeVN2009,PeVN2011}]\label{thm:PVN}
    Semitoric systems with at most one focus-focus point in each fiber of their momentum map are classified by $\bf{Y}$. That is, the construction of invariants described for simple semitoric systems in Section~\ref{sec:st-construct} is a bijection from $\mathcal{M}_{\text{ST}}^{\text{distinct}}$ to $\bf{Y}$.
\end{theorem}

Since there is a natural bijection between $\bf{Y}$ and $\widetilde{\bf{Y}}$, Theorem~\ref{thm:PVN} also implies that semitoric systems are classified by $\widetilde{\bf{Y}}$.

\begin{remark}
 Recall that a semitoric system $(M,\om,F=(J,H))$ is called simple if there is at most one focus-focus point in each fiber of $J$.
 Theorem~\ref{thm:PVN} as we have stated it is actually more general than the original result from~\cite{PeVN2009,PeVN2011}, since the original results only apply to simple systems (at most one focus-focus point in each fiber of $J$), while we have stated it for all systems that have at most one focus-focus point in each fiber of $F$. The fact that the classification still works in this slightly more general case follows from the result in~\cite{PPT-nonsimple}.
\end{remark}

\begin{theorem}[Classification of all semitoric systems~\cite{PPT-nonsimple}]\label{thm:non-simple}
    All semitoric systems (simple or not) are classified by $\widetilde{\bf{Z}}$. 
    That is, the construction of invariants for general semitoric systems described in Section~\ref{sec:st-construct} is a bijection from $\mathcal{M}_{\text{ST}}$ to $\widetilde{\bf{Z}}$.
\end{theorem}

Note that the three different sets of invariants that we have described ($\bf{Y}$, $\Tilde{\bf{Y}}$, and $\Tilde{\bf{Z}}$) are all based around attaching labels to the same basic object: the marked semitoric polygon.

\begin{corollary}\label{cor:semitoric-polygons}
    Every marked semitoric polygon can be obtained as the invariant from some semitoric system. In other words, the map from $\mathcal{M}_{\text{ST}}$ to the set of marked semitoric polygons defined by constructing the invariant as in Section~\ref{sec:construct-poly} is a surjection.
\end{corollary}

Corollary~\ref{cor:semitoric-polygons} is useful because it helps make it easier to work directly with the polygons. Note that the surjection described in the corollary is not a bijection because of the existence of the Taylor series and twisting index invariants.

\begin{remark}
    It is not evident from the way that the invariants are constructed above, but the $S^1$-equivariant symplectomorphism type of a semitoric integrable system is independent of the labels. That is, it is completely encoded in the marked semitoric polygon (and in fact is even independent of the vertical positions of the marked points). There are several ways to see this, but the easiest is to apply~\cite{HSS2015} and note that the Karshon graph, which completely characterizes the underlying symplectic manifold and Hamiltonian $S^1$-action, can be read directly off of the marked semitoric polygon invariant (see Corollary~\ref{cor:S1-type} below).
\end{remark}

\begin{remark}\label{rmk:computations}
The invariants of semitoric systems were developed abstractly in~\cite{PeVN2009,PeVN2011} for simple systems and~\cite{PPT-nonsimple} in general, but actually computing the invariants (especially the Taylor series and twisting index invariants) in explicit systems turns out to be very involved.
Techniques to perform these calculations, along with the calculations themselves for certain systems, appear in~\cite{ADH-spin-osc,ADH-momenta,AH21, dullin13, LFP,AHP-twist}.
See also the survey~\cite{AH-survey}.
\end{remark}

\begin{remark}
    With the classifications Theorem~\ref{thm:PVN} and Theorem~\ref{thm:non-simple} in hand, it is possible to study the moduli space of semitoric systems by examining and deforming the invariants. This was done for toric systems in dimension four~\cite{PPRStoric}, for semitoric systems~\cite{PaSTMetric2015}, and for toric systems in all dimensions~\cite{PelayoSantos23}.
\end{remark}

\begin{remark}
    There exist various generalizations of this theory. For instance, Faithful semitoric systems~\cite{HSS-vertical}, $b$-semitoric systems~\cite{b-semitoric}, and proper semitoric systems~\cite{PeRaVN2015}.
\end{remark}

\subsection{The generalized coupled angular momenta example}
\label{sec:generalized-spins}

In Example~\ref{ex:coupledspins} we described the example of the coupled angular momenta system. Here we give a more general example which, depending on choices of parameters, can have zero, one, or two focus-focus points, and for certain parameters can even have two focus-focus points in the same fiber. This example first appeared in~\cite{HoPa2017}.

    Let $M=S^2\times S^2$ and choose parameters $R_1,R_2\in\R_{>0}$ with $R_1\leq R_2$. View $S^2\times S^2$ as a submanifold of $\R^3\times \R^3$ with coordinates $(x_1,y_1,z_1,x_2,y_2,z_2)$. Equip $M$ with symplectic form $\om =-(R_1\om_{S^2}\oplus R_2\om_{S^2})$ where $\om_{S^2}$ is the usual volume form on the sphere. For any choice of parameters $t_1,t_2,t_3,t_4$ let
    \begin{equation}\label{eqn:generalized-spins}
     \begin{cases}
         J  &= R_1z_1 + R_2 z_2;\\
         H  &= t_1z_1 + t_2 z_2 + t_3 (x_1x_2+y_1y_2)+t_4z_1z_2.
     \end{cases}
    \end{equation}
    Then $(M,\om,(J,H))$ is the \emph{generalized coupled angular momenta system} from~\cite{HoPa2017}. Note that taking $t_1=-t$, $t_2=0$, $t_3=t_4=t$, we obtain the usual coupled angular momenta system with parameter $t$ from Example~\ref{ex:coupledspins}.
    \begin{theorem}[Proposition 3.13 from~\cite{HoPa2017}]
     The generalized angular momenta system in Equation~\eqref{eqn:generalized-spins} is integrable for all choices of parameters $R_1,R_2\in\R_{>0}$ and $t_1,t_2,t_3,t_4\in\R$ as long as $t_3\neq 0$.   
    \end{theorem}
     Note that in~\cite{HoPa2017} this is only stated for $R_1\neq R_2$, but this assumption is never used in the proof of integrability. Furthermore, note that this result is not a necessary and sufficient condition, since there exist choices of parameters for which the system is integrable but $t_3=0$.

     This system has different behavior depending on the choice of parameters. For instance:
     \begin{enumerate}
         \item If $t_1=R_1$ and $t_2=t_3=t_4=0$, then the system is toric for any $R_1,R_2\in\R_{>0}$. The Delzant polygon is shown in Figure~\ref{fig:gen-spins_toric}.
         \item If $t_1=t_3=t_4=\frac{1}{2}$ and $t_2=0$, then the system is semitoric with exactly one focus-focus point for any $R_1,R_2\in\R_{>0}$. A representative of the semitoric polygon is shown in Figure~\ref{fig:gen-spins_st}.
         \item If $R_1=1$, $R_2=2$, $t_1=t_2=\frac{1}{4}$, $t_3=\frac{1}{2}$, and $t_4=0$, then the system is semitoric with exactly two focus-focus points (Theorem 1.1 of~\cite{HoPa2017}). A representative of the semitoric polygon is shown in Figure~\ref{fig:gen-spins_st-2ff}.
         \item If $R_1=R_2=1$, $t_1=t_2=\frac{1}{4}$, $t_3=\frac{1}{2}$, and $t_4=0$, then the system is semitoric with exactly two focus-focus points, and both focus-focus points lie in the same fiber $(J,H)^{-1}(0,0)$, creating a double-pinched torus (this example is worked out in Section 1.3 of~\cite{PPT-nonsimple}). A representative of the semitoric polygon is shown in Figure~\ref{fig:gen-spins_non-simple}. See also Figure~\ref{fig:st-poly-non-simple}.
     \end{enumerate}

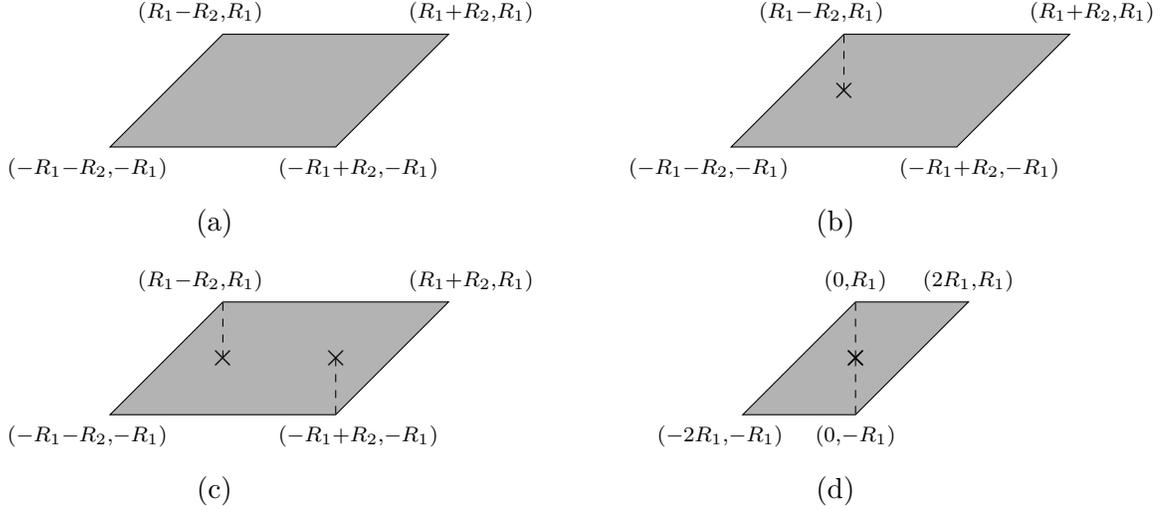
\begin{figure}
\begin{center}
\begin{subfigure}[b]{.35\linewidth}
\centering
\begin{tikzpicture}[scale=1.5]
\filldraw[draw=black, fill=gray!60] (0,0) 
  -- (1,1)
  -- (3,1)
  -- (2,0)
  -- cycle;
\draw (-.2,0) node[below] {$\scriptstyle{(-R_1-R_2,-R_1)}$};
\draw (.8,1) node[above] {$\scriptstyle{(R_1-R_2,R_1)}$};
\draw (3.2,1) node[above] {$\scriptstyle{(R_1+R_2,R_1)}$};
\draw (2.2,0) node[below] {$\scriptstyle{(-R_1+R_2,-R_1)}$};
\end{tikzpicture}
\caption{}
\label{fig:gen-spins_toric}
\end{subfigure}\qquad\qquad\qquad
\begin{subfigure}[b]{.35\linewidth}
\centering
\begin{tikzpicture}[scale=1.5]
\filldraw[draw=black, fill=gray!60] (0,0) 
  -- (1,1)
  -- (3,1)
  -- (2,0)
  -- cycle;
\draw (1,0.5) node {$\times$}; 
\draw [dashed] (1,0.5) -- (1,1); 
\draw (-.2,0) node[below] {$\scriptstyle{(-R_1-R_2,-R_1)}$};
\draw (.8,1) node[above] {$\scriptstyle{(R_1-R_2,R_1)}$};
\draw (3.2,1) node[above] {$\scriptstyle{(R_1+R_2,R_1)}$};
\draw (2.2,0) node[below] {$\scriptstyle{(-R_1+R_2,-R_1)}$};
\end{tikzpicture}
\caption{}
\label{fig:gen-spins_st}
  \end{subfigure}\\[6pt]
  \begin{subfigure}[b]{.35\linewidth}
\centering
\begin{tikzpicture}[scale=1.5]
\filldraw[draw=black, fill=gray!60] (0,0) 
  -- (1,1)
  -- (3,1)
  -- (2,0)
  -- cycle;
\draw (1,0.5) node {$\times$}; 
\draw [dashed] (1,0.5) -- (1,1); 
\draw (2,0.5) node {$\times$}; 
\draw [dashed] (2,0.5) -- (2,0); 
\draw (-.2,0) node[below] {$\scriptstyle{(-R_1-R_2,-R_1)}$};
\draw (.8,1) node[above] {$\scriptstyle{(R_1-R_2,R_1)}$};
\draw (3.2,1) node[above] {$\scriptstyle{(R_1+R_2,R_1)}$};
\draw (2.2,0) node[below] {$\scriptstyle{(-R_1+R_2,-R_1)}$};
\end{tikzpicture}
\caption{}
\label{fig:gen-spins_st-2ff}
\end{subfigure}\qquad\qquad\qquad
\begin{subfigure}[b]{.35\linewidth}
\centering
\begin{tikzpicture}[scale=1.5]
\filldraw[draw=black, fill=gray!60] (0,0) 
  -- (1,1)
  -- (2,1)
  -- (1,0)
  -- cycle;
\draw (1,0.5) node {$\times$}; 
\draw [dashed] (1,0.5) -- (1,1); 
\draw (1,0.5) node {$\times$}; 
\draw [dashed] (1,0.5) -- (1,0); 
\draw (-.2,0) node[below] {$\scriptstyle{(-2R_1,-R_1)}$};
\draw (1,1) node[above] {$\scriptstyle{(0,R_1)}$};
\draw (2,1) node[above] {$\scriptstyle{(2R_1,R_1)}$};
\draw (1,0) node[below] {$\scriptstyle{(0,-R_1)}$};
\end{tikzpicture}
\caption{}
\label{fig:gen-spins_non-simple}
  \end{subfigure}
\caption{Four examples of marked semitoric polygon representatives associated to the generalized coupled angular momenta system for different choices of parameters. Note that in the bottom right example there are two marked points which coincide, one with an upwards cut and one with a downwards cut (four representatives of a similar polygon appear in Figure~\ref{fig:st-poly-non-simple}).}
\end{center}
\end{figure}

\section{Semitoric families}
\label{sec:st-families}

In this section, we give a quick overview of the techniques and results of the papers~\cite{LFPal,LFPal-SF2}, which focus on certain parameter-dependent families of integrable systems and techniques to construct explicit examples.
The plan for this section is:
\begin{itemize}[nosep]
    \item In Section~\ref{sec:stfam-motivate} we motivate the strategy we will employ in this section to construct semitoric integrable systems;
    \item In Section~\ref{sec:stfam-defs} we define the key objects of this section, one-parameter families of integrable systems called \emph{semitoric families} and \emph{semitoric transition families};
    \item In Section~\ref{sec:stfam-invariants} we discuss the behavior of semitoric invariants as the parameter varies, and discuss the first new constructions of integrable systems with these techniques;
    \item In Section~\ref{sec:stfam-beyond}, we discuss the limitations of the techniques from Section~\ref{sec:stfam-invariants}, and explain how to circumvent these difficulties to obtain even more new examples, such as an example on $\mathbb{CP}^2$ which transitions between being toric, semitoric, and hypersemitoric (see Section~\ref{sec:hst}) depending on the parameter;
    \item In Section~\ref{sec:strategy7}, we discuss in slightly more detail the strategies we used to come up with the explicit examples discussed throughout this section, and we point the reader to several useful results in~\cite[Section 7]{LFPal-SF2};
    \item In Section~\ref{sec:min-models}, we discuss the semitoric minimal model program, and how the techniques from this section were successfully used to complete the list of an explicitly constructed semitoric system corresponding to each strictly minimal polygon;
    \item Finally, in Section~\ref{sec:stfam-coupled-spins-again}, following~\cite{HoPa2017}, we take a certain two-parameter subset of the parameter space of the generalized coupled angular momenta system (Section~\ref{sec:generalized-spins}) to see it from the point of view of semitoric families.
\end{itemize}

\subsection{Motivation}
\label{sec:stfam-motivate}

The papers~\cite{LFPal,LFPal-SF2} cover many topics, but one of their main motivations is the following general question:

\begin{question}\label{Q:SF-general}
    Given a list (or partial list) of semitoric invariants, when can we construct an explicit semitoric system $(M,\om,(J,H))$ with those invariants?
\end{question}

The classification results of~\cite{PeVN2009,PeVN2011} and~\cite{PPT-nonsimple} (Theorems~\ref{thm:PVN} and~\ref{thm:non-simple}) prove that given any admissible list of invariants, there exists a unique semitoric system with those invariants.
The process to construct the system from the invariants, described in \cite{PeVN2009,PeVN2011} for generic semitoric systems and generalized to all semitoric systems in~\cite{PPT-nonsimple}, proceeds by constructing local pieces of the system and delicately gluing them together.
This proves that the system exists, and gives some insight into its structure (since these pieces are well-understood), but does not provide much help in finding a completely explicit description of the system in terms of global functions on familiar manifolds.
Compare this with the construction of a toric integrable system from the associated Delzant polytope via Delzant's construction~\cite{De1988}, reviewed in Section~\ref{sec:delz-construction}, which is extremely explicit. Of course, this added complexity
in the construction reflects the additional complexity of semitoric systems themselves, and in fact such global functions on familiar manifolds might not exist - so in some sense this extra difficulty is unavoidable.
On the other hand, in some cases it should be possible to write down a system from its invariants in a simple form.

More specifically, in~\cite{LFPal,LFPal-SF2}, we attacked a more precise version of Question~\ref{Q:SF-general}:

\begin{question}\label{question2}
    Given a marked semitoric polygon, find an explicit semitoric system with that as its marked semitoric polygon invariant.
\end{question}

The technique that we develop is heavily inspired by the coupled angular momenta system, discussed above in Example~\ref{ex:coupledspins}. Recall that in this example there is a family of integrable systems depending on a parameter, but the underlying symplectic manifold $(M,\om)$ and the Hamiltonian $J$ generating the $S^1$-action are both independent of the parameter, so only the other integral (the one generating the non-periodic flow) depends on the parameter.
That is, it is a one-parameter family of the form $(M,\om,F_t = (J,H_t))$ for $0\leq t \leq 1$.

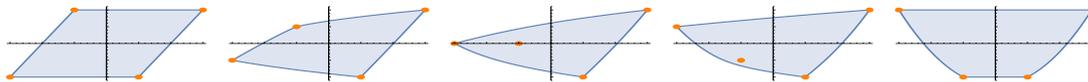
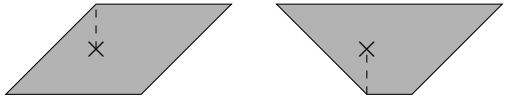
\begin{figure}[ht]
\centering
\begin{subfigure}{.95\linewidth}
\centering
\scalebox{1}[.6]{\includegraphics[width=420pt]{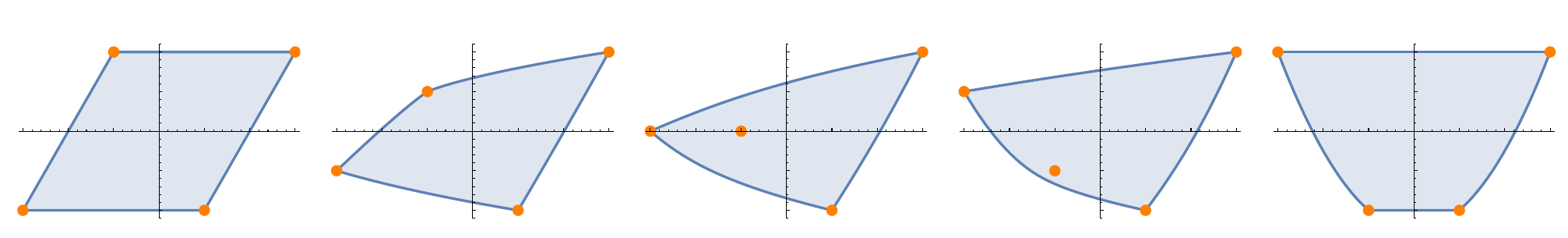}}
\caption{The momentum map image for the system as $t$ increases from $t=0$ to $t=1$.}
\label{fig:momentummapimage}
\end{subfigure}
\centering
\begin{subfigure}{.95\linewidth}
\centering
\begin{tikzpicture}[scale=.6]
\filldraw[draw=black, fill=gray!60] (0,0) node[anchor=north,color=black]{}
  -- (2,2) node[anchor=south,color=black]{}
  -- (5,2) node[anchor=south,color=black]{}
  -- (3,0) node[anchor=north,color=black]{}
  -- cycle;
\draw [dashed] (2,1) -- (2,2);
\draw (2,1) node[] {$\times$};
\begin{scope}[xshift = -2cm]
\filldraw[draw=black, fill=gray!60] (8,2) node[anchor=north,color=black]{}
  -- (13,2) node[anchor=south,color=black]{}
  -- (11,0) node[anchor=north,color=black]{}
  -- (10,0) node[anchor=south,color=black]{}
  -- cycle;
\draw [dashed] (10,0) -- (10,1);
\draw (10,1) node[] {$\times$};
\end{scope}
\end{tikzpicture}
\caption{Two representatives of the semitoric polygon of the system when $t=1/2$.}
\label{fig:stpoly}
\end{subfigure}
\caption{The momentum map image and semitoric polygons for the coupled angular momenta. Recall that the system is semitoric with one focus-focus point when $t=1/2$, and notice that the $t=0$ and $t=1$ momentum map images look similar to the semitoric polygons
 of the $t=1/2$ system.}
\label{fig:coupledangular}
\end{figure}

Figure~\ref{fig:coupledangular} shows the actual moment images of $F_t(M)$ as $t$ increases from $t=0$ to $t=1$, and compares them with two representatives of the semitoric polygon for the $t=\frac{1}{2}$ system.
The comparison of these two images is extremely leading: the images $F_0(M)$ and $F_1(M)$ look very similar to the two representatives of the semitoric polygon for the $t=\frac{1}{2}$ system\footnote{The $t=0$ is exactly equal to the polygon on the left, but the $t=1$ image is not a polygon (the edges are a bit bent), so it only resembles the polygon on the right.}.
This suggests that a semitoric system with a desired marked semitoric polygon invariant can, at least in some cases, be obtained by interpolating between systems related to representatives of the marked semitoric polygon invariant.
In other words, we'd like to perform this process backwards: start with a desired semitoric polygon and use two of its representatives to guess systems for parameters $t=0$ and $t=1$, somehow extrapolate this into a family of systems for $0\leq t \leq 1$, and (hopefully) find that the system for $t=\frac{1}{2}$ is semitoric with the desired marked semitoric polygon.

In fact, this technique has proven effective. Families of integrable systems of this type already existed in the coupled angular momenta system~\cite{SaZh1999,LFP} and its generalization~\cite{HoPa2017}, the general theory was developed in~\cite{LFPal,LFPal-SF2}, and the techniques developed therein were used in, for instance, \cite{HohMeu,GuHo}, to obtain new examples of integrable systems with certain desired polygons or configurations of focus-focus points.
One of the achievements of this program is the construction of an explicit semitoric system corresponding to each of the so-called ``strictly minimal'' semitoric polygons. Examples of these systems came from several papers~\cite{SaZh1999,LFP,HoPa2017,LFPal,LFPal-SF2}, and the list was finally completed in~\cite{LFPal-SF2}.
We discuss the semitoric minimal model program in Section~\ref{sec:min-models}.

\subsection{Semitoric families and semitoric transition families}
\label{sec:stfam-defs}

In this section we will describe some of the foundational definitions from~\cite{LFPal,LFPal-SF2}. A \emph{fixed $S^1$-family} is a family of integrable systems \[(M,\om,F_t=(J,H_t)) \text{ for }0\leq t \leq 1\] such that $\dim(M)=4$, $\mathcal{X}_{J}$ generates an effective $S^1$-action, and the map from $[0,1]\times M$ to $\R$ given by $(t,p)\mapsto H_t(p)$ is smooth.

\begin{definition}[{\cite[Definition 1.4]{LFPal}}]
    A \emph{semitoric family} with \emph{degenerate times} $t_1,\ldots,t_k\in[0,1]$ is a fixed-$S^1$ family $(M,\om,F_t)$ which is semitoric if and only if $t\notin\{t_1,\ldots,t_k\}$.
\end{definition}

As discussed above, semitoric families are useful for constructing explicit examples, but they are interesting in their own right as well.
To construct the examples we are interested in, it is necessary to understand how certain aspects of the system change when passing through the degenerate times and undergoing various bifurcations.
In~\cite[Section 3.2.2]{LFPal}, we describe the possible scenarios that can occur in a semitoric family around a degenerate time, and give examples of explicit systems which exhibit these behaviors. 
This is a small step towards the general goal of understanding bifurcations of integrable systems in the presence of a group action.

Let us now define the unmarked semitoric polygon invariant. Informally, it is exactly the data of the marked semitoric polygon invariant, except without the information of the height of each marked point (the horizontal position of each marked point is important though, since it impacts the corner conditions).
By~\cite[Lemma 3.7]{LFPal} the number of focus-focus points and the unmarked semitoric polygon invariant can only change at degenerate times in a semitoric family. To have a better understanding of how they change, we have to consider a more restrictive type of family.

\begin{definition}
    A \emph{semitoric transition family} with \emph{transition point} $p\in M$ and \emph{transition times} $t^-,t^+\in (0,1)$ is a semitoric family with degenerate times $t^-,t^+$ such that:
    \begin{itemize}
        \item $t^-<t^+$;
        \item for $t<t^-$ and $t>t^+$, the transition point $p$ is an elliptic-elliptic singular point;
        \item for $t^-<t<t^+$, the transition point $p$ is a focus-focus singular point;
        \item for $t\in\{t^-,t^+\}$, there are no degenerate singular points in $M\setminus\{p\}$;
        \item if $p$ is a maximum (resp.~minimum) of $H_0 |_{J^{-1}(J(p))}$, then $p$ is a minimum (resp.~maximum) of $H_1 |_{J^{-1}(J(p))}$.
    \end{itemize}
\end{definition}

The idea of the above definition is that as $t$ increases from $0$ to $1$, the value $F_t(p)$ starts on the top (or bottom) of the moment map image at an elliptic-elliptic value, passes though a degeneracy to become focus-focus, travels across the interior of the moment image as a focus-focus value, and then passes through a degeneracy again to become an elliptic-elliptic value on the opposite side of the moment image. This is modeling exactly the situation of the coupled angular momentum (see Example~\ref{ex:coupledspins} and Figure~\ref{fig:momentummapimage}).

\subsection{Invariants in semitoric families and first constructions}
\label{sec:stfam-invariants}

Now we want to understand how the invariants change with $t$ in a semitoric transition family.
For this section, fix a semitoric transition family $(M,\om,F_t)$, $0\leq t \leq 1$.
As mentioned above, by~\cite[Lemma 3.7]{LFPal} the number of focus-focus points and the unmarked semitoric polygon can only change at the degenerate times $t^-$ and $t^+$, and by the definition it's clear how the number of focus-focus points changes (it increases by 1 at $t^-$ and decreases by 1 at $t^+$). Let $t_0 \in (t^-,t^+)$. Since the unmarked semitoric polygon can only change at the two degenerate times, we can now compare the invariant for $t=0$, $t=t_0$, and $t=1$ to get a full understanding of the unmarked semitoric polygon for all $t\in[0,1]\setminus\{t^-,t^+\}$. Furthermore, suppose that the image of the transition point starts on the top boundary of the image of $M$ when $t=0$, and ends up on the bottom boundary when $t=1$.

To avoid having to define a lot of new notation, we will explain the result of~\cite[Lemma 3.14]{LFPal} in words. Recall that the $t=t_0$ system has one extra marked point, compared to $t=0$ and $t=1$, so there is one more cut which can be chosen to be up or down. Consider the set of all unmarked semitoric polygons for $t=t_0$ for which that cut is going upwards, and delete the cut from each one (but leave $\De$ unchanged, which will necessarily now have a vertex where the cut used to meet the top boundary). The resulting collection is the unmarked semitoric polygon invariant for the $t=0$ system. Similarly, to get the polygons for the $t=1$ system, start with the collection of polygons for $t=t_0$, choose the polygons which have the cut going downwards, and then delete the cut. See Figure~\ref{fig:poly-in-family}.

Roughly speaking, the main idea is that the set of unmarked semitoric polygons for the intermediate system ($t=t_0$) can be identified with the union of those polygons for the $t=0$ system with those polygons for the $t=1$ system, except that one of the cuts has been removed. Again, see Figure~\ref{fig:poly-in-family}.

\begin{figure}
    \centering
    \begin{tikzpicture}[scale = .75]

\draw[->] (4.4,1) -- (6.6,1);
\draw[->] (11.4,1) -- (13.6,1);

\draw (2,3.75) node[] {$t<t^-$};
\draw (9,3.75) node[] {$t^-<t<t^+$};
\draw (16,3.75) node[] {$t>t^+$};

\filldraw[draw=black, fill=gray!60] (0,0) node[anchor=north,color=black]{}
  -- (0,2) node[anchor=south,color=black]{}
  -- (2,2) node[anchor=south,color=black]{}
  -- (4,0) node[anchor=north,color=black]{}
  -- cycle;
\begin{scope}[xshift = 7cm, yshift = 1cm]
\filldraw[draw=black, fill=gray!60] (0,0) node[anchor=north,color=black]{}
  -- (0,2) node[anchor=south,color=black]{}
  -- (2,2) node[anchor=south,color=black]{}
  -- (4,0) node[anchor=north,color=black]{}
  -- cycle;
\draw [dashed] (2,1) -- (2,2);
\draw (2,1) node[] {$\times$};
\end{scope}

\begin{scope}[xshift = 7cm, yshift = -1.5cm]
\filldraw[draw=black, fill=gray!60] (0,0) node[anchor=north,color=black]{}
  -- (0,2) node[anchor=south,color=black]{}
  -- (4,2) node[anchor=south,color=black]{}
  -- (2,0) node[anchor=north,color=black]{}
  -- cycle;
\draw [dashed] (2,0) -- (2,1);
\draw (2,1) node[] {$\times$};
\end{scope}

\begin{scope}[xshift = 14cm]
\filldraw[draw=black, fill=gray!60] (0,0) node[anchor=north,color=black]{}
  -- (0,2) node[anchor=south,color=black]{}
  -- (4,2) node[anchor=south,color=black]{}
  -- (2,0) node[anchor=north,color=black]{}
  -- cycle;
\end{scope}
\end{tikzpicture}
    \caption{The marked semitoric polygons in a semitoric transition family can change at the degenerate times, and the (unmarked) polygons for $t^-<t<t^+$ are essentially the union of the set of semitoric polygons for $t<t^-$ and $t>t^+$. As $t$ increases from $t=t^-$ to $t=t^+$ in the above family, the marked point will travel from the top to the bottom of the polygon.}
    \label{fig:poly-in-family}
\end{figure}
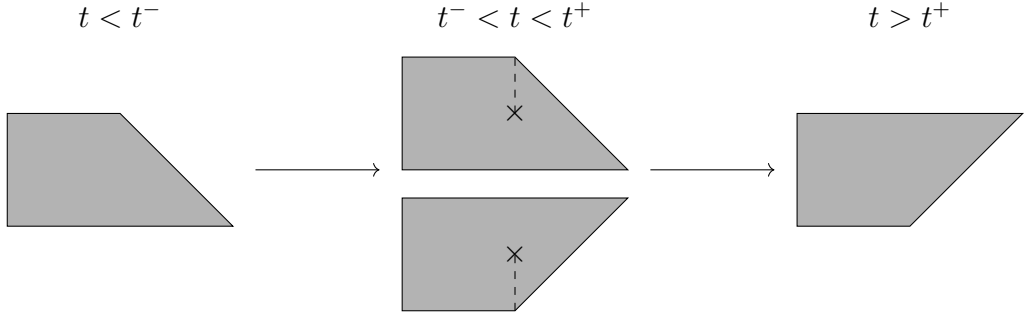

This result is in agreement with our experience with the coupled angular momenta system, shown in Figure~\ref{fig:coupledangular}. For $t<t^-$ and $t>t^+$, the system is a semitoric system with zero focus-focus points, which is called a \emph{toric type system}\footnote{For $t=0$ the coupled angular momenta system is actually honestly toric, not just toric type, but of course any toric system is also a toric type system.}. This means that the moment image can be straightened out to achieve a polygon, so it is no surprise that the images for $t=0$ and $t=1$ look similar to the polygons into which they can be straightened.

Given what we now know about the invariants of the systems in semitoric transition families, we can form a plan: in order to obtain a semitoric system with one focus-focus point and some given desired semitoric polygon, the strategy is to attempt to extrapolate between toric type systems whose polygons correspond to two of the choices of semitoric polygon (one with an upwards cut and one with a downwards cut). In~\cite{LFPal}, this strategy is used to obtain new examples of semitoric integrable systems, in families, on the first and second Hirzebruch surfaces which have various behavior (see Theorem 1.9 in that paper).
The techniques and results of~\cite{LFPal} were then applied to construct a semitoric system $(M,\om,(J,H_t))$ in~\cite{HohMeu} which has, for certain values $t$, four focus-focus points, and in particular has two double-pinched tori (as in~\cite{PPT-nonsimple}) for $t=1/2$.
More details about the strategy are discussed in Section~\ref{sec:strategy7}.


\subsection{Beyond semitoric families}
\label{sec:stfam-beyond}

Looking back at the description of~\cite[Lemma 3.14]{LFPal} above, we notice a subtle issue. 
Recall that representatives of the unmarked semitoric polygons for $t<t^-$ and $t>t^+$ are obtained by erasing cuts from representatives of the unmarked semitoric polygon for $t^-<t<t^+$.
As we saw in Section~\ref{sec:marked-st-poly}, the conditions for vertices on a cut are different than the conditions for those not on a cut, and so it may be that removing a cut yields an object which is no longer an unmarked semitoric polygon.
In fact, Figure~\ref{fig:st-poly-triangle} shows a valid marked semitoric polygon representative (described in Example~\ref{ex:st-poly-triangle}), but erasing the upwards cut produces the polygon from Figure~\ref{fig:not-delzant}, which does not satisfy the Delzant condition at its top vertex.

Thus, the result of~\cite[Lemma 3.14]{LFPal} implies that any semitoric system whose marked semitoric polygon invariant is the one generated by the representative from Example~\ref{ex:st-poly-triangle} \emph{cannot appear in a semitoric transition family}. If it did appear in such a family, then Lemma 3.14 implies that the polygons of either the $t=0$ or $t=1$ system in the family would not satisfy the conditions to be a semitoric polygon, which is impossible.

Therefore, while the techniques of~\cite{LFPal} were able to produce several new interesting examples, these techniques need to be generalized in order to achieve certain polygons. 
The obstruction that prohibits certain semitoric polygons from appearing in a semitoric system in a semitoric family is related to the relationship between the transition point and points of $M$ for which the $S^1$-action generated by $J$ has non-trivial isotropy, see~\cite[Section 5]{LFPal-SF2} for full details.
In short, if $(M,\om,F)$ is semitoric then any point in $M$ with non-trivial $S^1$-isotropy (points in so-called $\Z_k$-spheres, c.f.~\cite{karshon}) must get mapped to the boundary of $F(M)$, and this prohibits a focus-focus value in the interior from merging with the boundary there via a semitoric family.
The polygon from Example~\ref{ex:st-poly-triangle} in particular is an important example of a marked semitoric polygon, since it is one of the strictly minimal polygons (see Section~\ref{sec:min-models}), so we have to find a technique that works in this case.

Since the problem, very roughly speaking, occurs when the focus-focus value in the interior of $F(M)$ collides with the boundary of the moment image, to avoid this trouble we simply need to relax the requirement in semitoric transition families that the focus-focus value travels all of the way across. Furthermore, it turns out to be useful to allow non-semitoric systems for some values of $t$.

More specifically, a \emph{half semitoric transition family} with \emph{transition point $p\in M$} is a fixed-$S^1$ family $(M,\om,F_t)$ such that $p$ is semitoric for $t<t^-$ and $t^-<t<t^+$, but not necessarily for $t>t^+$. 
In particular, this allows us to avoid the situation that the transition point must interact with the singular points getting mapped to both the top and bottom boundaries of $F_t(M)$, since it is precisely on these boundaries that there may exist obstructions to the typical Hamiltonian-Hopf bifurcation that occurs in a semitoric transition family when the transition point changes from being focus-focus to elliptic-elliptic.
This allows for systems which become, for instance, \emph{hypersemitoric} when $t>t^+$. Hypersemitoric systems allow a wider variety of singular points compared to semitoric ones (they allow both hyperbolic blocks and certain well-behaved degenerate points), and in particular admit the existence of a configuration of singular points known as a \emph{flap}~\cite{EG-cusps, dullin-pelayo,HP-extend}. We discuss hypersemitoric systems more in Section~\ref{sec:hst}.

With these new concepts, and with some new strategies for constructing integrable systems (see Section~\ref{sec:strategy7}), we were able to obtain a completely explicit half semitoric transition family $(\mathbb{CP}^2,\om,F_t=(J,H_t))$, for $0\leq t \leq 1$, such that:
\begin{itemize}
    \item when $0\leq t < t^-$, the system is \textbf{semitoric with no focus-focus points} (i.e.~it is of \textbf{toric type}) ;
    \item when $t^-<t<t^+$, the system is \textbf{semitoric with one focus-focus point} (at $B=[0:0:1]$) with semitoric polygon as in Figure~\ref{fig:CP2-semitoric-poly};
    \item when $t^+<t\leq 1$, the system is \textbf{hypersemitoric} with one triangular flap with elliptic corner $F_t(B)$.
\end{itemize}
This system is described and studied in detail in~\cite[Theorem 8.1]{LFPal-SF2}.
Sketches of the Delzant polygon for the $t=0$ system and representatives of the marked semitoric polygon for $t\in (t^-,t^+)$ are shown in Figures~\ref{fig:CP2-toric-poly} and~\ref{fig:CP2-semitoric-poly}, respectively.
These techniques were later applied in~\cite{GuHo} to produce various examples of semitoric and hypersemitoric systems.

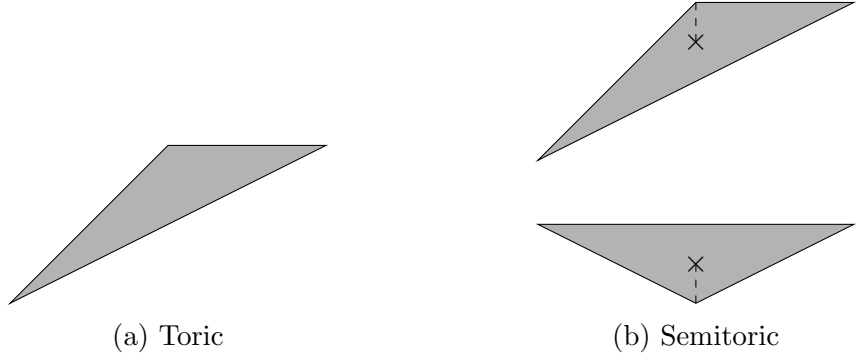
\begin{figure}
\begin{center}
\begin{subfigure}[b]{.30\linewidth}
\centering
\begin{tikzpicture}[scale=2.1]
\filldraw[draw=black, fill=gray!60] (0,0) 
  -- (1,1)
  -- (2,1)
  -- cycle;
\end{tikzpicture}
\caption{Toric}
\label{fig:CP2-toric-poly}
\end{subfigure}\qquad\quad
\begin{subfigure}[b]{.40\linewidth}
\centering
\begin{tikzpicture}[scale=2.1]
\filldraw[draw=black, fill=gray!60] (0,0) 
  -- (1,1)
  -- (2,1)
  -- cycle;
\draw (1,0.75) node {$\times$}; 
\draw [dashed] (1,0.75) -- (1,1); 
\begin{scope}[yshift = -40pt]
    \filldraw[draw=black, fill=gray!60] (0,1) 
  -- (1,.5)
  -- (2,1)
  -- cycle;
\draw (1,0.75) node {$\times$}; 
\draw [dashed] (1,0.75) -- (1,.5); 
\end{scope}
\end{tikzpicture}
\caption{Semitoric}
\label{fig:CP2-semitoric-poly}
  \end{subfigure}
\caption{Polygons related to the system discussed in Section~\ref{sec:stfam-beyond}. The Delzant polygon for the (toric) $t=0$ system is shown on the left and two of the representatives of the marked semitoric polygon for a choice of $t$ with $t^-<t<t^+$ (at which time the system is semitoric) are shown on the right.}
\end{center}
\end{figure}

\subsection{Strategies for constructing semitoric families and half-semitoric families}
\label{sec:strategy7}

The strategy for constructing these families is more than simply looking at the polygons. 
Typically, it is easiest to take at $t=0$ a system which is toric (not just toric type) and therefore can be constructed by Delzant's construction, as in Section~\ref{sec:delz-construction}. Since only the $H_t$ in a fixed-$S^1$ family $(M,\om,(J,H_t))$ depends on the parameter $t$, having a construction for $t=0$ determines the symplectic manifold $(M,\om)$ and the Hamiltonian $J$ generating the $S^1$-action.
After this, it only remains to come up with the function $H_t$ for $0<t\leq 1$, and in~\cite{LFPal,LFPal-SF2} there are strategies to do this.

The interested reader should consult Section 7 of~\cite{LFPal-SF2}, which describes in great detail the tools needed to develop these systems. Among other things, in~\cite[Section 7]{LFPal-SF2} we do the following:
\begin{itemize}
    \item Starting from the normal form of the Hamiltonian $S^1$-action generated by $J$ around a fixed point with weights $\{+1,-1\}$, we show how to write all of the possible local $H$ such that $(J,H)$ is integrable and $p$ is a critical point. We write this $H$ in terms of the coordinates of the normal form of $J$, and in term of the parameters in this description it is possible to immediately read off the type of the singular point (either elliptic-elliptic, focus-focus, or degenerate). We also discuss how changing these parameters can induce a Hamiltonian-Hopf bifurcation (in which a point changes between elliptic-elliptic and focus-focus type, by passing through a degeneracy). This helps with choosing a local expression for $H$, which in some cases can be extended to an appropriate global $H$.
    \item In Section 7 we also develop techniques to simplify verifying that a given candidate for a semitoric system is actually semitoric - which can otherwise be a lot of work, since in principle all singular points need to be checked to see if they are degenerate or hyperbolic.
\end{itemize}

We believe that the tools introduced in Section 7 of~\cite{LFPal-SF2} will be useful for anyone trying to construct or understand four-dimensional integrable systems which have an underlying $S^1$-action.

\subsection{The semitoric minimal model program}
\label{sec:min-models}

The techniques of~\cite{LFPal,LFPal-SF2} help with coming up with an explicit system from a given polygon, but this is still not a precise algorithm as in Delzant's construction for toric systems.
Thus, the most reasonable next step is to apply these techniques to certain foundational examples of marked semitoric polygons, and hopefully obtain explicit examples.
We therefore applied our strategy to a class of marked semitoric polygons from which all marked semitoric polygons can be constructed by a sequence of certain operations called corner and wall chops. These fundamental examples are the marked semitoric polygons called strictly minimal.

It is well-known that blowups can be performed on symplectic toric manifolds, and at the level of the polytope this amounts to performing an operation known as a \emph{corner chop}. This operation can be generalized to semitoric polygons.
In Figure~\ref{fig:corner-chop} we show the effect of the operation on an example of a marked semitoric polygon. Notice that the corner chop is impacted by the cut in the top representative, but in the bottom representative the part chopped off of the corner does not intersect the cut. This is the idea behind the operation: in at least one representative it behaves in the same way as the corner chop on toric systems, and that determines the impact on all other representatives. See Section 4 of~\cite{LFPal} for full details, and a description of how this is induced by a certain type of equivariant blowup on the manifold.
With this definition in hand, the results of~\cite{KPP2015,KPP2018} can be used to obtain a characterization of all \emph{minimal} semitoric polygons. That is, those polygons which can not be obtained from another marked semitoric polygon by a corner chop. Unfortunately, it isn't easy to use this characterization to obtain a concrete list of minimal polygons, and there seem to be too many marked semitoric polygons that are minimal in this sense (there are many minimal marked semitoric polygons that do not seem particularly fundamental).

\begin{figure}
\begin{center}
\begin{tikzpicture}[scale = .80]
\draw [->] (3,-0.5) -- (4,-0.5);
\draw [->] (8,-0.5) -- (9,-0.5);

\draw[decoration={brace,raise=5pt, amplitude = 5pt},decorate] (0,-3.2) -- (0,2.2);
\draw[decoration={brace,raise=5pt, amplitude = 5pt},decorate] ( 5,-3.2) -- (5,2.2);
\draw[decoration={brace,raise=5pt, amplitude = 5pt},decorate] ( 10,-3.2) -- (10,2.2);

\filldraw[draw=black, fill=gray!60] (0,0) node[anchor=north,color=black]{}
  -- (0,2) node[anchor=south,color=black]{}
  -- (1,2) node[anchor=south,color=black]{}
  -- (2,1) node[anchor=north,color=black]{}
  -- (2,0) node[anchor=north,color=black]{}
  -- cycle;
\draw [dashed] (1,1) -- (1,2);
\draw (1,1) node[] {$\times$};

\filldraw[draw=black, fill=gray!60] (5,0) node[anchor=north,color=black]{}
  -- (5,2) node[anchor=south,color=black]{}
  -- (6,2) node[anchor=south,color=black]{}
  -- (7,1) node[anchor=north,color=black]{}
  -- (7,0) node[anchor=north,color=black]{}
  -- cycle;
\draw [dashed] (6,1) -- (6,2);
\draw (6,1) node[] {$\times$};
  
\filldraw[draw=black, fill=gray!60, pattern=north east lines] (5,0.5) node[anchor=north,color=black]{}
  -- (5,2) node[anchor=south,color=black]{}
  -- (6,2) node[anchor=south,color=black]{}
  -- (6.5,1.5) node[anchor=north,color=black]{}
  -- (6,1.5) node[anchor=north,color=black]{}
  -- cycle;
  
\filldraw[draw=black, fill=gray!60] (10,0) node[anchor=north,color=black]{}
  -- (10,0.5) node[anchor=south,color=black]{}
  -- (11,1.5) node[anchor=south,color=black]{}
  -- (11.5,1.5) node[anchor=south,color=black]{}
  -- (12,1) node[anchor=north,color=black]{}
  -- (12,0) node[anchor=north,color=black]{}
  -- cycle;
\draw [dashed] (11,1) -- (11,1.5);
\draw (11,1) node[] {$\times$};
  
\filldraw[draw=black, fill=gray!60] (0,-3) node[anchor=north,color=black]{}
  -- (0,-1) node[anchor=south,color=black]{}
  -- (2,-1) node[anchor=north,color=black]{}
  -- (2,-2) node[anchor=north,color=black]{}
  -- (1,-3) node[anchor=south,color=black]{}
  -- cycle;
\draw [dashed] (1,-2) -- (1,-3);
\draw (1,-2) node[] {$\times$};

\filldraw[draw=black, fill=gray!60] (5,-3) node[anchor=north,color=black]{}
  -- (5,-1) node[anchor=south,color=black]{}
  -- (7,-1) node[anchor=north,color=black]{}
  -- (7,-2) node[anchor=north,color=black]{}
  -- (6,-3) node[anchor=south,color=black]{}
  -- cycle;
\draw [dashed] (6,-2) -- (6,-3);
\draw (6,-2) node[] {$\times$};
  
\filldraw[draw=black, fill=gray!60, pattern=north east lines] (5,-2.5) node[anchor=north,color=black]{}
  -- (5,-1) node[anchor=south,color=black]{}
  -- (6.5,-1) node[anchor=north,color=black]{}
  -- (6,-1.5) node[anchor=north,color=black]{}
  -- cycle;
  
\filldraw[draw=black, fill=gray!60] (10,-3) node[anchor=north,color=black]{}
  -- (10,-2.5) node[anchor=south,color=black]{}
  -- (11.5,-1) node[anchor=south,color=black]{}
  -- (12,-1) node[anchor=north,color=black]{}
  -- (12,-2) node[anchor=north,color=black]{}
  -- (11,-3) node[anchor=south,color=black]{}
  -- cycle;  
\draw [dashed] (11,-2) -- (11,-3);
\draw (11,-2) node[] {$\times$};

\end{tikzpicture}
\end{center}
\caption{An example of performing a corner chop on a marked semitoric polygon, showing two representatives. On the left are the original representatives, in the middle we have shaded in the area to be removed, and on the right are the resulting representatives after performing the corner chop. See~\cite[Section 4]{LFPal} for more details.}
\label{fig:corner-chop}
\end{figure}
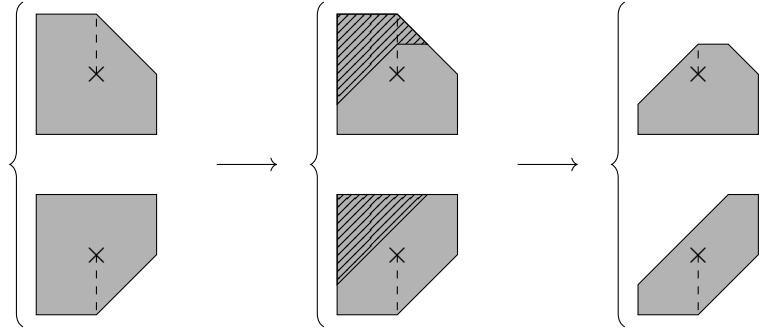

There is another operation on marked semitoric polygons that can further reduce the list of minimal models, though. This operation on polygons, called a \emph{wall chop} in~\cite{LFPal-SF2}, is the impact on the marked polygon invariant of an operation on the semitoric system called a \emph{semitoric blowup}, first studied by Zung~\cite{Zung03} and Symington~\cite{Sy2003} for almost toric fibrations, and being further explored in the semitoric case by Hohloch-Sabatini-Sepe-Symington in an upcoming project~\cite{HSSS-st-lift}. 
Even without working directly on the integrable system, it is possible to simply rely on the description in terms of the marked semitoric polygon and apply Corollary~\ref{cor:semitoric-polygons} to recover a system from the polygon if necessary. 
This is the strategy taken in~\cite[2.10.2]{HP-extend} and~\cite[3.1.4]{LFPal-SF2}, where this operation is described in detail and studied by working only with the polygon invariant\footnote{The upcoming paper by Hohloch-Sabatini-Sepe-Symington~\cite{HSSS-st-lift} is expected to work out the details of this operation on the manifold, which is more delicate (and rewarding) than simply considering the polygon as we do here. For the purposes of~\cite{HP-extend} and~\cite{LFPal-SF2}, knowing the impact on the polygon invariant suffices, but a better understanding of the operation geometrically on the manifold is desirable.}.
In~\cite[Theorem 4.8]{LFPal-SF2}, the authors use the results of~\cite{KPP2015,KPP2018} and the wall chop operation to obtain a finite list of classes of those marked semitoric polygons which are \emph{strictly minimal}, in the sense that they cannot be obtained from another marked semitoric polygon by either the corner chop or wall chop operation. Representatives of all minimal polygons are shown in Figure~\ref{fig:min-polys}.
We refer the reader to~\cite{LFPal} for a detailed description of the corner chop (and associated blowup at the level of the manifold) and to~\cite{HP-extend,LFPal-SF2} for a discussion of the wall-chop operation on the polygon.

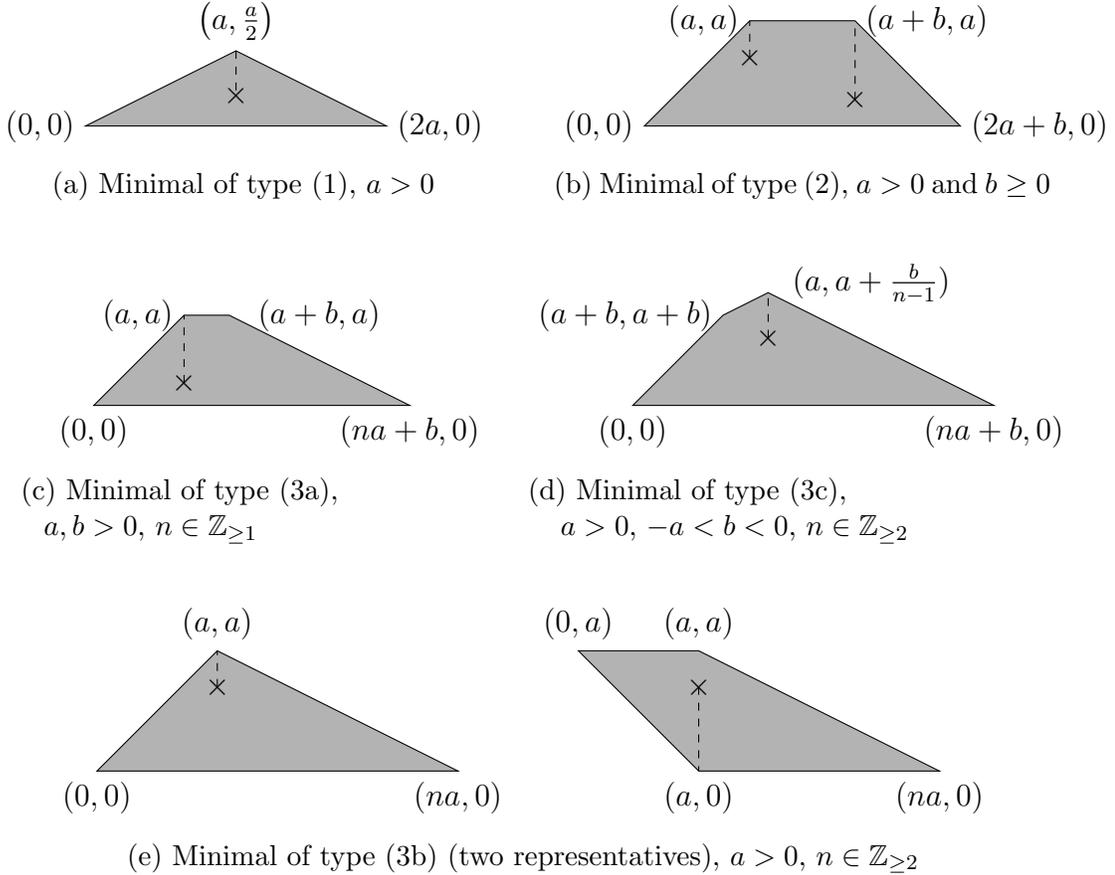
\begin{figure}[ht]
    \centering
    \begin{subfigure}{.4\textwidth}
    \centering
    \begin{tikzpicture}[scale = 1]
        \draw[draw=black, fill=gray!60] (0,0) -- (2,1) -- (4,0) -- cycle;
        \node at (2,0.4) {$\times$};
        \draw [dashed](2,1)--(2,0.4);
        \node at (0,0) [anchor = east] {$(0,0)$};
        \node at (2,1) [anchor = south] {$\left(a,\frac{a}{2}\right)$};
        \node at (4,0) [anchor = west] {$(2a,0)$};
    \end{tikzpicture}
    \caption{Minimal of type (1), $a>0$}
    \label{fig:stmin-1}
    \end{subfigure}\qquad
    \begin{subfigure}{.4\textwidth}
    \centering
    \begin{tikzpicture}[scale = .7]
        \draw[draw=black, fill=gray!60] (0,0) -- (2,2) -- (4,2) -- (6,0) -- cycle;
        \node at (2,1.3) {$\times$};
        \draw [dashed](2,1.3)--(2,2);
        \node at (4,.5) {$\times$};
        \draw [dashed](4,.5)--(4,2);
        \node at (0,0) [anchor = east] {$(0,0)$};
        \node at (2,2) [anchor = east] {$(a,a)$};
        \node at (4,2) [anchor = west] {$(a+b,a)$};
        \node at (6,0) [anchor = west] {$(2a+b,0)$};
    \end{tikzpicture}
    \caption{Minimal of type (2), $a>0$ and $b\geq 0$}
    \label{fig:stmin-2}
    \end{subfigure}\\[1.5em]
    \begin{subfigure}{.4\textwidth}
    \centering
    \begin{tikzpicture}[scale = .6]
        \draw[draw=black, fill=gray!60] (0,0) -- (2,2) -- (3,2) -- (7,0) -- cycle;
        \node at (2,.5) {$\times$};
        \draw [dashed](2,.5)--(2,2);
        \node at (0,0) [anchor = north] {$(0,0)$};
        \node at (2,2) [anchor = east] {$(a,a)$};
        \node at (3.4,2) [anchor = west] {$(a+b,a)$};
        \node at (7,0) [anchor = north] {$(na+b,0)$};
    \end{tikzpicture}
    \caption{Minimal of type (3a),\\ \vspace{1pt}\hspace{5pt} $a,b>0$, $n\in\Z_{\geq 1}$}
    \label{fig:stmin-3a}
    \end{subfigure}
    \begin{subfigure}{.4\textwidth}
    \centering
    \begin{tikzpicture}[scale = .6]
        \draw[draw=black, fill=gray!60] (0,0) -- (2,2) -- (3,2.5) -- (8,0) -- cycle;
        \node at (3,1.5) {$\times$};
        \draw [dashed](3,1.5)--(3,2.5);
        \node at (0,0) [anchor = north] {$(0,0)$};
        \node at (2,2) [anchor = east] {$(a+b,a+b)$};
        \node at (3.3,2.7) [anchor = west] {$(a,a+\frac{b}{n-1})$};
        \node at (8,0) [anchor = north] {$(na+b,0)$};
    \end{tikzpicture}
    \caption{Minimal of type (3c),\\ \vspace{1pt} \hspace{5pt} $a>0$, $-a<b<0$, $n\in\Z_{\geq 2}$}
    \label{fig:stmin-3c}
    \end{subfigure}\\[1.5em]
    \begin{subfigure}{.9\textwidth}
        \centering
        \begin{tikzpicture}[scale = .8]
        \draw[draw=black, fill=gray!60] (0,0) -- (2,2) -- (6,0) -- cycle;
        \node at (2,1.4) {$\times$};
        \draw [dashed](2,1.4)--(2,2);
        \node at (0,0) [anchor = north] {$(0,0)$};
        \node at (2,2) [anchor = south] {$(a,a)$};
        \node at (6,0) [anchor = north] {$(na,0)$};
        \begin{scope}[shift = {(8,0)}]
            \draw[draw=black, fill=gray!60] (0,2) -- (2,2) -- (6,0) -- (2,0) -- cycle;
            \node at (2,1.4) {$\times$};
            \draw [dashed](2,1.4)--(2,0);
            \node at (0,2) [anchor = south]  {$(0,a)$};
            \node at (2,2) [anchor = south] {$(a,a)$};
            \node at (6,0) [anchor = north] {$(na,0)$};
            \node at (2,0) [anchor = north] {$(a,0)$};
        \end{scope}
    \end{tikzpicture}
    \caption{Minimal of type (3b) (two representatives), $a>0$, $n\in\Z_{\geq 2}$}
    \label{fig:stmin-3b}  
    \end{subfigure}
    \caption{The minimal polygons from~\cite[Theorem 4.8]{LFPal-SF2}. All corners are fake or Delzant, except for the hidden corner at the top of the left representatives of (3b). If $b=0$ in the polygon of type (2), we obtain a non-simple system.}
    \label{fig:min-polys}
\end{figure}

By definition, all marked semitoric polygons can be obtained from a strictly minimal marked semitoric polygon by performing a sequence of corner chops and wall chops, so these polygons are particularly foundational. Between the papers~\cite{LFPal,LFPal-SF2,HoPa2017,LFP}, the techniques described in this section regarding families of integrable systems are used to obtain an explicit example for each strictly minimal semitoric polygon, a result which is summarized in~\cite[Theorem 1.11]{LFPal-SF2}.
Obtaining explicit examples for these minimal models was an important motivation behind the search for examples that led to the development of semitoric families and related concepts.

\subsection{The generalized coupled angular momenta as a two-parameter family}
\label{sec:stfam-coupled-spins-again}

Recall the generalized coupled angular momenta system from Section~\ref{sec:generalized-spins}, introduced in~\cite{HoPa2017}, with parameters $R_1,R_2,t_1,t_2,t_3,t_4\in\R$. For $s_1,s_2\in [0,1]$ take the parameters
\[
R_1 =1, \quad R_2=2, \quad t_1 = (1-s_2)(1-s_2), \quad t_2 = s_1 s_2, \quad t_3 = s_1 + s_2 - 2 s_1 s_2, \quad t_4 = s_1-s_2.
\]
Then we obtain a family of integrable systems which depends on two parameters. The image of the momentum map for varying $s_1,s_2$ is shown in Figure~\ref{fig:2FF-image}. As with the original coupled angular momenta system, the systems for which $s_1,s_2\in\{0,1\}$ look similar to the semitoric polygons which can be obtained for the system in the middle, with $s_1=s_2=\frac{1}{2}$. This subsystem, and the similarity between the images with $s_1,s_2\in\{0,1\}$ and the semitoric polygon, were already considered in the original paper~\cite{HoPa2017}, but we present it in this section since it can be best understood in the context of semitoric families.

\begin{figure}
    \centering
    \includegraphics[width=0.75\linewidth]{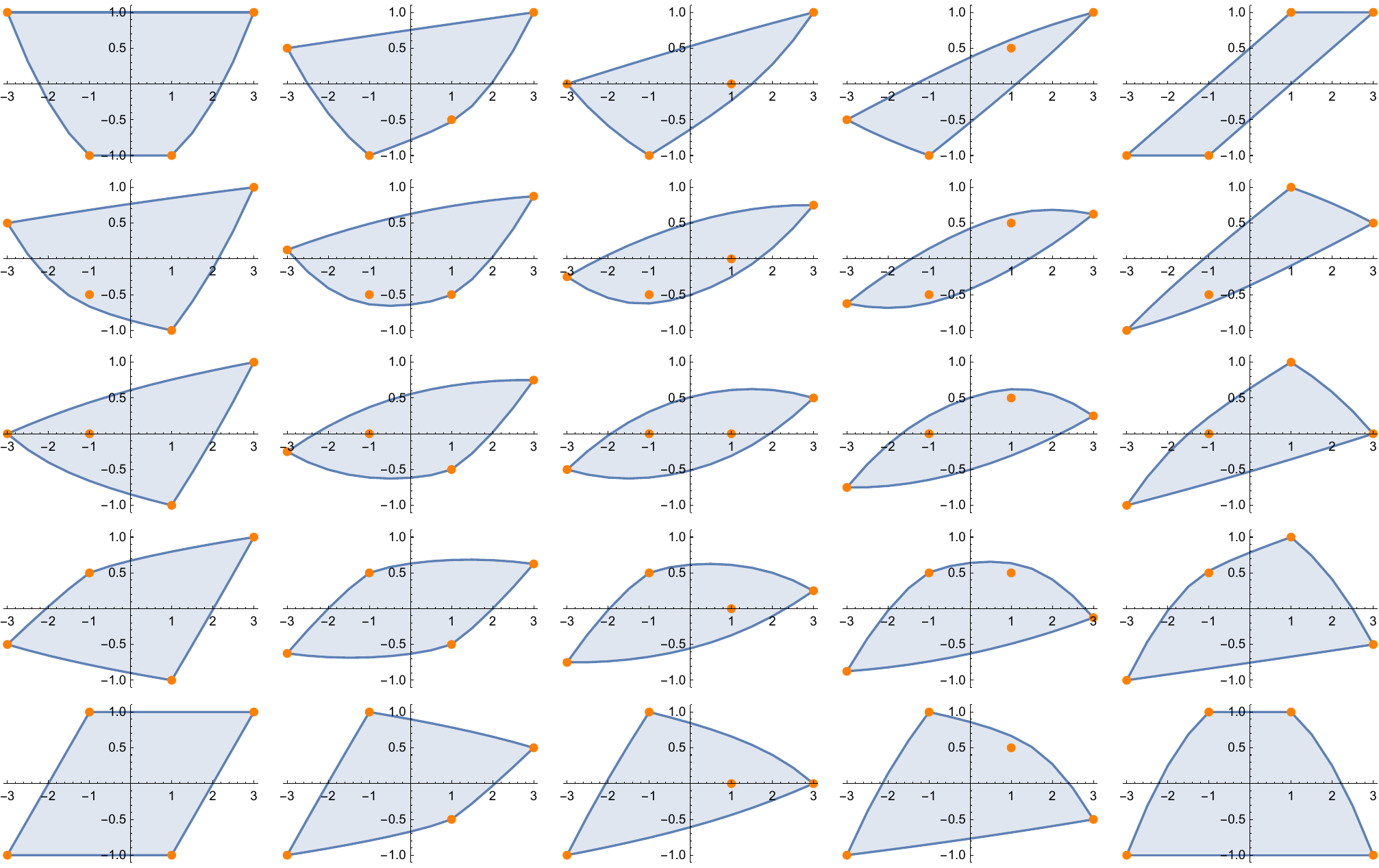}
    \caption{The image of the momentum map of the generalized coupled angular momenta system for various choices of parameters $s_1$ and $s_2$. Compare the systems in the corners with the representatives of the semitoric polygon shown in Figure~\ref{fig:st-poly-2FF}.}
    \label{fig:2FF-image}
\end{figure}

\section{Related topics and future directions}
\label{sec:closing}

There are many important and interesting aspects of this subject which we have either barely touched on or unfortunately not yet discussed at all in this paper.
In fact, there are far too many to even give an exhaustive list.
In this last section, we will point out a few interesting developments and results which are deeply related to integrable systems with torus symmetries but are not discussed above.
The descriptions here are brief, but we include references to more detailed accounts.
The last subsection is devoted to some speculation as to where the field is heading in the near future.

\subsection{\texorpdfstring{Integrable systems and underlying $S^1$-spaces}{Integrable systems and underlying S1-spaces}}
\label{sec:close-big-S1-section}

Let $(M,\om)$ be a compact symplectic $4$-manifold, and suppose that $J\colon M\to\R$ generates an $S^1$-action. Then $(M,\om,J)$ is called a \emph{Hamiltonian $S^1$-space}, often just called an \emph{$S^1$-space} for brevity. Karshon~\cite{karshon} introduced $S^1$-spaces and gave a classification of such spaces in terms of a certain type of weighted labeled graph, now called the \emph{Karshon graph} of the $S^1$-space.
If $(M,\om, (J,H))$ is a toric integrable system, then $(M,\om,J)$ is a $S^1$-space, and in that same paper Karshon explains how to determine the graph of $(M,\om,J)$ from the Delzant polytope of $(M,\om,(J,H))$. Furthermore, she shows that not all $S^1$-spaces arise in this way, and even gives conditions under which an $S^1$-space can be extended to a toric system. That is, given a Hamiltonian $S^1$-space $(M,\om,J)$ she gives necessary and sufficient conditions under which there exists a smooth $H\colon M\to\R$ such that $(M,\om,(J,H))$ is a toric integrable system. 

Hamiltonian $S^1$-spaces are now a central object in symplectic geometry, and this paper of Karshon's explains their interaction with toric integrable systems. 
There has been recent progress in understanding the interaction between the Hamiltonian $S^1$-space $(M,\om,J)$ and integrable systems, not necessarily toric, of the form $(M,\om,(J,H))$ for some $H$, which we review in this section.

\subsubsection{\texorpdfstring{Underlying $S^1$-spaces of semitoric systems}{Underlying S1-spaces of semitoric systems}}\label{sec:close-underlyingS1}

Let $(M,\om, (J,H))$ be a compact semitoric integrable system. Then $(M,\om,J)$ is a Hamiltonian $S^1$-space.
In~\cite{HSS2015}, the authors show how to read the Karshon graph of $(M,\om,J)$ off  of the (unmarked) semitoric polygon of $(M,\om, (J,H))$. The technique is very similar to Karshon's for the toric case~\cite{karshon}, but accounts for the added complexity arising from the cuts and fake or hidden corners in the polygon. It is worth noting that they show how to obtain the Karshon graph from \emph{only the semitoric polygon invariant} without needing to take the twisting index or Taylor series into consideration. Recall that the Karshon graph classifies Hamiltonian $S^1$-spaces up to $S^1$-invariant symplectomorphisms intertwining the momentum maps.
Thus, not only do Hohloch-Sabatini-Sepe provide a useful procedure for producing the Karshon from the semitoric polygon, they also obtain the following useful corollary:

\begin{corollary}[{Follows from~\cite[Theorem 3.1]{HSS2015}}]\label{cor:S1-type}
    Let $(M_i,\om_i,(J_i,H_i))$ be a compact semitoric system for $i=1,2$, and suppose that the unmarked semitoric polygon invariant is the same for these two systems. Then there exists an $S^1$-equivariant symplectomorphism $\Phi\colon M_1\to M_2$ such that $\Phi^* J_2 = J_1$.
\end{corollary}

Along the way, their paper implies that certain Karshon graphs cannot be obtained from any semitoric system in this way.
The same authors, along with M.~Symington, are currently working on a project to go the other direction. That is, to determine under what conditions on $(M,\om,J)$ does there exist an $H$ such that $(M,\om,(J,H))$ is semitoric~\cite{HSSS-st-lift}.

\subsubsection{Lifting to hypersemitoric systems}\label{sec:close-hst-lift}

In view of the observation from~\cite{HSS2015} that not all Hamiltonian $S^1$-spaces can be lifted to a semitoric system, it is natural to wonder if there is a larger, but still relatively nice, class of integrable systems to which \emph{all} Hamiltonian $S^1$-spaces can be lifted. 
In~\cite{HP-extend}, the authors observe that even allowing all non-degenerate singularities is not enough to lift all possible $S^1$-spaces: there exist examples of $S^1$-spaces $(M,\om,J)$ such that any integrable system $(M,\om,(J,H))$ must have at least one degenerate singular point.
Therefore, in~\cite{HP-extend}, a new class of integrable systems called hypersemitoric was introduced, which allows all non-degenerate points and also a certain relatively mild type of degenerate point called parabolic, and the following result was established:

\begin{theorem}[{\cite[Theorem 1.7]{HP-extend}}]\label{thm:hst-lift}
    Let $(M,\om,J)$ be any (compact) Hamiltonian $S^1$-space. Then there exists a smooth function $H\colon M\to\R$ such that $(M,\om,(J,H))$ is a hypersemitoric system.
\end{theorem}

See Section~\ref{sec:hst} for a discussion of hypersemitoric systems and parabolic points.

\subsection{Almost toric fibrations and symplectic topology}
\label{sec:close-ATFs-symp-top}

One of the important properties of semitoric systems is the existence of the global $S^1$-action on the system (see Section~\ref{sec:close-big-S1-section}).
That being said, there are actually a lot of interesting things to say for systems which have the same restrictions on singularities that semitoric systems have, but without the requirement of the $S^1$-action.
These are called \emph{almost toric fibrations}~\cite{Sy2003}.

\subsubsection{Almost toric fibrations}
\label{sec:ATFs}

In~\cite{Sy2003}, Symington discusses an object called an \emph{almost toric fibration} (ATF) which exhibits the same types of singularities as semitoric systems (elliptic-elliptic, elliptic-regular, and focus-focus), but has no requirement of a global $S^1$-action.
In this case, she is still able to obtain a polygon, called an \emph{almost toric base diagram}, by performing cuts and straightening out the integral affine structure, but in general the cuts that she performs will not all be parallel (as they are in semitoric systems). An example is shown in Figure~\ref{fig:atf}.

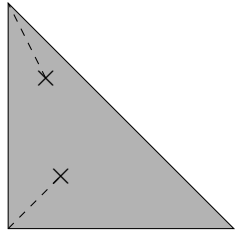
\begin{figure}
\begin{center}
\begin{tikzpicture}[scale=1.0]
\filldraw[draw=black, fill=gray!60] (0,0) node[anchor=north,color=black]{}
  -- (0,0) node[anchor=south,color=black]{}
  -- (0,3) node[anchor=south,color=black]{}
  -- (3,0) node[anchor=south,color=black]{}
  -- cycle;
\draw [dashed] (0,0) -- (.7,.7);
\draw (.7,.7) node[] {$\times$};
\draw [dashed] (.5,2) -- (0,3);
\draw (.5,2) node[] {$\times$};  
\end{tikzpicture}
\end{center}
\caption{An example of an almost toric base diagram for a fibration with two focus-focus points.}
\label{fig:atf}
\end{figure}

Around each focus-focus point there is a local $S^1$-action, and while these actions can be glued together into a global $S^1$-action for semitoric systems, this is not the case in general for ATFs. Note that any ATF over a disk with exactly one focus-focus point can always be viewed as a semitoric integrable system $(M,\om,(J,H))$, choosing coordinates so that $J$ generates the $S^1$-action.

ATFs were introduced in~\cite{Sy2003} and in that paper many basic properties of these fibrations were established, and they were classified up to fiber preserving diffeomorphisms in~\cite{LeuSym}. Recently, a classification up to fiber preserving symplectomorphisms has appeared~\cite{tang-almosttoric}.
ATFs have become a key tool to approach a wide variety of problems in symplectic topology~\cite{Via,Viainf,Viadel,EvaSmi,PasTon,CasVia,BreSch,AA-Via-Vic-singLagFib}.
Semitoric polygons were used to study the (non)-displacibility properties of the fibers of the momentum map in~\cite{HS-nondisplace}.

\subsubsection{Symplectic capacities}\label{sec:capacities}
Semitoric polygons~\cite{VN2003}, and more generally the polygons which appear in Symington's study almost toric fibrations~\cite{Sy2003}, can be used to develop combinatorial techniques for computing, or at least obtaining bounds for, certain symplectic capacities.
Equivariant capacities were introduced in~\cite{FiPaPe2016}, and the examples provided in that paper include capacities measuring how much volume of a symplectic manifold can be filled by an equivariantly embedded ball (or disjoint unions of equivariantly embedded  balls). 
These examples were computed in some cases in~\cite{packingIGL}.
This was preceded by the papers of Pelayo~\cite{Pe2006, Pe2007} and Pelayo-Schmidt~\cite{PeSc2008}, which studied torus-equivariant packing-type invariants of toric integrable systems.
More recently, Delzant polytopes were also used to obtain bounds on the Hofer-Zehnder capacity of a symplectic toric manifold~\cite{Liu-HZ-bounds}.

Let us now give a very rough idea of how equivariant capacities can be used to obtain interesting information (usually bounds) about traditional capacities.
Essentially, the idea is that equivariantly embedded balls appear as certain types of simplices in the Delzant polytope, semitoric polygon, or almost toric base diagram, depending on the situation. Then, the maximum size of the embedded ball can be computed by considering the convex geometry of the associated combinatorial object. Examples of traditional (non-equivariant) capacities that are computed by finding the largest symplectically embedded object (ball, polydisk, ellipsoid, etc), can therefore be \emph{bounded} by noting that the equivariant picture gives an example of such an embedded object, even though it may not be the largest once the equivariant condition is removed.

A natural way to proceed, then, is to consider a sequence of different ATFs on the same symplectic manifold, and obtain a sequence of bounds on the capacity at hand.
This is essentially the technique used in~\cite{CGHMP,CasVia,MMW-staircasepatterns, macgill, four-periodic} to study what is called ``infinite staircase'' behavior in capacities related to the largest possible size of embedded ellipsoid. The infinite staircase behavior was first observed and studied by McDuff-Schlenk~\cite{MS12}.

\subsection{Quantum considerations}
\label{sec:close-quantum}

This entire article has focused on classical integrable systems, which model classical mechanical systems, but there is an entire deep theory of the quantum counterpart to these systems, and many interesting questions about the interactions between the classical and quantum worlds. See, for instance, the papers~\cite{VNSepe,Pelayo-survey23} for surveys of both the classical and quantum theory.

Explaining this in any detail is outside of the scope of the current paper, but let us just mention that the \emph{inverse spectral problem} was one of the original motivations for the classification of simple semitoric systems~\cite{PeVN2009,PeVN2011}. 
This question asks if a classical integrable system can be recovered from the semiclassical joint spectra of the quantum counterpart.
One possible way to approach this problem is to recover certain invariants of the classical system from the semiclassical spectrum, and then prove that those invariants completely determine the classical system.
This was the strategy taken in~\cite{ChPeVN2013}, in which the 
authors show how to recover the Delzant polytope from the semiclassical spectrum of a quantum toric integrable system, and then apply the classification of toric systems (Theorem~\ref{thm:toric-classification}) to conclude that this determines the classical system. 
Generalizing this technique to simple semitoric systems necessitates a classification of such systems, which was provided in~\cite{PeVN2009,PeVN2011} in terms of five invariants.
Therefore, the question that remained was: can the five semitoric invariants be recovered from the semiclassical spectrum of a simple quantum semitoric system? In~\cite{LFPeVN2016}, it was proved that the first four invariants (all except for the twisting index) could be recovered from the spectrum, and in~\cite{LFVN} it was proved that all five invariants could be recovered.
Combining this with the classification result, this implies that a simple semitoric system is completely determined by the semiclassical spectrum of the associated quantum system.
Moreover, the paper~\cite{LFVN} also provides explicit algorithms to recover the simple semitoric invariants from the spectrum.
This strategy is explained in detail, along with several other topics in this area, in~\cite{PeVNfirststeps,Pelayo-inverse}.

Non-simple semitoric systems make an interesting appearance here. It has \emph{not} been proven that non-simple systems can be recovered from the semiclassical spectrum of the associated quantum system, and in fact the expectation is that non-simple semitoric systems are not spectrally determined, see~\cite[Section 8.2]{PPT-nonsimple}.

\subsection{Moving beyond semitoric}
\label{sec:beyond}

So far we have discussed toric systems (Section~\ref{sec:toric}) and semitoric systems (Sections~\ref{sec:st} and~\ref{sec:st-families}). In recent years, there has been great success with classifying, understanding, and constructing semitoric systems, so it is time to move further. There are two things to notice about semitoric systems:
\begin{itemize}
    \item there are restrictions on the types of singularities which can arise (hyperbolic blocks are excluded, along with any degenerate singularities); and
    \item they are confined to dimension four.
\end{itemize}
It's important to note that both hyperbolic singularities and certain degenerate singularities are common in physical systems, and of course many physically relevant systems have dimension greater than four. Thus, this gives us two ways to generalize beyond semitoric systems: allowing more types of singularities  or allowing higher dimensions, or both.
We discuss these two options briefly in the next two sections, Sections~\ref{sec:hst} and~\ref{sec:close-cmlx-one}.
There are many ways to expand on the class of allowed singularities compared to semitoric systems, but in Section~\ref{sec:hst} we focus on hypersemitoric systems because that is one of the most active areas at the moment.
Another interesting class of integrable systems are those which are \emph{non-degenerate}, i.e.~all singular points in the system are non-degenerate in the sense of Theorem~\ref{thm:eliasson}, and in particular those which are non-degenerate and also have an underlying complexity one torus action, as in Section~\ref{sec:close-cmlx-one}.

\subsubsection{Hypersemitoric systems}
\label{sec:hst}

Beyond the non-degenerate singular points described in Theorem~\ref{thm:eliasson}, there are many types of degenerate singular points, and one of the most well-understood classes are those known as \emph{parabolic}~\cite{bol-parabolic,EG-cusps,KudMar-circle,KudMar-invariants}. 
Parabolic points are common in physical examples, such as the Kovalevskaya top~\cite{BolRikFom-Ktop}.
Hyperbolic-regular points appear in one-parameter families, and there are only a few options for the behavior of an endpoint of one of these families. In particular, in the presence of a global $S^1$-action the only options are hyperbolic-elliptic points (which must occur on a fixed surface of the $S^1$-action) or degenerate singular points, which are typically parabolic.
Since they occur in natural examples, are relatively well behaved, and often appear with hyperbolic regular points, it is reasonable to consider the class of systems which are allowed to have parabolic singularities in addition to any non-degenerate singularity:

\begin{definition}[\cite{HP-extend}]
An integrable system $(M,\om,(J,H))$ is \emph{hypersemitoric} if
\begin{itemize}[noitemsep]
    \item The Hamiltonian flow of $J$ generates an $S^1$-action;
    \item $J$ is proper; and
    \item all singular points of the system are either non-degenerate or of parabolic type.
    \end{itemize}
\end{definition}

As mentioned in Section~\ref{sec:close-hst-lift}, there exist Hamiltonian $S^1$-spaces that cannot be lifted to a semitoric system, or indeed to any system without any degenerate singular points, and hypersemitoric systems were originally developed to be the ``nicest'' class of integrable systems to which all Hamiltonian $S^1$-spaces can be lifted~\cite{HP-extend}.
Since their introduction in the context of lifting $S^1$-spaces, there have been many papers studying or making use of hypersemitoric systems~\cite{LFPal-SF2,GuHo,GuHo-firststeps,HHM,dullin-pelayo}.
They are a natural candidate for classification, extending the semitoric classification from~\cite{PeVN2009,PeVN2011,PPT-nonsimple}.
In particular, a polygon type invariant has already been constructed for these systems~\cite{EHS-hyp-poly}.

For a simple example of a hypersemitoric system, consider the manifold $M=S^2 \times \T^2$ with symplectic form $\om$ given by the product of the area forms. Let $J$ denote the height function on the sphere. Then $J$ generates an $S^1$-action, but it can be seen that this cannot be extended to a $\T^2$-action (since the genus of the symplectic reduction is not zero, see~\cite{karshon}) or to a semitoric system (for the same reason, see~\cite{HSS2015}). Let $H$ denote the height function of the torus standing on its end (the one with one index 0, two index 1, and one index 2 points). Then $(M,\om,F=(J,H))$ is a hypersemitoric system. 
All singularities are non-degenerate, since they can be written as the product of two Morse functions.
The index 1 points of the height function on the torus produce hyperbolic blocks in the singularities of $F$, and for intermediate values of $H$ the fibers of $F$ are disconnected. See Figure~\ref{fig:S2-times-T2} and~\cite[Example 3.6]{HP-extend}.

\begin{figure}
\centering
\begin{tikzpicture}
\filldraw[thick, fill = gray!60] (0,0) node[anchor = north,color = black]{}
  -- (0,3) node[anchor = south,color = black]{}
  -- (3,3) node[anchor = south,color = black]{}
  -- (3,0) node[anchor = north,color = black]{}
  -- cycle;
\draw [thick] (0,2.25) -- (3,2.25);
\draw [thick] (0,.75) -- (3,.75);
\draw (0,.75) node[anchor=east]{HE};
\draw (0,2.25) node[anchor=east]{HE};
\draw (3,.75) node[anchor=west]{HE};
\draw (3,2.25) node[anchor=west]{HE};
\draw (1.5,.75) node[anchor=south]{HR};
\draw (1.5,2.25) node[anchor=south]{HR};
\draw (0,0) node[anchor=north]{EE};
\draw (0,3) node[anchor=south]{EE};
\draw (3,3) node[anchor=south]{EE};
\draw (3,0) node[anchor=north]{EE};
\draw[black,fill=black] (0,0) circle (.05cm);
\draw[black,fill=black] (0,.75) circle (.05cm);
\draw[black,fill=black] (0,2.25) circle (.05cm);
\draw[black,fill=black] (0,3) circle (.05cm);
\draw[black,fill=black] (3,0) circle (.05cm);
\draw[black,fill=black] (3,.75) circle (.05cm);
\draw[black,fill=black] (3,2.25) circle (.05cm);   
\draw[black,fill=black] (3,3) circle (.05cm);
\end{tikzpicture}
\caption{The image $F(M)$ of the hypersemitoric system on $S^2\times \T^2$. The points are labeled by their type: HR (hyperbolic-regular), HE (hyperbolic-elliptic), and or EE (elliptic-elliptic). The points in the boundary that are not labeled EE or HE are all elliptic-regular, and the other unlabeled points are all regular. The preimage of any of the regular values in the central region (between the two lines of HR points) is diffeomorphic to the disjoint union of two copies of $\T^2$.}
\label{fig:S2-times-T2}
\end{figure}
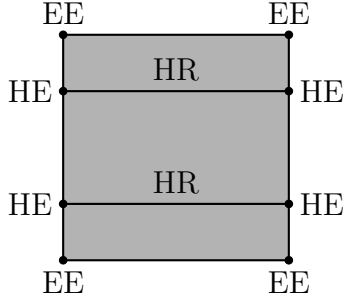

\subsubsection{Complexity one integrable systems}
\label{sec:close-cmlx-one}

In Sections~\ref{sec:st} and~\ref{sec:st-families}, we  focused on four-dimensional integrable systems of the form $(M,\om,F=(J,H))$ in which the flow of $\mathcal{X}_J$ generates an $S^1$-action. That is, we are considering integrable systems $(M,\om,(J,H))$ which have an \emph{underlying $S^1$-space} $(M,\om,J)$. 

Let $(M,\om)$ be a symplectic manifold of dimension $2n$. Then a \emph{complexity one torus action} on $(M,\om)$ is an effective Hamiltonian action of $\T^{n-1}$, and for such an action we call $(M,\om,\T^{n-1},\Phi)$ a \emph{complexity one space}, where $\Phi\colon M\to\mathfrak{t}^*$ is the momentum map of the action. For example, a Hamiltonian action of $S^1$ on a symplectic 4-manifold is complexity one.
Complexity one spaces are a huge area of research~\cite{SepeTolman,KT-top-of-quotients,SabSepe-top-properties,HolmKessler-equivcoh,HolmKessler-combo}, and for instance there have been several important invariants defined, and even a classification for a broad class of such spaces~\cite{KaTo-centered, KT-tall-invariants,KT-tall-classification}. In a complexity one space, $\Phi(M)\subset \mathfrak{t}^*$ is a polytope, and similar to the other classifications discussed in this paper, the complete invariant is obtained by labeling this polytope with extra data.

Now, associate $\mathfrak{t}^*$ with $\R^{n-1}$, and by a slight abuse of notation let $\Phi\colon M\to \R^{n-1}$. Then it turns out that the components of $\Phi$ are automatically independent almost everywhere and in Poisson involution, so, in particular, $\Phi$ can be thought of as being only one function short of being an integrable system.  
A \emph{complexity one integrable system} is an integrable system $(M,\om, F=(\Phi,f))$ where $(M,\om,\Phi)$ is a complexity one space.

As in the case of semitoric systems in dimension four, complexity one integrable systems (which are more general than semitoric, even in dimension four) strike a delicate balance: the underlying complexity one space provides enough rigidity that these systems are well behaved, while the freedom in the last component of the momentum map ensures that there are a wide variety of examples. Note that here we haven't included any restrictions on the types of singularities that can occur, as in the four-dimensional case it would be reasonable to start by considering systems with relatively nice singularities (non-degenerate, or even more restrictive) and expand to cases with a wider variety of allowed singularities.
For instance, recently Sepe-Tolman~\cite{SepeTolman} obtained conditions under which a complexity one integrable system with only non-degenerate singularities has connected fibers, generalizing known results about the connectedness of fibers in toric~\cite{Atiyah,GuSt1982,LMTW} and semitoric~\cite{VN2003,wac-thesis} systems.

With the notion of complexity one integrable systems in mind, all of the work that has been completed in dimension four can also be explored in higher dimensions: classifications, lifting problems (to various types of integrable systems), and translating between invariants of integrable systems and invariants of complexity one spaces, just to name a few possible directions. 
Given a complexity one space $(M,\om,\Phi)$, families of the form $(M,\om,(\Phi,f_t))$ for $0\leq t \leq 1$, generalizing the semitoric families discussed in Section~\ref{sec:st-families}, would be a fascinating subject to explore (considering bifurcations of such systems, for instance).
There will be significant difficulties in pursuing these questions in higher dimensions, but this direction promises to be fruitful, and it is a necessary step to attack natural questions in integrable systems.

\subsection{Looking forward}
\label{sec:close-looking-forward}

The classification of toric integrable systems from the late 1980s is a beautiful and useful result, but toric systems are much more symmetric than most examples of integrable systems. Semitoric systems are already a huge step towards more generality, and appear in many natural systems. They are much more complicated than toric systems, but nevertheless there has by now been significant progress in understanding, classifying, constructing, and investigating the bifurcation theory of such systems.
As semitoric systems and semitoric families are becoming better understood, it is a natural time to take the next steps, which we believe will consist of continuing to explore generalizations of semitoric systems, such as hypersemitoric systems, non-degenerate complexity one systems, and more general classes of complexity one systems.
In particular, we expect the following topics to continue to generate a large amount of activity:
\begin{enumerate}
    \item \textbf{Classifications:} Invariants and classifications of hypersemitoric systems, non-degenerate systems with underlying complexity one torus actions, and more general classes of complexity one systems;
    \item \textbf{The Lifting Question:} Conditions on when a complexity one action can be lifted an integrable system of a desired type, such as toric, non-degenerate, or higher dimensional analogues of semitoric;
    \item \textbf{Bifurcation Theory:} Bifurcation theory and deformations of families of the form $(M,\om,(\Phi,f_t))$, ${0\leq t \leq 1}$, where $(M,\om,\Phi)$ is a complexity one space, and in general bifurcation theory of integrable systems with symmetries;
    \item \textbf{Singular Integral Affine Structures:} Development of a rigorous notion of a singular integral affine structure on the base of the (singular Lagrangian) fibration determined by an integrable system, and determining conditions under which this structure does or does not determine the integrable system up to isomorphism. 
    See Remark~\ref{rmk:int-affine-question}.
\end{enumerate}
There is already a considerable body of research on semitoric systems, and adapting all of these techniques and results to the more general scenarios outlined above promises many opportunities for further exploration and discovery and, as usual, calls for more work.

\bibliographystyle{alpha}
\bibliography{ref}

\newcommand{\etalchar}[1]{$^{#1}$}
\begin{thebibliography}{CGHMP25}

\bibitem[AAVV25]{AA-Via-Vic-singLagFib}
S.~Achig-Andrango, R.~Vianna, and R.~Vicente.
\newblock On singular lagrangian fibrations and applications to symplectic embeddings {I}.
\newblock Preprint, \url{https://arxiv.org/abs/2506.01556}, 2025.

\bibitem[ADH19]{ADH-spin-osc}
J.~Alonso, H.R. Dullin, and S.~Hohloch.
\newblock Taylor series and twisting-index invariants of coupled spin-oscillators.
\newblock {\em J. Geom. Phys.}, 140:131--151, 2019.

\bibitem[ADH20]{ADH-momenta}
J.~Alonso, H.R. Dullin, and S.~Hohloch.
\newblock Symplectic classification of coupled angular momenta.
\newblock {\em Nonlinearity}, 33(1):417--468, 2020.

\bibitem[AH19]{AH-survey}
J.~Alonso and S.~Hohloch.
\newblock Survey on recent developments in semitoric systems.
\newblock {\em RIMS Kokyuroku}, 2137, 2019.

\bibitem[AH21]{AH21}
J.~Alonso and S.~Hohloch.
\newblock The {{height invariant}} of a {{four-parameter semitoric system}} with {{two focus}}\textendash{{focus singularities}}.
\newblock {\em Journal of Nonlinear Science}, 31(3):51, April 2021.

\bibitem[AHP25]{AHP-twist}
J.~Alonso, S.~Hohloch, and J.~Palmer.
\newblock The twisting index in semitoric systems.
\newblock {\em Nonlinearity}, 38(3):Paper No. 035018, 45, 2025.

\bibitem[Alo19]{jaume-thesis}
J.~Alonso.
\newblock {\em On the symplectic invariants of semitoric systems}.
\newblock PhD thesis, Universiteit Antwerpen, 2019.

\bibitem[Arn89]{arnold89}
V.~I. Arnol'd.
\newblock {\em Mathematical methods of classical mechanics}, volume~60 of {\em Graduate Texts in Mathematics}.
\newblock Springer-Verlag, New York, second edition, 1989.
\newblock Translated from the Russian by K. Vogtmann and A. Weinstein.

\bibitem[Ati82]{Atiyah}
M.~Atiyah.
\newblock Convexity and commuting {H}amiltonians.
\newblock {\em Bull. London Math. Soc.}, 14(1):1--15, 1982.

\bibitem[BF04]{Bolsinov-Fomenko}
A.~V. Bolsinov and A.~T. Fomenko.
\newblock {\em Integrable {H}amiltonian systems}.
\newblock Chapman \& Hall/CRC, Boca Raton, FL, 2004.
\newblock Geometry, topology, classification, Translated from the 1999 Russian original.

\bibitem[BGK18]{bol-parabolic}
A.~Bolsinov, L.~Guglielmi, and E.~Kudryavtseva.
\newblock Symplectic invariants for parabolic orbits and cusp singularities of integrable systems.
\newblock {\em Philos. Trans. Roy. Soc. A}, 376(2131):20170424, 29, 2018.

\bibitem[BHMM23]{b-semitoric}
J.~Brugu\'es, S.~Hohloch, P.~Mir, and E.~Miranda.
\newblock Constructions of {$b$}-semitoric systems.
\newblock {\em J. Math. Phys.}, 64(7):Paper No. 072703, 30, 2023.

\bibitem[BMMT18]{open-questions}
A.~Bolsinov, V.~Matveev, E.~Miranda, and S.~Tabachnikov.
\newblock Open problems, questions and challenges in finite-dimensional integrable systems.
\newblock {\em Philos. Trans. Roy. Soc. A}, 376(2131):20170430, 40, 2018.

\bibitem[BRF00]{BolRikFom-Ktop}
A.~V. Bolsinov, P.~Rikhter, and A.~T. Fomenko.
\newblock The method of circular molecules and the topology of the {K}ovalevskaya top.
\newblock {\em Mat. Sb.}, 191(2):3--42, 2000.

\bibitem[BS92]{bates-sni-actionangle}
L.~Bates and J.~\'Sniatycki.
\newblock On action-angle variables.
\newblock {\em Arch. Rational Mech. Anal.}, 120(4):337--343, 1992.

\bibitem[BS24]{BreSch}
J.~Brendel and F.~Schlenk.
\newblock Pinwheels as {L}agrangian barriers.
\newblock {\em Commun. Contemp. Math.}, 26(5):Paper No. 2350020, 21, 2024.

\bibitem[CD88]{cushman-duist-pendulum}
R.~{Cushman} and J.J. {Duistermaat}.
\newblock {The quantum mechanical spherical pendulum}.
\newblock {\em {Bull. Am. Math. Soc., New Ser.}}, 19(2):475--479, 1988.

\bibitem[CdS08]{daSi2008}
A.~Cannas~da Silva.
\newblock {\em Lectures on Symplectic Geometry}.
\newblock Springer-Verlag, Berlin, 2008.

\bibitem[CdVV79]{ColVey}
Y.~Colin~de Verdi\`ere and J.~Vey.
\newblock Le lemme de {M}orse isochore.
\newblock {\em Topology}, 18(4):283--293, 1979.

\bibitem[CGHMP25]{CGHMP}
D.~Cristofaro-Gardiner, T.~Holm, A.~Mandini, and A.R. Pires.
\newblock On infinite staircases in toric symplectic four-manifolds.
\newblock {\em J. Differential Geom.}, 129(2):335--413, 2025.

\bibitem[Cha13]{Cha}
M.~Chaperon.
\newblock Normalisation of the smooth focus-focus: a simple proof.
\newblock {\em Acta Math. Vietnam.}, 38(1):3--9, 2013.

\bibitem[CPVN13]{ChPeVN2013}
L.~Charles, \'{A}. Pelayo, and S.~V{\~{u}}~Ng\d{o}c.
\newblock Isospectrality for quantum toric integrable systems.
\newblock {\em Ann. Sci. \'{E}c. Norm. Sup\'{e}r. (4)}, 46(5):815--849, 2013.

\bibitem[CV22]{CasVia}
R.~Casals and R.~Vianna.
\newblock Full ellipsoid embeddings and toric mutations.
\newblock {\em Selecta Math. (N.S.)}, 28(3):Paper No. 61, 62, 2022.

\bibitem[Del88]{De1988}
T.~Delzant.
\newblock {H}amiltoniens p\'{e}riodiques et image convex de l'application moment.
\newblock {\em Bull. Soc. Math. France}, 116:315--339, 1988.

\bibitem[DKL{\etalchar{+}}23]{packingIGL}
Y.~Du, G.~Kosmacher, Y.~Liu, J.~Massman, J.~Palmer, T.~Thieme, J.~Wu, and Z.~Zhang.
\newblock Packing densities of {D}elzant and semitoric polygons.
\newblock {\em SIGMA Symmetry Integrability Geom. Methods Appl.}, 19:Paper No. 081, 42, 2023.

\bibitem[DM91]{DufMol}
J.-P. Dufour and P.~Molino.
\newblock Compactification d'actions de {${\bf R}^n$} et variables action-angle avec singularit\'{e}s.
\newblock In {\em Symplectic geometry, groupoids, and integrable systems ({B}erkeley, {CA}, 1989)}, volume~20 of {\em Math. Sci. Res. Inst. Publ.}, pages 151--167. Springer, New York, 1991.

\bibitem[DMH21]{HohMeu}
A.~De~Meulenaere and S.~Hohloch.
\newblock A family of semitoric systems with four focus-focus singularities and two double pinched tori.
\newblock {\em J. Nonlinear Sci.}, 31(4):Paper No. 66, 56, 2021.

\bibitem[DP16]{dullin-pelayo}
Holger~R. Dullin and {\'A}lvaro Pelayo.
\newblock Generating hyperbolic singularities in semitoric systems via {H}opf bifurcations.
\newblock {\em J. Nonlinear Sci.}, 26(3):787--811, 2016.

\bibitem[Dui80]{Dui88}
J.~J. Duistermaat.
\newblock On global action-angle coordinates.
\newblock {\em Comm. Pure Appl. Math.}, 33(6):687--706, 1980.

\bibitem[Dul13]{dullin13}
H.R. Dullin.
\newblock Semi-global symplectic invariants of the spherical pendulum.
\newblock {\em J. Differential Equations}, 254(7):2942--2963, 2013.

\bibitem[EG12]{EG-cusps}
K.~Efstathiou and A.~Giacobbe.
\newblock The topology associated with cusp singular points.
\newblock {\em Nonlinearity}, 25(12):3409--3422, 2012.

\bibitem[EHS24]{EHS-hyp-poly}
K.~Efstathiou, S.~Hohloch, and P.~Santos.
\newblock On the affine invariant of hypersemitoric systems.
\newblock Preprint, \url{https://arxiv.org/abs/2411.17509}, 2024.

\bibitem[Eli84]{Eliasson-thesis}
L.~H. Eliasson.
\newblock {\em {H}amiltonian systems with Poisson commuting integrals}.
\newblock PhD thesis, University of Stockholm, 1984.

\bibitem[Eli90]{Eli90}
L.~H. Eliasson.
\newblock Normal forms for {H}amiltonian systems with {P}oisson commuting integrals---elliptic case.
\newblock {\em Comment. Math. Helv.}, 65(1):4--35, 1990.

\bibitem[ES18]{EvaSmi}
J.~Evans and I.~Smith.
\newblock Markov numbers and {L}agrangian cell complexes in the complex projective plane.
\newblock {\em Geom. Topol.}, 22(2):1143--1180, 2018.

\bibitem[FHM{\etalchar{+}}25]{four-periodic}
C.~Farley, T.~Holm, N.~Magill, J.~Schroder, M.~Weiler, Z.~Wang, and E.~Zabelina.
\newblock Four-periodic infinite staircases for four-dimensional polydisks.
\newblock {\em Involve}, 18(1):25--78, 2025.

\bibitem[FPP18]{FiPaPe2016}
A.~Figalli, J.~Palmer, and \'A. Pelayo.
\newblock Symplectic {$G$}-capacities and integrable systems.
\newblock {\em Ann. Sc. Norm. Super. Pisa Cl. Sci. (5)}, 18(1):65--103, 2018.

\bibitem[GH22]{GuHo}
Y.~Gullentops and S.~Hohloch.
\newblock Creating hyperbolic-regular singularities in the presence of an {$S^1$}-symmetry.
\newblock Preprint, \url{https://arxiv.org/abs/2209.15631}, 2022.

\bibitem[GH24]{GuHo-firststeps}
Y.~Gullentops and S.~Hohloch.
\newblock Recent examples of hypersemitoric systems and first steps towards a classification: a brief survey.
\newblock In {\em Women in analysis and {PDE}}, Trends Math., pages 187--195. Birkh\"auser/Springer, Cham, [2024] \copyright 2024.

\bibitem[GS82]{GuSt1982}
V.~Guillemin and S.~Sternberg.
\newblock Convexity properties of the moment mapping.
\newblock {\em Invent. Math.}, 67:491--513, 1982.

\bibitem[GS84]{GS-book}
V.~Guillemin and S.~Sternberg.
\newblock {\em Symplectic techniques in physics}.
\newblock Cambridge University Press, Cambridge, 1984.

\bibitem[HHM]{HHM}
T.~Henriksen, S.~Hohloch, and N.~Martynchuk.
\newblock Towards hypersemitoric systems.
\newblock To appear in \textit{Conference Proceedings of RIMS Kokyuroku}, \url{https://arxiv.org/abs/2307.04483}.

\bibitem[HK19]{HolmKessler-equivcoh}
T.~Holm and L.~Kessler.
\newblock The equivariant cohomology of complexity one spaces.
\newblock {\em Enseign. Math.}, 65(3-4):457--485, 2019.

\bibitem[HK25]{HolmKessler-combo}
T.~Holm and L.~Kessler.
\newblock Equivariant cohomology of a complexity-one four-manifold is determined by combinatorial data.
\newblock {\em Adv. Math.}, 480:Paper No. 110497, 2025.

\bibitem[HP18]{HoPa2017}
S.~Hohloch and J.~Palmer.
\newblock A family of compact semitoric systems with two focus-focus singularities.
\newblock {\em Journal of Geometric Mechanics}, 10(3):331--357, 2018.

\bibitem[HP21]{HP-extend}
S.~Hohloch and J.~Palmer.
\newblock Extending compact {H}amiltonian {$S^1$}-spaces to integrable systems with mild degeneracies in dimension four.
\newblock Preprint, \url{https://arxiv.org/abs/2105.00523}, 2021.

\bibitem[HS24]{HS-nondisplace}
S.~Hohloch and P.~Santos.
\newblock (non)displaceability in semitoric systems.
\newblock Preprint, \url{https://arxiv.org/abs/2411.16601}, 2024.

\bibitem[HSS15]{HSS2015}
S.~Hohloch, S.~Sabatini, and D.~Sepe.
\newblock From compact semi-toric systems to {H}amiltonian {$S^1$}-spaces.
\newblock {\em Discrete Contin. Dyn. Syst.}, 35(1):247--281, 2015.

\bibitem[HSSS]{HSSS-st-lift}
S.~Hohloch, S.~Sabatini, D.~Sepe, and M.~Symington.
\newblock From {H}amiltonian $\mathbb{S}^1$-spaces to compact semi-toric systems.
\newblock \emph{In preparation}.

\bibitem[HSSS18]{HSS-vertical}
S.~Hohloch, S.~Sabatini, D.~Sepe, and M.~Symington.
\newblock Faithful semitoric systems.
\newblock {\em SIGMA Symmetry Integrability Geom. Methods Appl.}, 14:Paper No. 084, 66, 2018.

\bibitem[HZ94]{HZ-book}
H.~Hofer and E.~Zehnder.
\newblock {\em Symplectic invariants and {H}amiltonian dynamics}.
\newblock Birkh\"auser Advanced Texts: Basler Lehrb\"ucher. [Birkh\"auser Advanced Texts: Basel Textbooks]. Birkh\"auser Verlag, Basel, 1994.

\bibitem[Kar99]{karshon}
Y.~Karshon.
\newblock Periodic {H}amiltonian flows on four-dimensional manifolds.
\newblock {\em Mem. Amer. Math. Soc.}, 141(672):viii+71, 1999.

\bibitem[KL15]{karshon-lerman}
Y.~Karshon and E.~Lerman.
\newblock Non-compact symplectic toric manifolds.
\newblock {\em SIGMA Symmetry Integrability Geom. Methods Appl.}, 11:Paper 055, 37, 2015.

\bibitem[KM21a]{KudMar-invariants}
E.~Kudryavtseva and N.~Martynchuk.
\newblock ${C}^\infty$ symplectic invariants of parabolic orbits and flaps in integrable {H}amiltonian systems.
\newblock Preprint, \url{https://arxiv.org/abs/2110.13758}, 2021.

\bibitem[KM21b]{KudMar-circle}
E.~Kudryavtseva and N.~Martynchuk.
\newblock Existence of a smooth hamiltonian circle action near parabolic orbits.
\newblock {\em Regul. Chaotic Dyn.}, 26(6):732--741, 2021.

\bibitem[KPP18a]{KPP2015}
D.~M. Kane, J.~Palmer, and \'A. Pelayo.
\newblock Classifying toric and semitoric fans by lifting equations from {${\rm SL}_2(\mathbb{Z})$}.
\newblock {\em SIGMA Symmetry Integrability Geom. Methods Appl.}, 14:Paper No. 016, 43, 2018.

\bibitem[KPP18b]{KPP2018}
D.M. Kane, J.~Palmer, and \'A. Pelayo.
\newblock Minimal models of compact symplectic semitoric manifolds.
\newblock {\em J. Geom. Phys.}, 125:49--74, 2018.

\bibitem[KT01]{KaTo-centered}
Y.~Karshon and S.~Tolman.
\newblock Centered complexity one {H}amiltonian torus actions.
\newblock {\em Trans. Amer. Math. Soc.}, 353(12):4831--4861, 2001.

\bibitem[KT03]{KT-tall-invariants}
Y.~Karshon and S.~Tolman.
\newblock Complete invariants for {H}amiltonian torus actions with two dimensional quotients.
\newblock {\em J. Symplectic Geom.}, 2(1):25--82, 2003.

\bibitem[KT14]{KT-tall-classification}
Y.~Karshon and S.~Tolman.
\newblock Classification of {H}amiltonian torus actions with two-dimensional quotients.
\newblock {\em Geom. Topol.}, 18(2):669--716, 2014.

\bibitem[KT20]{KT-top-of-quotients}
Y.~Karshon and S.~Tolman.
\newblock Topology of complexity one quotients.
\newblock {\em Pacific J. Math.}, 308(2):333--346, 2020.

\bibitem[LFP19]{LFP}
Y.~Le~Floch and \'A. Pelayo.
\newblock Symplectic geometry and spectral properties of classical and quantum coupled angular momenta.
\newblock {\em J. Nonlinear Sci.}, 29(2):655--708, 2019.

\bibitem[LFP23]{LFPal-SF2}
Y.~Le~Floch and J.~Palmer.
\newblock Families of four-dimensional integrable systems with {$S^1$}-symmetries.
\newblock Preprint, \url{https://arxiv.org/abs/2307.10670}, 2023.

\bibitem[LFP24]{LFPal}
Y.~Le~Floch and J.~Palmer.
\newblock Semitoric families.
\newblock {\em Mem. Amer. Math. Soc.}, 302(1514):v+91, 2024.

\bibitem[LFPVN16]{LFPeVN2016}
Y.~Le~Floch, {\'A}.~Pelayo, and S.~V{\~u}~Ng{\d{o}}c.
\newblock Inverse spectral theory for semiclassical {J}aynes-{C}ummings systems.
\newblock {\em Math. Ann.}, 364(3-4):1393--1413, 2016.

\bibitem[LFVuN21]{LFVN}
Y.~Le~Floch and S.~V\~u Ng\d{o}c.
\newblock The inverse spectral problem for quantum semitoric systems.
\newblock Preprint, \url{https://arxiv.org/abs/2104.06704}, 115 pages, 2021.

\bibitem[Liu23]{Liu-HZ-bounds}
Y.~Liu.
\newblock A lower bound of the {H}ofer-{Z}ehnder capacity via {D}elzant polytopes.
\newblock Preprint, \url{https://arxiv.org/abs/2312.09526}, 2023.

\bibitem[LMTW98]{LMTW}
E.~Lerman, E.~Meinrenken, S.~Tolman, and C.~Woodward.
\newblock Nonabelian convexity by symplectic cuts.
\newblock {\em Topology}, 37(2):245--259, 1998.

\bibitem[LS10]{LeuSym}
N.C. Leung and M.~Symington.
\newblock Almost toric symplectic four-manifolds.
\newblock {\em J. Symplectic Geom.}, 8(2):143--187, 2010.

\bibitem[Mag24]{macgill}
N.~Magill.
\newblock Unobstructed embeddings in {H}irzebruch surfaces.
\newblock {\em J. Symplectic Geom.}, 22(1):109--152, 2024.

\bibitem[Mat96]{Mat96}
V.~S. Matveev.
\newblock Integrable {H}amiltonian systems with two degrees of freedom. {T}opological structure of saturated neighborhoods of points of focus-focus and saddle-saddle types.
\newblock {\em Mat. Sb.}, 187(4):29--58, 1996.

\bibitem[MBE21]{MaBrEf-monodromy}
N.~Martynchuk, H.~W. Broer, and K.~Efstathiou.
\newblock Recent advances in the monodromy theory of integrable {H}amiltonian systems.
\newblock {\em Indag. Math. (N.S.)}, 32(1):193--223, 2021.

\bibitem[Mey73]{meyer}
K.~Meyer.
\newblock Symmetries and integrals in mechanics.
\newblock In {\em Dynamical systems ({P}roc. {S}ympos., {U}niv. {B}ahia, {S}alvador, 1971)}, pages 259--272, 1973.

\bibitem[Min47]{mineur}
H.~Mineur.
\newblock Sur les syst\`emes m\'ecaniques dont les int\'egrales premi\`eres sont d\'efinies par des \'equations implicites.
\newblock {\em C. R. Acad. Sci. Paris}, 224:26--27, 1947.

\bibitem[MMW24]{MMW-staircasepatterns}
N.~Magill, D.~McDuff, and M.~Weiler.
\newblock Staircase patterns in {H}irzebruch surfaces.
\newblock {\em Comment. Math. Helv.}, 99(3):437--555, 2024.

\bibitem[Mol23]{Mol}
M.~Mol.
\newblock On the classification of multiplicity-free hamiltonian actions by regular proper symplectic groupoids.
\newblock Preprint, \url{https://arxiv.org/abs/2401.00570}, 2023.

\bibitem[MS12]{MS12}
D.~McDuff and F.~Schlenk.
\newblock The embedding capacity of 4-dimensional symplectic ellipsoids.
\newblock {\em Ann. of Math. (2)}, 175(3):1191--1282, 2012.

\bibitem[MS17]{McDSal}
D.~McDuff and D.~Salamon.
\newblock {\em Introduction to symplectic topology}.
\newblock Oxford Graduate Texts in Mathematics. Oxford University Press, Oxford, third edition, 2017.

\bibitem[MW74]{marsden-weinstein}
J.~Marsden and A.~Weinstein.
\newblock Reduction of symplectic manifolds with symmetry.
\newblock {\em Rep. Mathematical Phys.}, 5(1):121--130, 1974.

\bibitem[MZ04]{miranda-zung}
E.~Miranda and N.~T. Zung.
\newblock Equivariant normal form for nondegenerate singular orbits of integrable {H}amiltonian systems.
\newblock {\em Ann. Sci. \'Ecole Norm. Sup. (4)}, 37(6):819--839, 2004.

\bibitem[Pal17]{PaSTMetric2015}
J.~Palmer.
\newblock Moduli spaces of semitoric systems.
\newblock {\em J. Geom. Phys.}, 115:191--217, 2017.

\bibitem[Pel06]{Pe2006}
{\'A}~Pelayo.
\newblock Toric symplectic ball packing.
\newblock {\em Topology and its Appl.}, 157:3633--3644, 2006.

\bibitem[Pel07]{Pe2007}
{\'A}.~Pelayo.
\newblock Topology of spaces of equivariant symplectic embeddings.
\newblock {\em Proc. Amer. Math. Soc.}, 135(1):277--288, 2007.

\bibitem[Pel21]{Pelayo-inverse}
\'A. Pelayo.
\newblock Symplectic invariants of semitoric systems and the inverse problem for quantum systems.
\newblock {\em Indag. Math. (N.S.)}, 32(1):246--274, 2021.

\bibitem[Pel23]{Pelayo-survey23}
\'{A}. Pelayo.
\newblock Symplectic and inverse spectral geometry of integrable systems: A glimpse and open problems.
\newblock {\em Topology and its Applications}, 339:108577, 2023.

\bibitem[PPRS14]{PPRStoric}
{\'A}.~Pelayo, A.~R. Pires, T.~S. Ratiu, and S.~Sabatini.
\newblock Moduli spaces of toric manifolds.
\newblock {\em Geom. Dedicata}, 169:323--341, 2014.

\bibitem[PPT24]{PPT-nonsimple}
J.~Palmer, \'A. Pelayo, and X.~Tang.
\newblock Semitoric systems of non-simple type.
\newblock {\em Rev. R. Acad. Cienc. Exactas F\'is. Nat. Ser. A Mat. RACSAM}, 118(4):Paper No. 161, 32, 2024.

\bibitem[PRVuN17]{PeRaVN2015}
{\'A}.~Pelayo, T.~Ratiu, and S.~V\~{u}~Ng\d{o}c.
\newblock The affine invariant of proper semitoric integrable systems.
\newblock {\em Nonlinearity}, 30(11):3993--4028, 2017.

\bibitem[PS08]{PeSc2008}
{\'A}~Pelayo and B.~Schmidt.
\newblock Maximal ball packings of symplectic-toric manifolds.
\newblock {\em Intern. Math. Res. Not.}, 2008.
\newblock 24p, ID rnm139.

\bibitem[PS23]{PelayoSantos23}
{\'A}.~Pelayo and F.~Santos.
\newblock Moduli spaces of {D}elzant polytopes and symplectic toric manifolds.
\newblock Preprint, \url{https://arxiv.org/abs/2303.02369}, 2023.

\bibitem[PT20]{PasTon}
J.~Pascaleff and D.~Tonkonog.
\newblock The wall-crossing formula and {L}agrangian mutations.
\newblock {\em Adv. Math.}, 361:106850, 67, 2020.

\bibitem[PT24]{PT}
{\'A}.~Pelayo and X.~Tang.
\newblock Vu {N}goc's conjecture on focus-focus singular fibers with multiple pinched points.
\newblock {\em J. Fixed Point Theory Appl.}, 26(1):Paper No. 6, 34, 2024.

\bibitem[PVN09]{PeVN2009}
{\'A}.~Pelayo and S.~V{\~u}~Ng\d{o}c.
\newblock Semitoric integrable systems on symplectic 4-manifolds.
\newblock {\em Invent. Math.}, 177:571--597, 2009.

\bibitem[PVN11a]{PeVN2011}
{\'A}.~Pelayo and S.~V{\~u}~Ng\d{o}c.
\newblock Constructing integrable systems of semitoric type.
\newblock {\em Acta Math.}, 206:93--125, 2011.

\bibitem[PVN11b]{PeVNsymplthy2011}
{\'A}.~Pelayo and S.~V{\~u}~Ng\d{o}c.
\newblock Symplectic theory of completely integrable {H}amiltonian systems.
\newblock {\em Bull. Amer. Math. Soc.}, 48:409--455, 2011.

\bibitem[PVN12]{PeVNfirststeps}
{\'A}.~Pelayo and S.~V{\~u}~Ng{\d{o}}c.
\newblock First steps in symplectic and spectral theory of integrable systems.
\newblock {\em Discrete Contin. Dyn. Syst.}, 32(10):3325--3377, 2012.

\bibitem[SS22]{SabSepe-top-properties}
S.~Sabatini and D.~Sepe.
\newblock On topological properties of positive complexity one spaces.
\newblock {\em Transform. Groups}, 27(2):723--735, 2022.

\bibitem[ST24]{SepeTolman}
D.~Sepe and S.~Tolman.
\newblock Connectedness of level sets for non-degenerate integrable systems that extend complexity one torus actions.
\newblock Preprint, \url{https://arxiv.org/abs/2402.05814}, 2024.

\bibitem[SVuN18]{VNSepe}
D.~Sepe and S.~V\~u Ng\d{o}c.
\newblock Integrable systems, symmetries, and quantization.
\newblock {\em Lett. Math. Phys.}, 108(3):499--571, 2018.

\bibitem[Sym03]{Sy2003}
M.~Symington.
\newblock Four dimensions from two in symplectic topology.
\newblock In {\em Topology and geometry of manifolds ({A}thens, {GA}, 2001)}, volume~71 of {\em Proc. Sympos. Pure Math.}, pages 153--208. Amer. Math. Soc., Providence, RI, 2003.

\bibitem[SZ99]{SaZh1999}
D.~A. Sadovski{\'\i} and B.~I. Z\^hilinski{\'\i}.
\newblock Monodromy, diabolic points, and angular momentum coupling.
\newblock {\em Phys. Lett. A}, 256(4):235--244, 1999.

\bibitem[Tan24]{tang-almosttoric}
X.~Tang.
\newblock A note on the symplectic classification of almost-toric systems.
\newblock Preprint, \url{https://arxiv.org/abs/2410.08175}, 2024.

\bibitem[Vey78]{Vey}
J.~Vey.
\newblock Sur certains syst\`emes dynamiques s\'{e}parables.
\newblock {\em Amer. J. Math.}, 100(3):591--614, 1978.

\bibitem[Via14]{Via}
R.~Vianna.
\newblock On exotic {L}agrangian tori in {$\mathbb{CP}^2$}.
\newblock {\em Geom. Topol.}, 18(4):2419--2476, 2014.

\bibitem[Via16]{Viainf}
R.~Vianna.
\newblock Infinitely many exotic monotone {L}agrangian tori in {$\mathbb{CP}^2$}.
\newblock {\em J. Topol.}, 9(2):535--551, 2016.

\bibitem[Via17]{Viadel}
R.~Vianna.
\newblock Infinitely many monotone {L}agrangian tori in del {P}ezzo surfaces.
\newblock {\em Selecta Math. (N.S.)}, 23(3):1955--1996, 2017.

\bibitem[VuN03]{VN2003}
S.~V\~{u}~Ng\d{o}c.
\newblock On semi-global invariants for focus-focus singularities.
\newblock {\em Topology}, 42(2):365--380, 2003.

\bibitem[VuN07]{VN2007}
S.~V\~{u}~Ng\d{o}c.
\newblock Moment polytopes for symplectic manifolds with monodromy.
\newblock {\em Adv. Math.}, 208(2):909--934, 2007.

\bibitem[VuNW13]{VNWac}
S.~V\~{u}~Ng\d{o}c and C.~Wacheux.
\newblock Smooth normal forms for integrable {H}amiltonian systems near a focus-focus singularity.
\newblock {\em Acta Math. Vietnam.}, 38(1):107--122, 2013.

\bibitem[Wac13]{wac-thesis}
C.~Wacheux.
\newblock {\em Syst\`{e}mes int\'{e}grables semi-toriques et polytopes moment}.
\newblock PhD thesis, Universit\'{e} de Rennes 1, 2013.

\bibitem[Wil36]{Williamson}
J.~Williamson.
\newblock On the {A}lgebraic {P}roblem {C}oncerning the {N}ormal {F}orms of {L}inear {D}ynamical {S}ystems.
\newblock {\em Amer. J. Math.}, 58(1):141--163, 1936.

\bibitem[Zou92]{Zou}
M.~Zou.
\newblock Monodromy in two degrees of freedom integrable systems.
\newblock {\em J. Geom. Phys.}, 10(1):37--45, 1992.

\bibitem[Zun96]{zung96}
N.T. Zung.
\newblock Symplectic topology of integrable {H}amiltonian systems. {I}. {A}rnold-{L}iouville with singularities.
\newblock {\em Compositio Math.}, 101(2):179--215, 1996.

\bibitem[Zun97]{Zung97}
N.T. Zung.
\newblock A note on focus-focus singularities.
\newblock {\em Differential Geom. Appl.}, 7(2):123--130, 1997.

\bibitem[Zun03]{Zung03}
N.T. Zung.
\newblock Symplectic topology of integrable {H}amiltonian systems. {II}. {T}opological classification.
\newblock {\em Compositio Math.}, 138(2):125--156, 2003.

\end{thebibliography}

 {\small
   \noindent
   \\
   Joseph Palmer\\
   Department of Mathematics,\\
   Seeley Mudd Building,\\
   Amherst College,\\ Amherst, MA, USA, 01002.\\ 
   {\em E-mail:} \texttt{jpalmer@amherst.edu}
 }

\end{document}